\documentclass[11pt]{article}
\usepackage{amsfonts}
\usepackage{amsmath}
\usepackage{amsthm}
\usepackage{amscd}
\usepackage{amsmath}
\usepackage{amssymb}
\usepackage{amscd}
\usepackage{amsfonts}
\usepackage{amsbsy}
\usepackage{graphicx}
\usepackage{booktabs}
\usepackage{makecell}
\usepackage{enumitem}
\usepackage{tikz}
\usepackage{booktabs}
\usepackage{multirow}
\usepackage{diagbox}
\usepackage{appendix}
\usepackage{mathrsfs}
\usepackage{caption}
\usepackage{float}
\usepackage{subcaption}
\usepackage{hyperref}
\hypersetup{colorlinks=true,citecolor=blue,%
        linkcolor=blue,urlcolor=blue,bookmarksnumbered=true,
        bookmarksopen=true,bookmarksopenlevel=1,breaklinks=true}
\newtheorem{theorem}{Theorem}[section]
\newtheorem{proposition}[theorem]{Proposition}
\newtheorem{lemma}[theorem]{Lemma}
\newtheorem{corollary}[theorem]{Corollary}
\newtheorem{remark}[theorem]{Remark}

\theoremstyle{definition}
\newtheorem{definition}[theorem]{Definition}

\numberwithin{equation}{section}
\numberwithin{figure}{section}

\begin{document}
\title
{Stochastic bifurcation of a three-dimensional stochastic Kolmogorov system}

\author{Dongmei Xiao\footnote{School of Mathematical Sciences, CMA-Shanghai, Shanghai Jiao Tong University, Shanghai 200240, China, e-mail: \href{mailto:xiaodm@sjtu.edu.cn}{xiaodm@sjtu.edu.cn}}         \and
  Deng Zhang\footnote{School of Mathematical Sciences, CMA-Shanghai, Shanghai Jiao Tong University, 200240, Shanghai, China, e-mail: \href{mailto:dzhang@sjtu.edu.cn}{dzhang@sjtu.edu.cn}}
       \and Chenwan Zhou\footnote{School of Mathematical Sciences, CMA-Shanghai, Shanghai Jiao Tong University, 200240, Shanghai, China; Department of Mathematical Sciences, Durham University, DH1 3LE, UK, e-mail: \href{mailto:daydayupzcw@sjtu.edu.cn}{daydayupzcw@sjtu.edu.cn}}}
 \date{}

\maketitle
\begin{abstract}
In this paper we systematically investigate the stochastic bifurcations 
of both ergodic stationary measures and global dynamics for stochastic Kolmogorov differential systems, 
which relate closely to the change of the sign of Lyapunov exponents. 
It is derived that there exists a threshold $\sigma_0$ such that,
if the noise intensity $\sigma \geq\sigma_0$,
the noise destroys all bifurcations of the deterministic system
and the corresponding stochastic Kolmogorov system is uniquely ergodic.
On the other hand, 
when the noise intensity $\sigma<\sigma_0$,
the stochastic system undergoes bifurcations 
from the unique ergodic stationary measure 
to three different types of ergodic stationary measures: 
(I) finitely many ergodic measures supported on rays, 
(II) infinitely many ergodic measures supported on rays, 
(III) infinitely many ergodic measures supported on invariant cones. 
Correspondingly, 
the global dynamics undergo similar bifurcation phenomena, 
which even displays infinitely many Crauel random periodic solutions in the sense of \cite{ELR21}.   
Furthermore, we prove that as $\sigma$ tends to zero,
the ergodic stationary measures converge to either Dirac measures supported on equilibria, 
or to Haar measures supported on non-trivial 
deterministic periodic orbits. 

\medskip
\noindent
{\bf MSC2020 subject classifications:}  60H10,  37G35, 37H15,  34F05. 

\medskip
\noindent
{\bf Keywords:} Stochastic Kolmogorov system, Lyapunov exponent, stochastic bifurcation, ergodicity, Crauel random periodic solution.
\end{abstract}
\tableofcontents

\section{Introduction and main results}
\begin{sloppypar}

\subsection{Background}

Kolmogorov system is the classical model in population dynamics proposed by Kolmogorov \cite{kl}, which describes the growth rate of populations in a community of $n$ interacting species and is defined by the following system of ordinary differential equations
\begin{equation}\label{sys1.1}
\frac{dx_i(t)}{dt}=x_i(t)P_i(x_1(t),...,x_n(t)), \ \ i=1,\cdots, n,
\end{equation}
where $x_i(t)$ represents the population  number
(density) of the $i$-th species at time $t$ and
$P_i \in C(\mathbb{R}^n)$ is its per capita growth rate. This model has played an important role in describing  the behavior of the interactions of species in population ecology, 
and has been widely used in many areas, such as game dynamics, network dynamics, turbulence dynamics, see \cite{busse1, busse3,Hofbauer, Knebel1, Knebel2} and references therein. 
As pointed out by Smale \cite{smale},  
the dynamic behavior of any given $m$-dimensional dynamical system can be realized by Kolmogorov system \eqref{sys1.1} with $n>m$ under some competitive conditions.
Thus, the rich dynamics of the Kolmogorov system \eqref{sys1.1}
has attracted significant interest in the literature,
 see, e.g., \cite{hirschSmith,llibre, Smith_2017, Zeeman} and references therein.

In this paper we consider the 3D cubic Kolmogorov system driven by linear multiplicative Wiener noise
\begin{equation}\label{sys5}
\begin{cases}
dx_1=x_1(\alpha-\alpha x_1^2-(2\alpha+d_1)x_2^2+d_2x_3^2)dt+\sigma x_1dW_t ,\\
dx_2=x_2(\alpha+d_1x_1^2-\alpha x_2^2-(2\alpha+d_3)x_3^2)dt+\sigma x_2dW_t,\\
dx_3=x_3(\alpha-(2\alpha+d_2)x_1^2+d_3x_2^2-\alpha x_3^2)dt+\sigma x_3dW_t,\\
\end{cases}
\end{equation}
where $\sigma>0$ represents the strength of noise,
$(W_t)$ is the Wiener process,
and the drift term is parameterized by $\alpha>0$ and
$d_i\in \mathbb{R}$, here $i=1,2,3$.
In particular,
in the absence of noise,
system \eqref{sys5}
reduces to the deterministic cubic Kolmogorov system
\begin{equation}\label{sys6}
\begin{cases}
\frac{dx_1}{dt}=x_1(\alpha-\alpha x_1^2-(2\alpha+d_1)x_2^2+d_2x_3^2),\\
\frac{dx_2}{dt}=x_2(\alpha+d_1x_1^2-\alpha x_2^2-(2\alpha+d_3)x_3^2),\\
\frac{dx_3}{dt}=x_3(\alpha-(2\alpha+d_2)x_1^2+d_3x_2^2-\alpha x_3^2). \\
\end{cases}
\end{equation}

The main interest of the present work is to 
characterize the stochastic bifurcation 
of both ergodic stationary measures and 
global dynamics for the stochastic Kolmogorov system \eqref{sys5}.

{Bifurcation} of dynamical system 
usually describes {sudden qualitative}
or {topological changes} of the long-term dynamical behavior of dynamical systems,
when some parameters of dynamical systems 
vary continuously in small neighborhoods of a value.
This particular parameter value is called {\it bifurcation value (or bifurcation point) },
and the corresponding changing parameter is called {\it bifurcation parameter}.
For random dynamical systems, stochastic bifurcation is often considered
from the perspective of either steady-state distribution or ergodic invariant measure.
The phenomenological bifurcation 
concerns sudden changes of stationary distributions 
as bifurcation parameters change in a small neighborhood of a bifurcation value, 
while the dynamical bifurcation describes changes 
of ergodic invariant measures. 

Stochastic bifurcation phenomena have attracted considerable interests in the literature 
and are extensively studied for dynamical models 
driven by additive noise. 
For instance,   
pitchfork bifurcations with additive noise 
were studied in \cite{BEN23,FL2}. 
In \cite{Doan1}, 
three dynamical phases are identified 
which include a random strange attractor 
with positive Lyapunov exponent.  
See also \cite{CE23} for the positivity of Lyapunov exponent 
for normal formal of a Hopf bifurcation perturbed by additive noise.

Positivity of Lyapunov exponents usually  
relates to chaotic phenomena of dynamics,  
see, e.g., the nice explanations by Young \cite{Young1,Young13} 
and Bedrossian, Blumenthal and Punshon-Smith \cite{BBP23}. 
One typical model is the 2D 
Navier-Stokes equation (NSE) driven by additive noise. 
Ergodicity for this stochastic fluid model 
is well-known, 
see, e.g.,  
\cite{BKL02, FM95,HM06,KNS20, KS00, EMS01} 
and references therein. 
In \cite{BBS22}, 
Bedrossian, Blumenthal and Punshon-Smith proved the  
positivity of the top Lyapunov exponent 
for  
the Lagrangian flow generated by 2D stochastic NSE with non-degenerate Gaussian noise. 
More general Euler-like systems including stochastic Lorenz 96 system have been studied in \cite{BBS22.1}. 
For the Lagrangian flow of 2D stochastic NSE with degenerate bounded noise, 
the positivity of the top Lyapunov exponent 
has  been recently proved by Nersesyan 
and the last two named authors \cite{NZZ24}. 

Compared to the extensive results 
in the case of additive noise, 
there are not many results on stochastic bifurcations 
in the multiplicative noise case, 
which is another typical noise for 
stochastic models. 
For instance, 
a stochastic Hopf bifurcation was studied in \cite{PB1} 
for SDE with multiplicative noise. 
For 
pull-back trajectories and ergodic stationary measures
for stochastic Lotka-Volterra systems with multiplicative noise, 
we refer to \cite{jiang1}. 
Recently,  
Engel, Lamb and Rasmussen \cite{ELR19} 
established the existence of a bifurcation for a stochastically driven limit cycle, 
indicated by the change of the sign 
of top Lyapunov exponents, 
which relates to an open problem in \cite{Young2, Young3, Young1}.

\medskip
In this paper 
we give a complete characterization of the  bifurcation phenomena for the 3D stochastic Kolmogorov system \eqref{sys5}, 
depending upon the strength of the noise and the parameters $\alpha$ and $d_i$, $i=1,2,3$. 
The stochastic Kolmogorov system undergoes bifurcations  from a unique ergodic stationary measure 
to three different types of ergodic stationary measures: 
(I) finitely many ergodic measures supported on rays, 
(II) infinitely many ergodic measures supported on rays, 
(III) infinitely many ergodic measures supported on invariant cones.  
Interestingly, 
the bifurcation phenomena relate closely to 
the change of the sign of Lyapunov exponents.

Furthermore, we systematically investigate the classification of stochastic dynamics through the perspective of pull-back $\Omega$-limit sets.  
It is shown that 
the bifurcation phenomena exhibit 
for four different types of pull-back $\Omega$-limit sets, 
which again relate to different signs of Lyapunov exponents: 
(I') the unique random equilibrium $O$, 
(II') finitely many random equilibria, 
(III') infinitely many  random equilibria,  
(IV') infinitely many Crauel random periodic solutions.

In addition, we prove that, via the vanishing noise limit, 
the ergodic stationary measures of stochastic Kolmogorov system \eqref{sys5} converges to 
either Dirac measures supported on equilibria 
or Haar measures supported on 
periodic orbits 
to the deterministic system  \eqref{sys6}.

\subsection{Main results}  

Let us first mention that the 3-D cubic Kolmogorov system \eqref{sys1.1} with an invariant sphere 
has been recently studied in  \cite{xiao}. 
It is shown that system \eqref{sys6} has the following invariant sphere in $\mathbb{R}^3$
\begin{equation}\label{shpere}
 \mathbb{S}^2=\{(x_1,x_2,x_3): \ x_1^2+x_2^2+x_3^2=1\}\subset \mathbb{R}^3,
\end{equation}
and $\mathbb{S}^2$  is an isolated invariant set of system \eqref{sys6} if and only if  $\alpha\not=0$.
Without loss of generality, we consider the case $\alpha>0$ for systems \eqref{sys5} and \eqref{sys6}. 

Moreover, system \eqref{sys6} is invariant under the coordinate transforms
$(x_1, x_2, x_3)$ $\rightarrow (-x_1, x_2, x_3)$, $(x_1, x_2, x_3)\rightarrow(x_1, -x_2, x_3)$
and $(x_1, x_2, x_3)\rightarrow(x_1, x_2, -x_3)$. Moreover,
the planes $x_i=0$, $i=1,2,3$, are invariant
and the flow generated by \eqref{sys6} is symmetric with respect to these planes.
Hence, we focus on system \eqref{sys6} in $\mathbb{R}^3_+$ in the sequel.

\subsubsection{Stochastic bifurcation of ergodic stationary measures}

For any ergodic stationary measure 
$\mu \in \mathcal{P}(\mathbb{R}^3_+)$, 
let $\lambda_i(\mu)$, $i=1,2,3$, denote the corresponding Lyapunov exponents.  

We also need some notations 
for the geometrics related to system \eqref{sys6}. 
For any $y\in \mathbb{R}^3_+$, 
let $\mathcal{L}(y):=\{\lambda y: \lambda>0\}$
denote the ray passing through the point $y$ 
and $\overline{\mathcal{L}(y)}$ 
its closure in $\mathbb{R}^3_+$.
Moreover, 
for any $h\in ({h}^*, \infty)$, where
    \begin{equation}  \label{h*-def}
{h}^*:= \prod\limits_{i=1}^3 \left(\frac{\alpha+d_{4-i}}{3\alpha+d_1+d_2+d_3}\right)^{-\frac{\alpha+d_{4-i}}{3\alpha+d_1+d_2+d_3}}, 
\end{equation}
let $\Gamma(h)$ denote the closed orbit 
to system \eqref{sys6} 
and 
$\Lambda(h):=\{\lambda y: y\in \Gamma(h), \lambda\ge 0\}$  
the corresponding invariant cone.

The first main result of this paper 
is formulated in Theorem \ref{bifurcations} below, which describes the bifurcation 
of ergodic stationary measures 
depending upon the strength of the noise 
and the parameters in system \eqref{sys6}. 

\begin{theorem} \label{bifurcations} (Bifurcation of ergodic stationary measures) There exists a bifurcation parameter $\sigma^2$ and a bifurcation point $2\alpha$, such that the stochastic Kolmogorov system \eqref{sys5} undergoes a bifurcation of ergodic stationary measures. More precisely,
\begin{itemize}
  \item [(i)] When $\sigma^2>2\alpha$, system \eqref{sys5} has a
        unique ergodic stationary measure $\delta_O$, which corresponds to the unique globally attracting random equilibrium $O$, 
        and $\lambda_i(\delta_O) < 0$, $i=1,2,3$.         
  \item [(ii)] When $\sigma^2=2\alpha$, system \eqref{sys5} has a unique ergodic stationary measure $\delta_O$, which corresponds to the unique globally attracting random equilibrium $O$, 
  but with $\lambda_i(\delta_O) = 0$, $i=1,2,3$.
  
  \item [(iii)] When $\sigma^2<2\alpha$, system \eqref{sys5} has other ergodic stationary measures except $\delta_O$. 
  The random equilibrium $O$ is however unstable and 
  $\lambda_i(\delta_O) > 0$, $i=1,2,3$. 
  
  Furthermore, the other ergodic stationary measures exhibit finer bifurcation phenomena
      depending on the sign
      of the parameters $\alpha+d_i$ (
      $i=1,2,3$), which are related to the sign of Lyapunov exponents of random non-zero equilibria $u_g(\omega)\textbf{e}_i$ 
      {(see Subsection \ref{logistic-type} below), $i=1,2,3$: }
      
      \begin{itemize}
        \item [(iii.1)] 
        If $\prod_{i=1}^3(\alpha+d_i)= 0$, then there exist infinitely many ergodic stationary measures, each of which is supported on a ray $\overline{\mathcal{L}(Q)}$ 
        for some equilibrium $Q$ 
        of the deterministic system \eqref{sys6}.
        
        \item [(iii.2)] 
        If $\alpha+d_i$ are all positive (or all negative) for $i=1,2,3$, 
        then there exist 5 ergodic stationary measures 
        supported on $O$ 
        or rays $\overline{\mathcal{L}(Q)}$ 
        corresponding to 4 equilibria 
        $Q (\not= O)$ of \eqref{sys6}, and infinitely many ergodic stationary measures supported on invariant cones 
        $\Lambda(h)$,  
        where $h>h^*$ with $h^*$ given by \eqref{h*-def}.  
        
        \item [(iii.3)] 
        If  $\prod_{i=1}^3(\alpha+d_i) \neq 0$, and $(\alpha+d_i)(\alpha+d_j)<0$ for some $i\neq j$, $i,j\in \{1,2,3\}$,
  then there are only 4 ergodic stationary measures, 
  supported on $O$ 
  or  rays $\overline{\mathcal{L}(Q)}$ 
  corresponding to 3 equilibria $Q$ of \eqref{sys6}. 
      \end{itemize}
\end{itemize}
\end{theorem} 

\begin{remark} 
$(i)$ 
We note that 
the sign of Lyapunov exponents $\lambda_i(\delta_O)$, 
$i=1,2,3$, 
changes in the birfurcation cases $(i)-(iii)$. 
That is, 
the Lyapunov exponents of $\delta_O$ 
are all negative when $\sigma^2>2\alpha$, 
all zero when $\sigma^2=2\alpha$, 
while all positive when $\sigma^2<2\alpha$. 

$(ii) $
For the sign of Lyapunov exponents 
in the case of Theorem \ref{bifurcations} $(iii)$, 
let us take the random equilibrium $u_g(\omega)e_1$ as an example.  
One has that 
$\lambda_1(u_g(\omega)e_1) <0$ 
in all cases of Theorem \ref{bifurcations} $(iii.1)-(iii.3)$. However, 
for the other two Lyapunov exponents 
$\lambda_i(u_g(\omega)e_1)$, $i=2,3$,   
in the case of 
 Theorem \ref{bifurcations} $(iii.1)$ 
it may happen that both are zero, 
in the case of 
 Theorem \ref{bifurcations} $(iii.2)$ 
one is negative and the other is positive, 
while in the case of 
 Theorem \ref{bifurcations} $(iii.3)$ it may happen that
 both are positive or both are negative.  

Let us mention that 
in the zero Lyapunov case in Theorem  \ref{bifurcations} $(iii.1)$, 
there display further bifurcations of global dynamics, 
which will be given in detail in Theorem \ref{main-thm-2} 
$(ii.1_a)$-$(ii.1_c)$ below.

$(iii)$ 
When system \eqref{sys5} has only finite ergodic stationary measures, 
these measures are all
{\it hyperbolic} except the case where $\sigma^2=2\alpha$. 
Here,  
  hyperbolicity means all Lyapunov exponents are non-zero.
  
The change of hyperbolicity 
indeed leads to stochastic bifurcations of ergodic stationary measures.  
For instance, 
the hyperbolicity changes in the three cases of 
Theorem \ref{bifurcations} $(i)$-$(iii)$. 
While  
in the case of Theorem \ref{bifurcations} $(iii.3)$, 
the 4 ergodic stationary measures
are all hyperbolic, 
and thus they do not display further 
bifurcations. 

$(iv)$  
In Subsection \ref{Stochastic bifurcation}, 
we also prove that system \eqref{sys5} undergoes a bifurcation of the densities of ergodic stationary measures
generated by non-zero equilibria. 
More precisely, 
the density 
is an unimodal function 
when $\sigma^2<\alpha$, 
but is decreasing 
when $\alpha\le\sigma^2<2\alpha$ 
(see Theorem \ref{3dp} below). 

$(v)$
In the case of Theorem \ref{bifurcations} $(iii)$, 
the uniqueness of ergodic stationary measures
is derived on every invariant cone 
$\Lambda(h)\setminus\{O\}$, 
by utilizing the strong Feller and irreducibility  of the Markov semigroup 
associated to \eqref{sys5}.

\end{remark}

\subsubsection{Classification of global dynamics via pull-back $\Omega$-limit sets}

 Based on Theorem \ref{bifurcations}, we further 
 derive the complete classification of global stochastic dynamics via pull-back $\Omega$-limit sets. 

Theorem \ref{main-thm-2} below reveals 
the transition 
    from the unique random equilibrium 
    to infinitely many random equilibria, or 
    even to infinitely many Crauel random periodic solutions 
    (see Definition \ref{Crauel-def} 
 in the Appendix), 
related to the change of the sign of Lyapunov exponents 
in Thereom \ref{bifurcations}.

Let $\Omega_x$ 
denote the pull-back $\Omega$-limit set of the trajectories 
 of system \eqref{sys5} 
 starting from $x$. 

\begin{theorem}\label{main-thm-2}(Classification of global  dynamics via pull-back $\Omega$-limit sets) 

For $\mathbb{P}$-a.e. $\omega\in \Omega$ 
and any $x\in
      \mathbb{R}^3_+$, 
the following holds: 

\begin{itemize}
  \item [(i)] When $\sigma^2\ge 2\alpha$, the origin $O$ is the unique random equilibrium, 
  and 
      $\Omega_x(\omega) = \{O\}$. 
      
  \item [(ii)] When $\sigma^2< 2\alpha$, system \eqref{sys5} has other random equilibria except the origin $O$. 
  More precisely, 
  we have 
  
      \begin{itemize}
      \item [(ii.1)] 
      In the case  of Theorem \ref{bifurcations} $(iii.1)$, 
      $\Omega_x(\omega)$ 
       belongs to  infinitely many random equilibria 
       generated by deterministic equilibria. 
       Moreover, the following geometrical properties hold: 
      \begin{itemize}
        \item [(ii.1$_a$)] if there exists a unique $i\in\{1,2,3\}$ such that $\alpha+d_i=0$, then there are infinitely many random equilibria forming one curve for each noise realization;
        \item [(ii.$1_b$)] if there are two $i,j\in \{1,2,3\}$, $i\ne j$ such that $\alpha+d_i=\alpha+d_j=0$, then there are infinitely many random equilibria forming two curves for each noise realization;
        \item [(ii.$1_c$)] if $\alpha+d_i=0$ for all $i=1,2,3$, then there are infinitely many random equilibria forming a surface on $\mathbb{R}^3_+$ for each noise realization.
      \end{itemize}
      
       \item [(ii.2)] In the case of Theorem \ref{bifurcations} $(iii.2)$, 
       there are 5 random equilibria and infinitely many Crauel random periodic solutions. 
       Moreover, 
      $\Omega_x(\omega)$ is either one of the 5 random equilibria or a random cycle corresponding to a Crauel random periodic solution.
  \item [(ii.3)] In the case of Theorem \ref{bifurcations} $(iii.3) $, 
  $\Omega_x(\omega)$ belongs to  4 distinct random equilibria, 
  whose convex combinations contain all random equilibria. 
\end{itemize}
\end{itemize}
\end{theorem}

\begin{remark} 
(i)
In  \cite{ELR21}, 
Engel and Kuehn gave several two-dimensional examples (see Examples 2 and 3 in \cite{ELR21}) to show the existence 
of Crauel random periodic solutions, 
which corresponds to a unique limit cycle of 
related deterministic system multiplied by a random equilibrium. 

Inspired by \cite{ELR21}, 
Theorem \ref{main-thm-2} (ii.2) 
provides a different model, 
which have infinitely many Crauel random periodic solutions 
that correspond to infinitely many periodic orbits of 
deterministic Kolmogorov system 
multiplied by a random equilibrium. 
The existence of infinitely many Crauel random periodic solutions makes it possible to further consider Poincar\'e bifurcation in the stochastic setting 
(for poincar\'e bifurcation in the deterministic setting see \cite{dla}.)

(ii) It is known that positive Lyapunov exponents are associated with chaotic behavior, and the zero Lyapunov exponent is related to the bifurcation phenomenon (See \cite{al, BBS22.1, ELR19}).
 
The new bifurcation phenomena 
$(ii.1_a)$-$(ii.1_c)$ displaying in the subcase $(ii.1)$ 
of Theorem \ref{main-thm-2} 
indeed relates to the number of 
zero Lyapunov exponents. 
To be more precise, let us take the random equilibrium $u_g(\omega)\textbf{e}_1$ for an example. 
The number of zero Lyapunov exponents 
of $u_g(\omega)\textbf{e}_1$ is at most one 
in the case of Theorem \ref{main-thm-2} $(ii.1_a)$,  
while at most two in the case of Theorem \ref{main-thm-2} $(ii.1_b)$, 
but in the case of Theorem \ref{main-thm-2} $(ii.1_c)$  
the number of zero Lyapunov exponents is exactly two. 

This fact shows that there exist rich dynamics 
in the zero Lyapunov exponent regime.
\end{remark}

\subsubsection{Further comments}

\paragraph{$(i)$ {\it Classification of global dynamics for deterministic Kolmogorov system:}}  
The complete classification of global dynamics  
is also proved for the deterministic Kolmogorov system \eqref{sys6}. 
More precisely, 
we prove that there are 6 different topological phase portraits 
in Subsection \ref{classdeterministic} below.  
For the convenience of readers, 
the visual phase diagrams 
are  shown in Figure \ref{global dynamics}.   

Furthermore, 
the bifurcation phenomenon of global dynamics is shown 
for deterministic system \eqref{sys6}, 
which is related to the loss of the hyperbolicity of some
 orbits such as equilibrium and periodic orbits. 
See Figures \ref{line}-\ref{positive}.

It is worth noting that, 
compared to stochastic Kolmogorov system \eqref{sys5}, 
the deterministic Kolmogorov system \eqref{sys6} has more delicate dynamics such as heteroclinic orbits 
(see Figure \ref{global dynamics} (iii.a) below).

\paragraph{$(ii)$ {\it Vanishing noise limit:}}
The relationship, 
via vanishing noise limit, 
between ergodic measures 
for the stochastic and deterministic Kolmogorov systems is studied as well. 

As the noise intensity tends to zero, 
we prove that 
the ergodic stationary measures of stochastic Kolmogorov system \eqref{sys5} converges to
the ergodic invariant measures of the deterministic system \eqref{sys6},
which are Dirac measures supported on equilibria 
or Haar measures supported on 
periodic orbits. 

In order to characterize the support of the invariant measures, 
we use the Poincar\'e recurrence theorem. 
The detailed proof is contained in Subsection \ref{Subsec-Vanish} below.

\medskip 
{\bf Organization:}   
In Section \ref{Sec-Deter-Kol},   
we give the complete classification of global dynamics and show the global bifurcation diagrams of the deterministic Kolmogorov system \eqref{sys6}.  
Then, Section \ref{Sec-Stoch-form} contains a stochastic decomposition formula, 
which connects solutions to deterministic and stochastic Kolmogorov systems. 
Several useful long-term dynamical behaviors of logistic-type equations are shown there as well. 
Sections \ref{pull-back limit set}-\ref{Sec-Stocha-Kol} are mainly devoted to the stochastic Kolmogorov system \eqref{sys5}. 
We first characterize the pull-back $\Omega$-limit sets 
in Section \ref{pull-back limit set}.  
In Section \ref{Sec-Ergo-Meas}, we obtain two types of ergodic stationary measures related to equilibria and invariant cones, and establish the relationship, 
via the vanishing noise limit, between stationary measures for the deterministic and stochastic Kolmogorov systems.  
Section \ref{Sec-Stocha-Kol} contains the proof of the main results, i.e., Theorems \ref{bifurcations} and \ref{main-thm-2}. Finally, the Appendix contains some preliminaries of random dynamical systems and probability used in this paper.

\medskip
\paragraph{\bf A guide to notations}
For the convenience of readers,
we list the notations
that are used in this paper.

\paragraph{\bf Deterministic Kolmogorov system:}
\begin{itemize} 
 \item  $\textrm{Int}\mathbb{R}^3_+:=\{(x_1, x_2, x_3)\in \mathbb{R}^3_+: x_i>0, \,\forall \,i=1,2,3\}$ denotes the interior of $\mathbb{R}^3_+$,
  $\partial\mathbb{R}^3_+:=\{(x_1, x_2, x_3)\in \mathbb{R}^3_+: x_i=0, \,\, \exists \,\, i\in \{1, 2, 3\}\}$ is the boundary of $\mathbb{R}^3_+$,
  $\mathbb{S}^2:=\{x\in \mathbb{R}^3: x_1^2+x_2^2+x_3^2=1\}$ is the unit sphere in $\mathbb{R}^3$,
  and $\mathbb{S}^2_+: =\mathbb{S}^2 \cap \mathbb{R}_+^3$, $\partial\mathbb{S}^2_+:=\mathbb{S}^2_+\cap \partial\mathbb{R}^3_+$ denotes the boundary of $\mathbb{S}^2_+$, $\textrm{Int}\mathbb{S}^2_+:=\mathbb{S}^2_+ \cap \textrm{Int}\mathbb{R}^3_+$ denotes the interior of $\mathbb{S}^2_+$.

 \item Let $\Psi=\Psi(t, x)$ denote the solution to the
 deterministic Kolmogorov system \eqref{sys6} at time $t$ with the initial value $x\in \mathbb{R}_+^3$.
 
  \item $\mathcal{E}$ denotes the set of all equilibria of system \eqref{sys6},
  that is, if $Q\in \mathcal{E}$, then $\Psi(t, Q)=Q, \,\,\forall \,\,t \geq 0$.
 
\item $\mathcal{L}(y):=\{\lambda y: \lambda>0\}$
denotes the ray passing through the point $y$, $y\in \mathbb{R}^3_+$, and $\overline{\mathcal{L}(y)}$ is the closure of $\mathcal{L}(y)$ in $\mathbb{R}^3_+$.

\item $\Gamma(h)$ is the closed orbit for each $h\in ({h}^*, \infty)$, where
    \begin{equation*}
{h}^*:= \prod\limits_{i=1}^3 \left(\frac{\alpha+d_{4-i}}{3\alpha+d_1+d_2+d_3}\right)^{-\frac{\alpha+d_{4-i}}{3\alpha+d_1+d_2+d_3}}.
\end{equation*}

\item $\Lambda(h):=\{\lambda y: y\in \Gamma(h), \lambda\ge 0\}$ is the cone for $h\in ({h}^*, \infty)$.

\item $\omega_d(x)$ is the $\omega$-limit set of the deterministic trajectory $\Psi$ to \eqref{sys6}, defined by, for $x\in \mathbb{R}^3_+$,
\begin{equation}\label{omega}
 \omega_d(x)=\{y:  \exists\ \textrm{an sequence} \ t_k \ \textrm{such that} \lim_{t_k \uparrow+\infty}\Psi(t_k, x)=y \}.
\end{equation}
Correspondingly, the attracting domain of $\omega_d(x)$ is defined by
$$\mathcal{A}(\omega_d(x)):=\{y \in \mathbb{R}_+^3: \displaystyle \lim_{t\rightarrow +\infty} \textrm{dist}(\Psi(t,y),\omega_d(x))=0\}.$$
In particular, for $Q\in \mathcal{E}$,
$$\mathcal{A}(Q):=\{y \in \mathbb{R}_+^3: \displaystyle \lim_{t\rightarrow +\infty} \textrm{dist}(\Psi(t,y),Q)=0\},$$
and for $h\in ({h}^*, \infty)$,
$$\mathcal{A}(\Gamma(h)):=\{y \in \mathbb{R}_+^3: \displaystyle \lim_{t\rightarrow +\infty} \textrm{dist}(\Psi(t,y),\Gamma(h))=0\}.$$

\item $\alpha_d(x)$ is the $\alpha$-limit set of the deterministic trajectory $\Psi$ to \eqref{sys6}, defined by,
  for $x\in \mathbb{R}^3_+$,
\begin{equation}\label{alpha-omega}
 \alpha_d(x)=\{y:  \exists\ \textrm{a sequence} \ t_k \ \textrm{such that} \lim_{t_k\downarrow-\infty}\Psi(t_k, x)=y \}.
\end{equation}
\end{itemize}

\paragraph{\bf Stochastic Kolmogorov system:}
\begin{itemize}
  \item  Let $\mathcal{B}(\mathbb{R}^3)$ be the Borel  $\sigma$-algebra
on $\mathbb{R}^3$, and $\mathcal{B}_b(\mathbb{R}^3)$ (resp. $\mathcal{C}_b(\mathbb{R}^3))$
be the set of all real bounded Borel (resp. continuous)
measurable functions on $\mathbb{R}^3$.

  \item Let $\Phi=\Phi(t,\omega, x)$ be the solution to the stochastic Kolmogorov system \eqref{sys5}
  at time $t$ with the initial value $x\in \mathbb{R}_+^3$, $\omega \in \Omega$.
      Let $a=(a^{ij})$ and $b$ be the corresponding diffusion matrix and drift term of  \eqref{sys5},
      respectively.

  \item
      Let $\mathcal{P}(\mathbb{R}^3)$ and $\mathcal{P}(\mathbb{R}^3_+)$
   denote, respectively,
   the set of all probability measures on $\mathbb{R}^3$ and $\mathbb{R}_+^3$,
   $\mathcal{P}_e(\mathbb{R}^3_+)$
   is the set of all ergodic stationary measures of $\Phi$.

\item  $(P_t)$ is the Markov semigroup corresponding to the stochastic Kolmogorov system \eqref{sys5}.

  \item The pull-back $\Omega$-limit set of the trajectory $\{\Phi(t, \theta_{-t}\omega, x)\}_{t\geq0}$ is defined by
\begin{equation*}\label{pullback}
     \Omega_x(\omega):=\bigcap_{t>0}\overline{\bigcup_{\tau\geq t}\Phi(\tau, \theta_{-\tau}\omega,x)}.
\end{equation*}

  \item $\mu_Q$ is the stationary measure related to the equilibrium $Q \in \mathcal{E}$.

  \item $\nu_h$ is the stationary measure related to
 the cone $\Lambda(h)$.

  \item $\mathscr{L}^{\sigma}$ is the Fokker-Planck operator  defined by
$$ \mathscr{L}^{\sigma}f(x):= \langle \nabla f(x), b(x) \rangle + \frac12 a^{ij}\partial^2_{ij}f(x),\ \ \forall f \in C^2. $$

\end{itemize}

\section{Deterministic Kolmogorov system}   \label{Sec-Deter-Kol}

This section is devoted to the topological classification and bifurcations of global dynamics of
the deterministic Kolmogorov system \eqref{sys6}.

Note that system \eqref{sys6} has three invariant planes $x_i=0$, $i=1,2,3$, and an invariant sphere
\begin{equation}\label{shpere}
 \mathbb{S}^2=\{(x_1,x_2,x_3): \ x_1^2+x_2^2+x_3^2=1\}\subset \mathbb{R}^3.
\end{equation}
We first prove that system \eqref{sys6} is dissipative in $\mathbb{R}^3$ and the invariant sphere $\mathbb{S}^2$
is a global attractor  in $\mathbb{R}^3\setminus \{O\}$ for $\alpha >0$
and all $(d_1,d_2,d_3)\in \mathbb{R}^3$.
Due to the axisymmetry of system \eqref{sys6},
it suffices to study the topological classification
of the global dynamics of system \eqref{sys6}
in the first octant $\mathbb{R}_+^3$, here
\begin{equation*}
 \mathbb{R}^3_+=\{(x_1,x_2,x_3)\in\mathbb{R}^3: \ x_1\ge 0, x_2\ge 0, x_3\ge 0\}.
\end{equation*}
 Since $\mathbb{S}^2$
is a global attractor of system \eqref{sys6}, we only study the topological classification
of the global dynamics of system \eqref{sys6} on
the  invariant sphere $\mathbb{S}^2_+$, where
\begin{equation*}
 \mathbb{S}^2_+=\{(x_1,x_2,x_3)\in\mathbb{R}_+^3: \ x_1^2+x_2^2+x_3^2=1\}\subset \mathbb{R}_+^3.
\end{equation*}

We say that two global dynamics of system \eqref{sys6} on the invariant sphere $\mathbb{S}^2_+$ are {\it topologically equivalent} if there exists a homeomorphism from one onto the other
which sends orbits on  $\mathbb{S}^2_+$ of system \eqref{sys6} to orbits preserving or reversing the direction of the flow. Our main aim is to prove that the global dynamics of system \eqref{sys6}
on the invariant sphere $\mathbb{S}^2_+$ have and only have 6 different topological classifications, whose phase portraits are shown in
Figure \ref{global dynamics} (i) - (v). Moreover, choosing $\alpha+d_1$, $\alpha+d_2$
and $\alpha+d_3$ as bifurcation parameters of system \eqref{sys6}, denoted by $(m_1,m_2,m_3)$ for simplicity,  we consider bifurcation of system \eqref{sys6} in
the parameter space $(m_1,m_2,m_3)\in \mathbb{R}^3$ at bifurcation point $(0,0,0)$, and obtain the global bifurcation diagram and the corresponding topological phase portraits of system \eqref{sys6}, see Figures \ref{line}-\ref{positive}.

\subsection{Classification of global dynamics}\label{classdeterministic}

We first prove that system \eqref{sys6} is dissipative in $\mathbb{R}^3$ and the invariant sphere $\mathbb{S}^2$
is a global attractor  in $\mathbb{R}^3\setminus \{O\}$ for $\alpha >0$
and all $(d_1,d_2,d_3)\in \mathbb{R}^3$.

\begin{lemma}(Global attractor)\label{dglobal}
System \eqref{sys6} is dissipative in $\mathbb{R}^3$,
and the invariant sphere $\mathbb{S}^2$ given by \eqref{shpere}
is a global attractor in $\mathbb{R}^3\setminus \{O\}$.
That is,
$\omega_d (x_0) \subset \mathbb{S}^2$ for any $x_0\in \mathbb{R}^3 \setminus \{O\}$.
\end{lemma}

\begin{proof}
Since the origin $O$ is an equilibrium of system \eqref{sys6}
and all three eigenvalues of the Jacobian matrix at $O$ are positive,
$O$ is a local repeller of system \eqref{sys6}.

Hence,
for any $x_0\in \mathbb{R}^3\setminus\{O\}$
there exists a constant $c(x_0)>0$ such that the solution $\Psi(t, x_0)$ of system \eqref{sys6} passing through $x(0)=x_0$ satisfies
\begin{equation}\label{estimate0}
   \inf\limits_{t\geq 0} \|\Psi(t, x_0)\|\geq c(x_0)
   >0.
\end{equation}

Let
$$L(x):=x_1^2+x_2^2+x_3^2-1,\,\, x=(x_1, x_2, x_3)\in \mathbb{R}^3.$$
Then, by straightforward computations,
for any $x_0\in \mathbb{R}^3$,
\begin{equation}\label{Leq}
\frac{dL(\Psi(t,x_0))}{dt}|_{\eqref{sys6}}
 =-2\alpha\|\Psi(t, x_0)\|^2L(\Psi(t, x_0))
 \left\{\begin{array}{lll}
<0, \ & {\textrm{if}}\ L(x_0)>0;\\
=0, \ & {\textrm{if}}\ L(x_0)=0;\\
>0, \ & {\textrm{if}}\ L(x_0)<0.
\end{array}
\right.
\end{equation}
This yields that system \eqref{sys6} is dissipative in $\mathbb{R}^3$.

Further, from equation \eqref{Leq} and \eqref{estimate0},  we have
$$\|L(\Psi(t, x_0))\|\leq \|L(x_0)\|\exp\{\int_0^t -2\alpha {c^2(x_0)}ds\},\ \forall t\ge 0, x_0\in \mathbb{R}^3\setminus \{O\}.$$
Thus,
$$\displaystyle \lim_{t\rightarrow +\infty}\|L(\Psi(t, x_0))\|=0.$$
This yields that $\omega_d(x_0)\subseteq \mathbb{S}^2$ for any $x_0\in \mathbb{R}^3\setminus \{O\}$, hence, the invariant sphere $\mathbb{S}^2$
is a global attractor of system \eqref{sys6} in $\mathbb{R}^3\setminus \{O\}$.
\end{proof}


Note that the existence of first integrals plays important role in the study of dynamics of differential systems. To study global dynamics of system \eqref{sys6} in $\mathbb{R}_+^3$, we try to find the first integrals of system \eqref{sys6}.
 Since system \eqref{sys6} has four invariant algebraic surfaces: three coordinate planes and $\mathbb{S}^2$, by virtue of the Darboux theory of integrability  in \cite{Darboux}
we construct the  first integrals of system \eqref{sys6} in the interior of $\mathbb{R}_+^3$ denoted by $\textrm{Int}\mathbb{R}_+^3$  as follows, where
$$\textrm{Int}\mathbb{R}_+^3=\{(x_1,x_2,x_3)\in\mathbb{R}^3: \ x_1>0, x_2> 0, x_3> 0\}.$$
\begin{lemma}\label{first integral}(Existence of first integrals).
\begin{itemize}
  \item[(i)] If $3\alpha+d_1+d_2+d_3\neq 0$, then system \eqref{sys6} has
   a first integral $H_1(x_1, x_2, x_3)$ in  $\textrm{Int}\mathbb{R}_+^3$,
      $$
      H_1(x_1, x_2, x_3)= \prod\limits_{i=1}^3 x_i^{-\frac{2(\alpha+d_{4-i})}{3\alpha+d_1+d_2+d_3}} \|x\|^2, $$
 where $\|x\|=\sqrt{\sum_{i=1}^3x_i^2}.$
  \item [(ii)] If $3\alpha+d_1+d_2+d_3=0$ and $\sum_{i=1}^{3}(\alpha+d_i)^2\neq 0$,
  then system \eqref{sys6} has  a first integral $H_2(x_1, x_2, x_3)$ in $\textrm{Int}\mathbb{R}_+^3$, where
      $$H_2(x_1, x_2, x_3)= \prod\limits_{i=1}^3 x_i^{\alpha+d_{4-i}}.$$
\end{itemize}
\end{lemma}
\begin{proof}
Since $H_1(x_1, x_2, x_3)$ and $H_2(x_1, x_2, x_3)$ are continuously differentiable functions in $\textrm{Int}\mathbb{R}_+^3$,  from straightforward computations we have $\forall x\in \textrm{Int}\mathbb{R}_+^3$,
\begin{equation*}
\begin{array}{ll}
\langle b(x), \nabla H_1(x) \rangle\equiv 0, \ & {\textrm{if}}\ \sum_{i=1}^{3}(\alpha+d_i)\neq 0;\\
\langle b(x), \nabla H_2(x) \rangle\equiv 0, \ & {\textrm{if}}\ \sum_{i=1}^{3}(\alpha+d_i)=0, \sum_{i=1}^{3}(\alpha+d_i)^2\neq 0,
\end{array}
\end{equation*}
where $b(x)$ is vector field of system \eqref{sys6} (or the drift of $\Psi$)  in $\textrm{Int}\mathbb{R}_+^3$, and $\langle \cdot,\cdot \rangle$ is an inner product. Hence, $H_1(x_1, x_2, x_3)$ is a first integral
of system \eqref{sys6} in  $\textrm{Int}\mathbb{R}_+^3$ if $3\alpha+d_1+d_2+d_3\neq 0$, and $H_2(x_1, x_2, x_3)$ is a first integral
of system \eqref{sys6} in  $\textrm{Int}\mathbb{R}_+^3$ if $3\alpha+d_1+d_2+d_3=0$ and $\sum_{i=1}^{3}(\alpha+d_i)^2\neq 0$.
\end{proof}
Note that the level set of the first integral
\begin{equation}\label{cone0}
\Lambda_i(h):=\{(x_1, x_2, x_3): H_i(x_1, x_2, x_3)=h\in I_i \}, \ i=1,2
\end{equation}
is invariant under the flow  $\Psi$ of system \eqref{sys6} in $\textrm{Int}\mathbb{R}_+^3$ by definition of the first integral,
where $I_i\subset \mathbb{R}$ is the image interval of $H_i(x)$ in $\textrm{Int}\mathbb{R}_+^3$.
And $\textrm{Int}\mathbb{R}_+^3$ is foliated by $\Lambda_i(h)$ for any a $h\in I_i$. Hence,  system \eqref{sys6} in $\textrm{Int}\mathbb{R}_+^3$ can be reduced to a differential system on  $\Lambda_i(h)$.

For the sake of the statement, we recall some terminology.
An equilibrium point of system \eqref{sys6} in $\mathbb{R}_+^3$ is called {\it boundary equilibrium}
if at least one of its coordinates is zero,
otherwise it is called {\it positive equilibrium}, that is, three coordinates of the equilibrium point are positive. An equilibrium point is called {\it isolated  equilibrium}
if there is a neighborhood of the equilibrium point in $\mathbb{R}^3$ such that
there is no other equilibrium point in this neighborhood, otherwise the equilibrium point is said to be {\it non-isolated}. The topological classification of
 an equilibrium point can be characterized by its local stable, unstable and center manifolds, see the invariant manifold theorem in \cite{dla}. And these local manifolds
 of an equilibrium point are closely related to the sign of the real parts of eigenvalues of the Jacobi matrix of system \eqref{sys6}
at the equilibrium point.
An equilibrium has {\it $k$-dimensional local stable (resp. unstable, center) manifold}
if there are exactly $k$ eigenvalues $\lambda_i$  with  ${\rm Re}(\lambda_i)<0 \ (resp. >0, \ resp. =0)$, where $1\le k\le 3$.
 An equilibrium point is called {\it hyperbolic equilibrium} if the real parts of all eigenvalues are not zero, otherwise it is called {\it non-hyperbolic equilibrium}.
Further, if there is at least one zero eigenvalue of the equilibrium, then the non-hyperbolic equilibrium is said to be {\it degenerated}.

We are now in the position to study the local dynamics of system \eqref{sys6} in $\mathbb{R}^3_+$ including the existence and topological classification of equilibrium points. It is clear that
system \eqref{sys6} always has four boundary equilibrium points $O=(0,0,0)$,
$\textbf{e}_1=(1,0,0)$, $\textbf{e}_2=(0,1,0)$ and $\textbf{e}_3=(0,0,1)$ in $\mathbb{R}^3_+$
for any  $\alpha>0$ and $(d_1,d_2,d_3)\in \mathbb{R}^3$. Using straightforward computations,
we obtain  all equilibria of system \eqref{sys6} in $\mathbb{R}^3_+$ as follows.

\begin{proposition}(Existence of equilibria) \label{pro1}
System \eqref{sys6}  has only isolated equilibria in $\mathbb{R}^3_+$ if and only if $\Pi_{i=1}^3(\alpha+d_i)\not=0$.
More precisely,
\begin{itemize}
  \item [($i$)] if  $\alpha+d_i>0$ ($\alpha+d_i<0$, resp.) for all $i\in\{1,2,3\}$, then system \eqref{sys6}  has only five isolated equilibria
  $O, \textbf{e}_1, \textbf{e}_2, \textbf{e}_3, Q^*$ in $\mathbb{R}^3_+$,
  where $Q^*=(q^*_1, q^*_2, q^*_3)$ is positive equilibrium, here
  $$q^*_i=\sqrt{\dfrac{\alpha+d_{4-i}}{3\alpha+d_1+d_2+d_3}},\ \ i=1,2,3;$$
  \item [($ii$)] if $\prod_{i=1}^3(\alpha+d_i) \neq 0$ and there exist $i\neq j$, $i,j\in \{1,2,3\}$,
  such that $(\alpha+d_i)(\alpha+d_j)<0$, then  system \eqref{sys6} has only four isolated boundary equilibria
  $O, \textbf{e}_1, \textbf{e}_2, \textbf{e}_3$ in $\mathbb{R}^3_+$.

\end{itemize}

System \eqref{sys6} has both isolated equilibria and non-isolated equilibria in $\mathbb{R}^3_+$ if and only if $\Pi_{i=1}^3(\alpha+d_i)=0$.
More precisely,
\begin{itemize}
  \item [($iii$)] if there is only one $i_0\in \{1, 2, 3\}$ such that $\alpha+d_{4-i_0}=0$ and $\alpha+d_j\not=0$
  for all $ j\in\{1,2,3\}\setminus\{i_0\}$, then system \eqref{sys6} has only two isolated equilibria $O, \textbf{e}_{i_0}$,
  $i_0\in \{1,2,3\}$,
  and infinitely many non-isolated equilibria
  which fills the curve  section
  \begin{equation}\label{curves}
  \Gamma_{ij}:=\{x\in \mathbb{R}_+^3: x^2_i+x^2_j =1, x_{i_0}=0\} \subset \mathbb{R}_+^3,
  \end{equation}
  where $i<j$ and $i, j\in \{1,2,3\}\setminus\{i_0\}$.
  \item[($vi$)] if there exist $i,j\in \{1,2,3\}$ such that $\alpha+d_i=0$,
  $\alpha+d_j=0$  and $\alpha+d_k\not=0$, where $k\in \{1,2,3\}\setminus\{i,j\}$, then system \eqref{sys6} has  a unique isolated equilibrium  $O$
  and infinitely many non-isolated equilibria which fill two curve sections of
  $\{\Gamma_{12}, \Gamma_{13}, \Gamma_{23}\}$.
 \item[($v$)] if $\alpha+d_i=0$ for all $i=1, 2, 3,$  then system \eqref{sys6} has  a unique isolated equilibrium  $O$
  and infinitely many non-isolated equilibria which fill the invariant sphere
  $\mathbb{S}^2_+$.
\end{itemize}

All equilibria of system \eqref{sys6} except $O$ in $\mathbb{R}_+^3$ are located on $\mathbb{S}^2_+$.
\end{proposition}

To discuss the topological classification of an equilibrium, we calculate three eigenvalues of each  isolated equilibrium and non-isolated equilibria of system \eqref{sys6} in $\mathbb{R}_+^3$. The following table gives
the possible isolated equilibria and the corresponding three eigenvalues.
\begin{table}[H]
		\centering
		\caption{Possible isolated equilibria and the corresponding three eigenvalues }\label{SecType_eq1_atleast7}
		\setlength{\tabcolsep}{10pt}
		\begin{tabular}{ll}
			\toprule
			Equilibrium & three eigenvalues\\
			\midrule
			$O=(0,0,0)$ & $\alpha,\,\alpha,\,\alpha$\\
			$\textbf{e}_1=(1,0,0)$ & $-2\alpha,\,\alpha+d_1,\,-(\alpha+d_2)$\\
			$\textbf{e}_2=(0,1,0)$ & $\,-(\alpha+d_1),-2\alpha,\,\alpha+d_3$\\
			$\textbf{e}_3=(0,0,1)$ & $\,\alpha+d_2,\,-(\alpha+d_3),-2\alpha$\\
            $Q^*=(q_1^*, q_2^*, q_3^*)$ & $\lambda_{Q^*}i,  -\lambda_{Q^*}i, -2\alpha$, here $\lambda_{Q^*}=2\sqrt{\frac{(\alpha+d_1)(\alpha+d_2)(\alpha+d_3)}{3\alpha+d_1+d_2+d_3}}$\\

			\bottomrule
		\end{tabular}
	\end{table}

Even though there are three (two) cases for system \eqref{sys6} having non-isolated equilibria (only isolated equilibria, resp.) in Proposition \ref{pro1}, there exist many different sets of parameter conditions of system \eqref{sys6} in these cases (i) - (vi), i.e.
the case (i) ((ii), (iii), (vi)) has  two (six, twelve, six, resp.) different sets of parameter conditions.
Note that the two (six, twelve, six) different sets of parameter conditions in case (i) ((ii), (iii), (vi), resp.) can be exchanged to one (one, two, one) different sets of parameter conditions in case (i) ((ii), (iii), (vi), resp.) under either permutation of the order among coordinates $(x_1, x_2, x_3)$ or change time $t$ to $-t$ if we consider dynamics of system \eqref{sys6} on $\mathbb{S}^2_+$. Hence, in the sense of topologically equivalent,  we only need to consider topological classification of equilibria of system \eqref{sys6} in the following six different sets of parameter conditions:
  \begin{itemize}
  \item [(i)] $\alpha+d_1>0, \alpha+d_2>0$, and $\alpha+d_3>0$;
  \item [(ii)]  $\alpha+d_1<0, \alpha+d_2>0$, and $\alpha+d_3<0$;
  \item [(iii.a)]  $\alpha+d_1=0, \alpha+d_2>0$, and $ \alpha+d_3>0$;
  \item [(iii.b)]  $\alpha+d_1>0, \alpha+d_2=0$, and $ \alpha+d_3<0$;
  \item [(vi)]  $\alpha+d_1=0, \alpha+d_2=0$, and $\alpha+d_3<0$;
  \item [(v)]  $\alpha+d_1=0, \alpha+d_2=0$, and $\alpha+d_3=0$.
  \end{itemize}

 We denote the ray passing through the point $P\in \mathbb{R}^3_+$ by $\mathcal{L}(P):=\{\lambda P: \lambda>0\}$. And $\Gamma_{ij}$ defined by \eqref{curves} is the curve section. Lemmas \ref{isolated} and  \ref{contiunnm} below shows the local dynamics of every equilibria of system \eqref{sys6} in $\mathbb{R}^3_+$ under the above six different sets of parameter conditions.
\begin{lemma}\label{isolated}
If system \eqref{sys6} has isolated equilibria, then these isolated equilibria are all hyperbolic expect the positive equilibrium $Q^*$. Moreover, the equilibrium $O$ always is a local repeller with three-dimensional unstable manifold in $\mathbb{R}^3_+$, and the local dynamics of others are as follows.
\begin{itemize}
  \item [(i)] If $\alpha+d_i>0$, $i=1,2,3$, then boundary equilibrium $\textbf{e}_1$ has two-dimensional stable manifold on plane $\{x\in \mathbb{R}_+^3:\ x_2=0\}$ and one-dimensional unstable manifold on curve section $\Gamma_{12}$; $\textbf{e}_2$ has two-dimensional stable manifold on plane $\{x\in \mathbb{R}_+^3:\ x_3=0\}$ and one-dimensional unstable manifold on $\Gamma_{23}$; $\textbf{e}_3$ has two-dimensional stable manifold on plane $\{x\in \mathbb{R}_+^3:\ x_1=0\}$ and one-dimensional unstable manifold on  $\Gamma_{13}$; and positive equilibrium $Q^*$ is a center on its two-dimensional center manifold in $\mathbb{S}^2_+$  and $Q^*$ has a one-dimensional stable manifold $\mathcal{L}(Q^*)$ in $\mathbb{R}^3_+$.
 \item [(ii)] If $\alpha+d_1<0, \alpha+d_2>0$ and $\alpha+d_3<0$, then boundary equilibrium $\textbf{e}_1$ has three-dimensional stable manifold on $\mathbb{R}^3_+$; $\textbf{e}_2$ has two-dimensional stable manifold on plane $\{x\in \mathbb{R}_+^3:\ x_1=0\}$ and one-dimensional unstable manifold on  $\Gamma_{12}$;  $\textbf{e}_3$ has one-dimensional stable manifold on the positive $x_3$-axis and two-dimensional unstable manifold on $\mathbb{S}^2_+$.
\end{itemize}
\end{lemma}
\begin{proof}
All eigenvalues of the Jacobi matrix at each isolated equilibrium have been shown in Table \ref{SecType_eq1_atleast7}. Then by Proposition \ref{pro1}, it is not hard to check that each isolated equilibrium is hyperbolic except the positive equilibrium. Clearly, the three eigenvalues of the boundary equilibrium $O$ are $\alpha>0$. Thus, $O$ is a local repeller with a three-dimensional unstable manifold in $\mathbb{R}^3_+$. In the following, we consider the local dynamics of the other isolated equilibria in case ($i.a$) and case ($ii.a$).

{\bf Case (i)}:  if $\alpha+d_i>0, i=1,2,3$, then the three eigenvalues of Jacobi matrix at boundary equilibrium  $\textbf{e}_1$ are $-2\alpha<0$, $\alpha+d_1>0$ and $-(\alpha+d_2)<0$, whose
  associated eigenvectors are $(1,0,0)$, $(0,1,0)$ and $(0,0,1)$, respectively. It can be checked that the positive $x_1$-axis, $\Gamma_{13}$ and $\Gamma_{12}$ is an invariant manifold of system \eqref{sys6}, which tangents to eigenvector $(1,0,0)$, $(0,0,1)$ and $(0,1,0)$, respectively. Hence, the two-dimensional stable manifold of $\textbf{e}_1$ is on the plane $\{x\in \mathbb{R}^3_+: x_2=0\}$ and the one-dimensional unstable manifold of $\textbf{e}_1$ is on $\Gamma_{12}$. Using the similar arguments, the local dynamics of boundary equilibria  $\textbf{e}_2$ and $\textbf{e}_3$ can be obtained.

It remains to verify the local dynamics of positive equilibrium $Q^*$.  Since the three eigenvalues of Jacobi matrix at $Q^*$ are $\pm \lambda_{Q^*}i$ and $-2\alpha$, $Q^*$ has a two-dimensional center manifold which is tangent at $Q^*$ to a plane spanned by the associated eigenvectors of  $\pm \lambda_{Q^*}i$  and a one-dimensional stable manifold which is tangent at $Q^*$ to a line spanned by the associated eigenvector of $-2\alpha$. Note that  $\mathbb{S}^2_+$ is a unique two-dimensional attractor passing through $Q^*$ by Lemma \ref{dglobal}. So the two-dimensional center manifold of $Q^*$ is on $\mathbb{S}^2_+$. Further, by Lemma \ref{first integral} we know that system \eqref{sys6} has a first integral $H_1(x)$, where
$x\in \textrm{Int}{\mathbb{R}^3_+}$. Therefore, the following reduced system of system \eqref{sys6} on  $\mathbb{S}^2_+$
\begin{equation}\label{reduce-2-dimension}
\begin{cases}
\frac{dx_1}{dt}=x_1(\alpha+d_2-(\alpha+d_2) x_1^2-(2\alpha+d_1+d_2)x_2^2),\\
\frac{dx_2}{dt}=x_2(-(\alpha+d_3)+(2\alpha+d_1+d_3)x_1^2+(\alpha+d_3)x_2^2)\\
\end{cases}
\end{equation}
has a first integral ${\tilde{H}}_1(x_1,x_2)$ in $\textrm{Int}{\mathbb{S}^2_+}$, where
$${\tilde{H}}_1(x_1,x_2)=x_1^{-\frac{2(\alpha+d_3)}{3\alpha+d_1+d_2+d_3}}x_2^{-\frac{2(\alpha+d_2)}{3\alpha+d_1+d_2+d_3}}(1-x_1^2-x_2^2)^{-\frac{(\alpha+d_1)}{3\alpha+d_1+d_2+d_3}}.$$
This leads that the positive equilibrium $Q^*$ is a center on  $\mathbb{S}^2_+$ by Poincar\'e center theorem.

Now we turn to prove that the ray $\mathcal{L}(Q^*)$ is exactly the one-dimensional stable manifold of $Q^*=(q_1^*, q_2^*, q_3^*)$. Since $Q^*=(q_1^*, q_2^*, q_3^*)$ is a positive equilibrium of system \eqref{sys6},
 we have
\begin{equation}\label{equi}
\begin{cases}
\alpha-\alpha (q_1^*)^2-(2\alpha+d_1) (q_2^*)^2+d_2 (q_3^*)^2=0,\\
\alpha+d_1 (q_1^*)^2-\alpha (q_2^*)^2-(2\alpha+d_3)(q_3^*)^2=0,\\
\alpha-(2\alpha+d_2)(q_1^*)^2+d_3 (q_2^*)^2-\alpha (q_3^*)^2=0.
\end{cases}
\end{equation}
For any $x\in \mathcal{L}(Q^*)\setminus\{Q^*\}$, there exists an  $1\not=s>0$ such that
$x= (s q_1^*, s q_2^*, s q_3^*)$. Then
the  vector field of system \eqref{sys6} at $x$ is
\begin{align*}
 b(x)&=\left(
             \begin{array}{c}
               sq_1^*(\alpha-s^2\alpha (q_1^*)^2-s^2(2\alpha+d_1)(q_2^*)^2+s^2d_2(q_3^*)^2) \\
               sq_2^*(\alpha+s^2d_1(q_1^*)^2-\alpha s^2(q_2^*)^2-(2\alpha+d_3)s^2(q_3^*)^2) \\
               sq_3^*(\alpha-(2\alpha+d_2)s^2(q_1^*)^2+d_3s^2(q_2^*)^2-\alpha s^2 (q_3^*)^2)\\
             \end{array}
           \right)\\
           &=\alpha s(1-s^2)\left(
              \begin{array}{c}
                q_1^* \\
                q_2^* \\
                q_3^* \\
              \end{array}
            \right)
 \end{align*}
 by \eqref{equi}.
Thus, $b(x)$ is parallel to the ray $\mathcal{L}(Q^*)$,
which implies that $\mathcal{L}(Q^*)$ is invariant under \eqref{sys6}.

Moreover, it follows from Lemma \ref{dglobal} that
$\omega_d(x)\subseteq \mathbb{S}^2_+$
for any $x\in \mathcal{L}(Q^*)$. Note that
$\omega_d(x)\subseteq  \mathbb{S}^2_+ \bigcap \mathcal{L}(Q^*) = \{Q^*\}$,
which yields that $\omega_d(x)=\{Q^*\}$. Thus, by the uniqueness of the stable manifold, $\mathcal{L}(Q^*)$ is the one-dimensional stable manifold of $Q^*$.

{\bf Case (ii)}: if $\alpha+d_1<0, \alpha+d_2>0$ and $\alpha+d_3<0$, then the local dynamics of each boundary equilibrium $\textbf{e}_i$ ($i=1,2,3$) can be characterized by the similar method in case ($i.a$). To save the space, we hence omit the proof.
\end{proof}

\begin{lemma}\label{contiunnm} If system \eqref{sys6} has non-isolated equilibria, then these non-isolated equilibria are non-hyperbolic. More precisely,
\begin{itemize}
  \item [(iii.a)] if $\alpha+d_1=0, \alpha+d_2>0$ and $\alpha+d_3>0$, then every points  on $\Gamma_{12}$
  are  non-isolated equilibria, and there exists a unique  non-isolated equilibrium $\bar{Q}:=(\bar{q}_1,\bar{q}_2, 0) \in \Gamma_{12}$ with $\bar{q}_1>0$, which divides  $\Gamma_{12}$ into two parts $\Gamma_{12}^-$ with $x_1<\bar{q}_1$
  and $\Gamma_{12}^+$ with $x_1>\bar{q}_1$ such that $\bar{Q}$ has one-dimensional stable manifold $\mathcal{L}(\bar{Q})$ and two-dimensional center manifold on $\mathbb{S}^2_+$; for any $Q_-\in \Gamma_{12}^-$, $Q_-$
  has one-dimensional stable manifold $\mathcal{L}({Q_-})$, one-dimensional center manifold on $\Gamma_{12}^-$ and one-dimensional unstable manifold on $\mathbb{S}^2_+$; for any $Q_+\in \Gamma_{12}^+$, $Q_+$
  has two-dimensional stable manifold spanned by $\mathcal{L}({Q_+})$ and a curve on $\mathbb{S}^2_+$, one-dimensional center manifold on $\Gamma_{12}$;
   \item [(iii.b)] if  $\alpha+d_1>0, \alpha+d_2=0$ and $ \alpha+d_3<0$, then every points on $\Gamma_{13}$ are non-isolated equilibria. And for any $Q\in \Gamma_{13}$, it has one-dimensional unstable manifold on $\mathbb{S}^2_+$, one-dimensional center manifold on $\Gamma_{13}$ and one-dimensional stable manifold $\mathcal{L}(Q)$.

  \item [(vi)] if  $\alpha+d_1=0, \alpha+d_2=0$ and $\alpha+d_3<0$, then every points on either  $\Gamma_{12}$ or $\Gamma_{13}$ are non-isolated equilibria. For any $Q\in \Gamma_{12}$, it has one-dimensional center manifold $\Gamma_{12}$ and two-dimensional stable manifold spanned by $\mathcal{L}(Q)$ and a curve on $\mathbb{S}^2_+$. And for any $Q\in \Gamma_{23}$, it has one-dimensional center manifold $\Gamma_{23}$, one-dimensional stable manifold $\mathcal{L}(Q)$ and one-dimensional unstable manifold on $\mathbb{S}^2_+$.

  \item [(v)] if  $\alpha+d_1=0, \alpha+d_2=0$ and $\alpha+d_3=0$, then every points in $\mathbb{S}^2_+$ are non-isolated equilibria. For any $Q\in \mathbb{S}^2_+$, $Q$ has one-dimensional stable manifold $\mathcal{L}(Q)$ and two-dimensional center manifold on $\mathbb{S}^2_+$.
  \end{itemize}
\end{lemma}

\begin{proof}
Based on the analysis of three eigenvalues and the corresponding invariant manifold of a non-isolated equilibrium,  we can obtain the conclusions in Lemma \ref{contiunnm}. Due to similar arguments, we only prove one case of four cases, for example,
case (iii.a) as follows.

If $\alpha+d_1=0, \alpha+d_2>0$ and $\alpha+d_3>0$, then the three eigenvalues of Jacobi matrix at the non-isolated equilibrium $Q(x_1, x_2, 0)\in \Gamma_{12}$ are $\lambda_1=0, \lambda_2=-2\alpha, \lambda_3=-(2\alpha+d_2+d_3)x_1^2+\alpha+d_3$. 

Note that $0<\frac{\alpha+d_3}{2\alpha+d_2+d_3}<1$. Let $\bar{q}_1:= \sqrt{\frac{\alpha+d_3}{2\alpha+d_2+d_3}}$. Then $0<\bar{q}_1<1$. Thus, $\bar{Q}=(\bar{q}_1, \bar{q}_2, 0)$ is the unique non-isolated equilibrium in $ \Gamma_{12}$ such that
the corresponding eigenvalues of $\bar{Q}$ are $0, -2\alpha, 0$, where $\bar{q}_2:=\sqrt{1-\bar{q}^2_1}$.  This yields that $\bar{Q}$ has a two-dimensional center manifold and a one-dimensional stable manifold. By the same method in the proof of case (i.a) in Lemma \ref{isolated}, we obtain that $\mathcal{L}(\bar{Q})$ is invariant and for any $x\in \mathcal{L}(\bar{Q})$, $\omega_d(x)=\{\bar{Q}\}$. Then, by the uniqueness of the stable manifold, $\mathcal{L}(\bar{Q})$ is the one-dimensional stable manifold of  $\bar{Q}$.
Since $\bar{Q}\in \mathbb{S}^2_+$ and $\mathbb{S}^2_+$ is a global attractor of system \eqref{sys6} in $\mathbb{R}_+^3$, the two-dimensional center manifold of $\bar{Q}$ is on $\mathbb{S}^2_+$ by the invariant manifold theory.

We now consider the non-isolated equilibrium in $\Gamma_{12}\setminus\{\bar{Q}\}$.

If $Q_-\in \Gamma_{12}^-$, then the eigenvalues of $Q_-$ are $0$, $\lambda_{1Q_-}<0$ and $\lambda_{2Q_-}>0$. Hence, the non-isolated equilibrium $Q_-$
 has one-dimensional stable manifold $\mathcal{L}(Q_-)$, one-dimensional unstable manifold on $\mathbb{S}^2_+$ and one-dimensional center manifold on $\Gamma_{12}$.

 If $Q_+\in \Gamma_{12}^+$, then the eigenvalues of $Q_+$ are $0$, $\lambda_{1Q_+}<0$ and $\lambda_{2Q_+}<0$. It can be checked that $Q_+$ has a one-dimensional center manifold on
$\Gamma_{12}$ and a two-dimensional stable manifold spanned by the ray $\mathcal{L}(Q_+)$ and a curve on  $\mathbb{S}^2_+$.
\end{proof}

 Let
\begin{equation}\label{h2}
\begin{split}
{h}^*:&=H_1(Q^*)= \prod\limits_{i=1}^3 \left(\frac{\alpha+d_{4-i}}{3\alpha+d_1+d_2+d_3}\right)^{-\frac{\alpha+d_{4-i}}{3\alpha+d_1+d_2+d_3}},\\
\end{split}
\end{equation}
where $H_1(x)$ is the first integral of system \eqref{sys6} in Lemma \ref{first integral}, $Q^*$ is the positive equilibrium, and $\bar{Q}$ is a non-isolated boundary equilibrium, whose first two coordinates are $\bar{q}_1$ and $\bar{q}_2$ in Lemma \ref{contiunnm}. We are now ready to classify the global dynamics of system \eqref{sys6}.

\begin{theorem}(Classification of global dynamics)\label{gd1}
Global dynamics of system \eqref{sys6} has and only has the following 6 different topological phase portraits in $\mathbb{R}^3_+$.
\begin{itemize}
\item [(i)]When $\alpha+d_i>0$, $i=1,2,3$, the global attractor $\mathbb{S}_+^2$  consists of periodic orbits $\Gamma(h)=\mathbb{S}_+^2\cap \Lambda_1(h)$ for any $h\in (h^*, \infty)$, positive equilibrium $Q^*$ and the heteroclinic polycycle $\partial \mathbb{S}_+^2$. The phase portrait is shown on the right of Figure \ref{global dynamics}. (i).

    Further, we can characterize the omega set $\omega_d(x)$ of any $x\in \mathbb{R}_+^3$ as follows.
    $\omega_d(x)=\Gamma(h)$ if $x\in \Lambda_1(h)$ for any $h\in (h^*, \infty)$;  $\omega_d(x)=\{Q^*\}$ if $x\in \mathcal{L}(Q^*)$;
    $\omega_d(x)\in \{\textbf{e}_1, \textbf{e}_2, \textbf{e}_3\}$ if $x\in \partial \mathbb{R}^3_+\setminus\{O\}$.
   The corresponding phase portrait is shown on the left of Figure \ref{global dynamics} (i).

\item [(ii)]
  If $\alpha+d_1<0, \alpha+d_2>0, \alpha+d_3<0$,
  then $\textbf{e}_1$ ($\textbf{e}_3$) is a stable (unstable, resp.) node on $\mathbb{S}^2_+$,  $\textbf{e}_2$ is a saddle on $\mathbb{S}^2_+$ and the orbits from $\textbf{e}_3$ except $\Gamma_{23}$ go to $\textbf{e}_1$. The phase portrait on $\mathbb{S}^2_+$ is shown on the right of Figure \ref{global dynamics} (ii).

Further,
  $\omega_d(x)=\{ \textbf{e}_1\}$ if $x\in \textrm{Int}\mathbb{R}^3_+$; $\omega_d(x)\in \{\textbf{e}_1, \textbf{e}_2,\textbf{e}_3 \}$ for any $x\in \partial \mathbb{R}^3_+\setminus \{O\}$. The corresponding phase portrait is shown on the left of Figure \ref{global dynamics} (ii).

 \item [(iii.a)]
    If $\alpha+d_1=0, \alpha+d_2>0, \alpha+d_3>0$, then $\mathbb{S}^2_+$ consists of infinitely many heteroclinic orbits on $\textrm{Int}\mathbb{S}^2_+$, infinitely many equilibria filled with $\Gamma_{12}$ and boundary heteroclinic orbits on $\partial \mathbb{S}_+^2$. The phase portrait is shown on the right of Figure \ref{global dynamics} (iii.a).

  Further, $\omega_d(x)$ is one of equilibria on $\Gamma_{12}^+$ if $x\in \textrm{Int}\mathbb{R}^3_+$; $\omega_d(x)\in \{\textbf{e}_3, Q\in \Gamma_{12}\}$ if $x\in \partial \mathbb{R}^3_+\setminus\{O\}$.
 The corresponding phase portrait is shown on the left of Figure \ref{global dynamics} (iii.a).

\item [(iii.b)]
If $\alpha+d_1>0, \alpha+d_2=0, \alpha+d_3<0$,
      then $\textbf{e}_2$ is a stable node on $\mathbb{S}^2_+$ which attracts all orbits except $\Gamma_{13}$. The phase portrait on $\mathbb{S}^2_+$ is shown on the right of Figure \ref{global dynamics} (iii.b)

Moreover, $\omega_d(x)= \{\textbf{e}_2\}$ if $x\in \textrm{Int}\mathbb{R}^3_+$; $\omega_d(x)\in \{\textbf{e}_2, Q\in \Gamma_{13}\}$ if $x\in \partial \mathbb{R}^3_+\setminus \{O\}$.
The corresponding phase portrait is shown on the left of Figure \ref{global dynamics} (iii.b).
  \item [(iv)]
   If $\alpha+d_1=0, \alpha+d_2=0, \alpha+d_3<0$,
  then $\mathbb{S}^2_+$ consists of heteroclinic orbits on $\textrm{Int}\mathbb{S}^2_+$, infinitely many equilibria filled with $\Gamma_{12}$ and $\Gamma_{13}$, and a boundary heteroclinic orbit with endpoints $\textbf{e}_3$ and $\textbf{e}_2$.
  The phase portrait is shown on the right of Figure \ref{global dynamics} (iv)

 Moreover, $\omega_d(x)$ is one of equilibria on $\Gamma_{12}$ if $x\in \textrm{Int}\mathbb{R}^3_+$; $\omega_d(x)\in \{Q: Q\in \Gamma_{12}\bigcup \Gamma_{13}\}$
 if $x\in \partial \mathbb{R}^3_+\setminus \{O\}$.
The corresponding phase portrait is shown on the left of Figure \ref{global dynamics} (iv).
  \item [(v)]
   If $\alpha+d_i=0$ for all $i\in \{1, 2, 3\}$,
  then $\mathbb{S}^2_+$ consists of equilibria. The phase portrait is shown on the right of Figure \ref{global dynamics} (v).

 Moreover, $\omega_d(x)=\{Q_x\}$ if $x\in \mathcal{L}(x)$, where $Q_x:=\mathcal{L}(x)\bigcap \mathbb{S}^2_+$. The corresponding phase portrait is shown on the left of Figure \ref{global dynamics} (v).
\end{itemize}
\end{theorem}

\begin{figure}[H]
\centering
\begin{subfigure}{\textwidth}
\centering
\includegraphics[width=8.3cm, height=3.5cm]{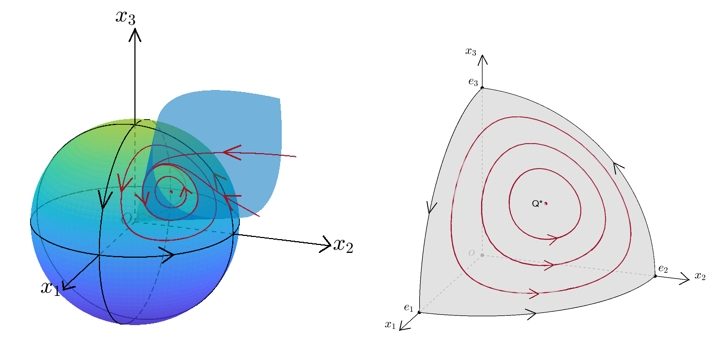}

(i) $\alpha+d_1>0, \alpha+d_2>0, \alpha+d_3>0$
\end{subfigure}
\end{figure}

\begin{figure}[H]
\begin{subfigure}{\textwidth}
\centering
  \includegraphics[width=8.3cm, height=3.5cm]{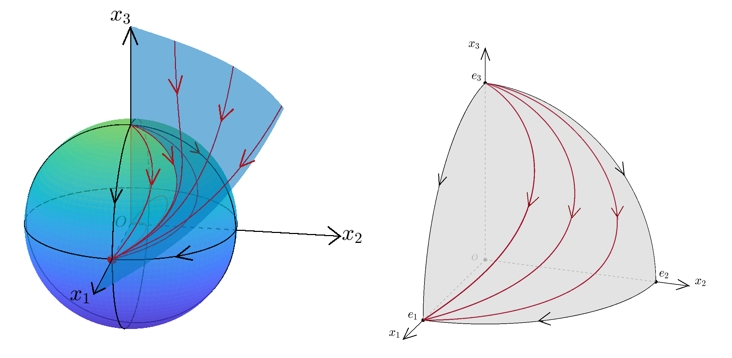}

(ii) $\alpha+d_1<0, \alpha+d_2>0, \alpha+d_3<0$
\end{subfigure}

\begin{subfigure}{\textwidth}
\centering
\includegraphics[width=8.5cm, height=4cm]{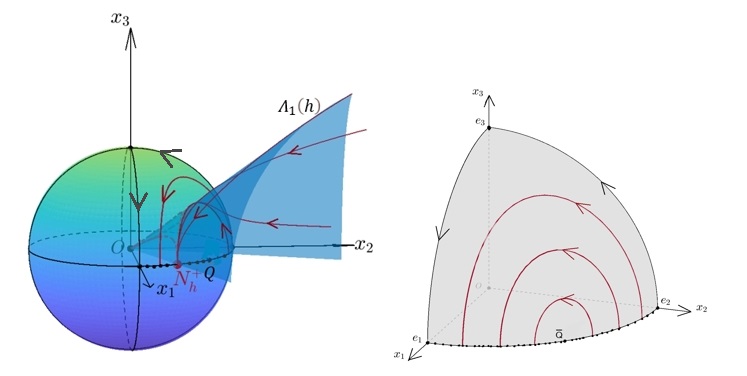}

(iii.a) $\alpha+d_1=0, \alpha+d_2>0, \alpha+d_3>0$
\end{subfigure}

\begin{subfigure}{\textwidth}
\centering
  \includegraphics[width=8.5cm, height=3.5cm]{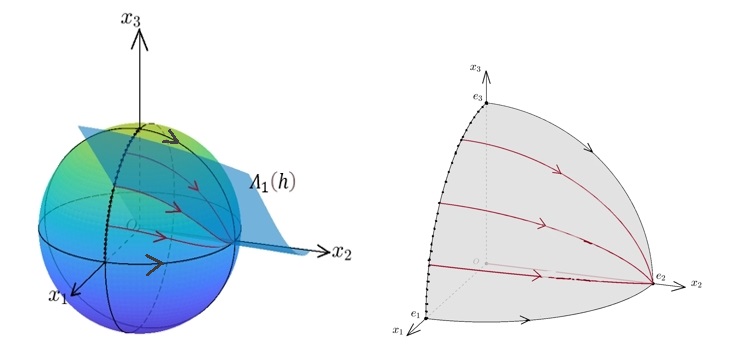}

 (iii.b)  $\alpha+d_1>0, \alpha+d_2=0, \alpha+d_3<0$
  \end{subfigure}

  \begin{subfigure}{\textwidth}
\centering
  \includegraphics[width=8.5cm, height=3.5cm]{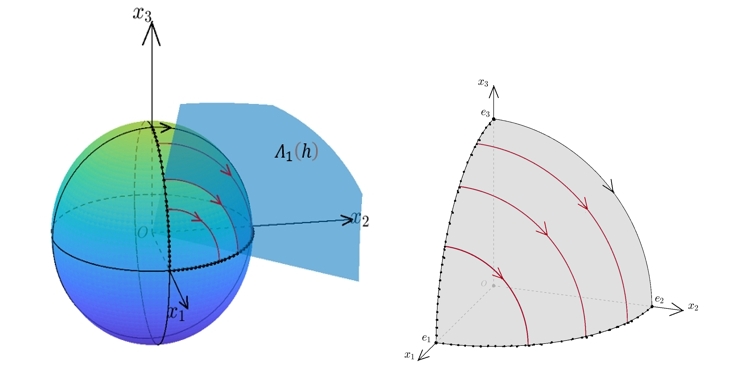}

(iv) $\alpha+d_1=0, \alpha+d_2=0, \alpha+d_3<0$
\end{subfigure}

   \begin{subfigure}{\textwidth}
\centering
\includegraphics[width=8.3cm, height=3.5cm]{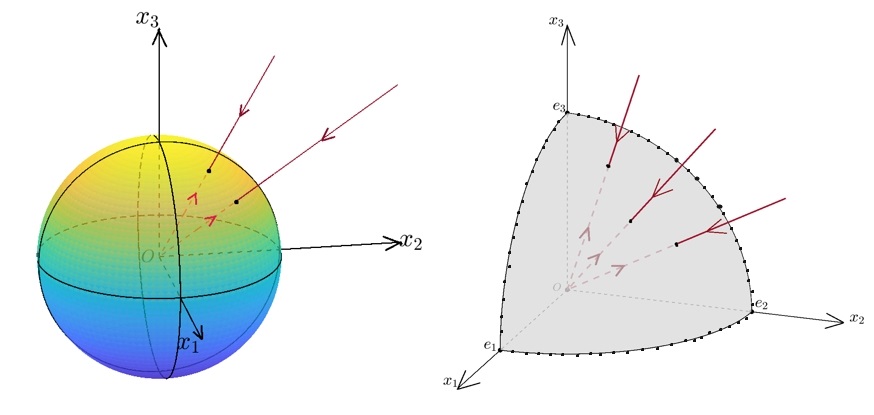}

(v) $\alpha+d_1=0, \alpha+d_2=0, \alpha+d_3=0$
\end{subfigure}

  \caption{Global dynamics of system \eqref{sys6} in $\mathbb{R}^3_+$ and $\mathbb{S}^2_+$}
\label{global dynamics}
 \end{figure}

\begin{proof}
 We first claim that $\omega_d(x)\subseteq \{\Gamma_{12}, \Gamma_{13}, \Gamma_{23}\}$ for any $x\in \partial\mathbb{R}^3_+\setminus \{O\}$. In fact,  $\omega_d(x)\subseteq \mathbb{S}^2_+$ for any $x\in \mathbb{R}^3_+\setminus \{O\}$ by Lemma \ref{dglobal}. Note that
 system \eqref{sys6} has three invariant planes $\{x\in \mathbb{R}^3:\ x_i=0\}$ with $i=1,2,3$.  Thus,
  $$\omega_d(x)\subseteq \mathbb{S}^2_+\cap \left(\cup_{i=1}^3\{x\in \mathbb{R}_+^3:\ x_i=0\}\right)=\{\Gamma_{12}, \Gamma_{13}, \Gamma_{23}\}$$ for any $x\in \partial\mathbb{R}^3_+\setminus \{O\}$. Moreover, we prove that $\omega_d(x)$ is one of equilibria on $\{\Gamma_{12}, \Gamma_{13}, \Gamma_{23}\}$ for any $x\in \partial\mathbb{R}^3_+\setminus \{O\}$. More precisely, if $x\in \{x\in \mathbb{R}_+^3:\ x_i=0\}$, we verify that $\omega_d(x)$ is one of equilibria on $\Gamma_{kl}$, $k,l\in\{1,2,3\}\setminus\{i\}$ and $k<l$. Due to the similar method, we only  verify that $\omega_d(x)$ is one of equilibria on $\Gamma_{12}$ if $x\in \{x\in \mathbb{R}_+^3:\ x_3=0\}$.

  On the invariant plane $\{x\in \mathbb{R}^3:\ x_3=0\}$,   system \eqref{sys6} can be reduced to the following two-dimensional differential system
  \begin{equation}\label{reduce-2-dimension-plane}
\begin{cases}
\frac{dx_1}{dt}=x_1(\alpha-\alpha x_1^2-(2\alpha+d_1)x_2^2),\\
\frac{dx_2}{dt}=x_2(\alpha+d_1x_1^2-\alpha x_2^2)
\end{cases}
\end{equation}
in $\mathbb{R}_+^2$. It can be checked that system \eqref{reduce-2-dimension-plane} in $\mathbb{R}_+^2$ has only three boundary equilibria $(0,0)$, $(1,0)$ and $(0,1)$  if $\alpha+d_1\not=0$,  and there are infinitely many equilibria filled with $\Gamma_{12}$ if $\alpha+d_1=0$. When $\alpha+d_1\not=0$, system \eqref{reduce-2-dimension-plane} has a stable hyperbolic node (saddle) $(1,0)$ and a hyperbolic saddle (node) $(0,1)$ if $\alpha+d_1<0$ ($\alpha+d_1>0$, resp.). Hence, the boundary equilibrium $(1,0)$ ($(0,1)$) is a global attractor for system  \eqref{reduce-2-dimension-plane} in $\mathbb{R}_+^2\setminus\{(0,0), (0,1)\}$ ($\mathbb{R}_+^2\setminus\{(0,0), (1,0)\}$, resp.) if $\alpha+d_1<0$ ($\alpha+d_1>0$, resp.). This implies that $\omega_d(x)$ is one of equilibria on the endpoints of $\Gamma_{12}$ if $x\in \{x\in \mathbb{R}_+^3:\ x_3=0\}$ and $\alpha+d_1\not=0$. On the other hand, if $\alpha+d_1=0$, then system \eqref{reduce-2-dimension-plane} becomes
\begin{equation}\label{reduce-2-dimension-plane1}
\begin{cases}
\frac{dx_1}{dt}=x_1(\alpha-\alpha x_1^2-\alpha x_2^2),\\
\frac{dx_2}{dt}=x_2(\alpha-\alpha x_1^2-\alpha x_2^2)
\end{cases}
\end{equation}
Any a $Q\in \Gamma_{12}$ is a degenerate equilibrium with a negative eigenvalue of system \eqref{reduce-2-dimension-plane1}.  Consider the ray $\mathcal{L}(Q)$ passing through $Q$, we have that $\mathcal{L}(Q)$ is the one-dimensional stable manifold of $Q$ by computation.
Hence,  $\omega_d(x)=\{Q\}$ if $x\in \mathcal{L}(Q)$ for any a $Q\in \Gamma_{12}$. This leads that $\omega_d(x)$ is one of equilibria on $\Gamma_{12}$ if $x\in \{x\in \mathbb{R}_+^3:\ x_3=0\}$.

In the following  it is to discuss the dynamics of system \eqref{sys6} on $\textrm{Int}\mathbb{S}^2_+$ and in $\textrm{Int}\mathbb{R}^3_+$ for the case (i)-(v). We consider system \eqref{sys6} restricted to $\mathbb{S}^2_+$ and obtain the reduced two-dimensional system \eqref{reduce-2-dimension}.
On the one hand, the dynamics of system \eqref{reduce-2-dimension} can be obtained by Lemma \ref{isolated} and Lemma \ref{contiunnm}. This leads to the conclusions (i) - (v) on $\mathbb{S}^2_+$, see the right pictures in Figure \ref{global dynamics}.

On the other hand,  system \eqref{sys6} has one of the two first integrals $H_1(x_1,x_2,x_3)$ and $H_2(x_1,x_2,x_3)$ in $\textrm{Int}\mathbb{R}^3_+$ by Lemma \ref{first integral}. This yields that $\Lambda_i(h)$ defined by \eqref{cone0} is invariant for each $h\in I_i$,  $i=1,2$. Taking into account the invariance of $\mathbb{S}^2_+$,  one has that the intersection of $\Lambda_i(h)$ and $\mathbb{S}^2_+$ defined by $\Lambda_i(h)\bigcap\mathbb{S}^2_+$ is an orbit of system \eqref{sys6}. In $\Lambda_i(h)$, every points $x\in \Lambda_i(h)\setminus\{\Lambda_i(h)\bigcap\mathbb{S}^2_+\}$ will be attracted by $\Lambda_i(h)\bigcap\mathbb{S}^2_+$, that is $\omega_d(x)=\Lambda_i(h)\bigcap\mathbb{S}^2_+$ for $x\in \Lambda_i(h)$, see the left pictures in Figure \ref{global dynamics}. The proof is finish.
\end{proof}

As a consequence of Theorem \ref{gd1},
we give a decomposition of $\mathbb{R}_+^3$
according to attractive domains of orbits of system \eqref{sys6}. Recall that $\mathcal{E}$ denotes the set of all equilibria of system \eqref{sys6}, and $\mathcal{A}(\cdot)$ represents the attractive domain of an orbit.
\begin{corollary}\label{decomposition}
\begin{itemize}
  \item [(I)]
  If $\alpha+d_i>0 \ (<0)$  for $i=1,2,3$,  then
  $$\mathbb{R}_+^3=\{\bigcup_{Q\in \mathcal{E}}\mathcal{A}(Q)\}
  \cup \{\bigcup_{h\in ({h}^*, \infty) }\mathcal{A}(\Gamma(h))\}.$$
  \item [(II)]
  If either $\prod_{i=1}^3(\alpha+d_i)= 0$ or $\prod_{i=1}^3(\alpha+d_i) \neq 0$ and there exist $i\neq j$, $i,j\in \{1,2,3\}$,
  such that $(\alpha+d_i)(\alpha+d_j)<0$,
  then
  $$\mathbb{R}_+^3=\bigcup_{Q\in \mathcal{E}}\mathcal{A}(Q).$$
\end{itemize}
\end{corollary}

\subsection{Global bifurcations}  \label{Subsec-Bifur}

From Theorem \ref{gd1}, one can see that
 dynamics of system \eqref{sys6} changes significantly under the change of parameters $(\alpha+d_1,\alpha+d_2,\alpha+d_3)$ in the neighborhood of $(0,0,0)$. This implies some bifurcation phenomena occur, which is related to loss of the {\it hyperbolicity} of some
 orbits such as equilibrium, and periodic orbits of system \eqref{sys6}. In the subsection we choose $(\alpha+d_1,\alpha+d_2,\alpha+d_3)$ as bifurcation parameters.  For simplicity, let $m_i:= \alpha+d_i$,  $i=1, 2, 3$. We consider global bifurcation of system \eqref{sys6}
 when bifurcation parameters ${\bf m}:=(m_1, m_2, m_3)$ vary in the parameter space $\mathbb{R}^3$. It is clear that the origin  ${\bf 0}:=(0, 0 ,0)\in \mathbb{R}^3$ is a bifurcation value (or bifurcation point) since system \eqref{sys6} has infinitely many degenerated equilibria
 filling $\mathbb{S}_+^2$ as ${\bf m}={\bf 0}$. According to the classification of global dynamics in Theorem \ref{gd1}, we know that there are three bifurcation lines defined by
 $$l_i:=\{(m_1, m_2, m_3)\in \mathbb{R}^3: m_j=m_k=0\}, \ i\in \{1,2,3\}, j,k\in \{1,2,3\}\setminus \{i\},$$
  and  three bifurcation planes defined  by
  $$\textrm{\uppercase\expandafter{\romannumeral2}}_i:=\{(m_1, m_2, m_3)\in \mathbb{R}^3:\  m_i=0\}, \ i=1, 2, 3,
  $$
  see the colored lines and colored planes in Figure \ref{parameter}.

\begin{figure}[H]
  \centering
  
  \includegraphics[width=5cm]{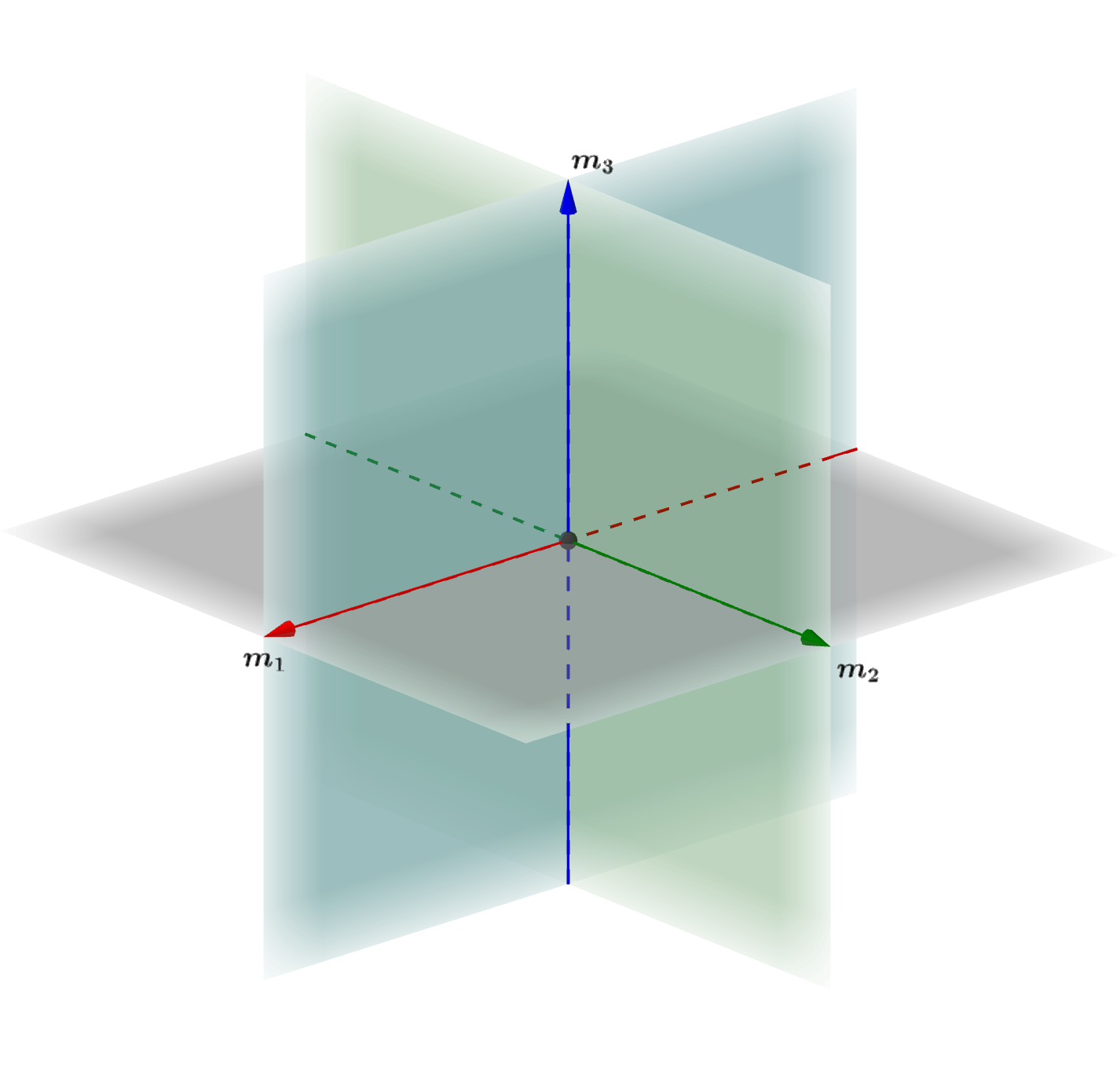}\\
  \caption{Global bifurcation diagram in parameter space $\mathbb{R}^3$}\label{parameter}
\end{figure}

 Moreover, the bifurcation point ${\bf 0}$ divides each bifurcation line $l_i, i=1, 2, 3$ into two parts as follows.
 $$l_i^+:=\{(m_1, m_2, m_3)\in l_i: m_i>0\}, \ l_i^-:=\{(m_1, m_2, m_3)\in l_i: m_i<0\}.$$
 The bifurcation lines divide each bifurcation plane \uppercase\expandafter{\romannumeral2}$_i$ into four parts denoted by \uppercase\expandafter{\romannumeral2}$_i^j$, $j=1,2,3,4$,
 and the bifurcation planes divide the parameter space $\mathbb{R}^3$ into eight parts denoted by $D^+_j$, $D^-_j$ for $j=1,..., 4$.
 Thus, the parameter space $\mathbb{R}^3$ is divided into 27 regions, that is, the point ${\bf 0}$, $l^+_i$, $l^-_i$, \uppercase\expandafter{\romannumeral2}$^j_i$, $D^+_j$, $D^-_j$ for $i=1, 2, 3$ and $j=1,2,3,4$. Due to the symmetry, we only give the bifurcation diagrams in the bifurcation line $l_2$ (see Figure \ref{line}), in the bifurcation plane $\textrm{\uppercase\expandafter{\romannumeral2}}_1$ (see Figure \ref{plane}) and in the case where  $m_3>0$ (see Figure \ref{positive}). Here
\begin{align*}
& \textrm{\uppercase\expandafter{\romannumeral2}}_1^1:=\{m_1=0, m_2>0, m_3>0\}; \ \ \textrm{\uppercase\expandafter{\romannumeral2}}_1^2:=\{m_1=0, m_2<0, m_3>0\}; \\
&  \textrm{\uppercase\expandafter{\romannumeral2}}_1^3:=\{m_1=0, m_2<0, m_3<0\};\ \  \textrm{\uppercase\expandafter{\romannumeral2}}_1^4:=\{m_1=0, m_2>0, m_3<0\}; \\
& \textrm{\uppercase\expandafter{\romannumeral2}}_2^1:=\{m_1>0, m_2=0, m_3>0\}; \ \ \textrm{\uppercase\expandafter{\romannumeral2}}_2^2:=\{m_1<0, m_2=0,m_3>0\}; \\
& D^+_1:=\{m_1>0, m_2>0, m_3>0\}; \ \ D^+_2:=\{m_1<0, m_2>0, m_3>0\}; \\
& D^+_3:=\{m_1<0, m_2<0, m_3>0\}; \ \ D^+_4:=\{m_1>0, m_2<0, m_3>0\};
\end{align*}

\begin{figure}[H]\label{bifurcation}
  \centering
  \includegraphics[width=10cm, height=5cm]{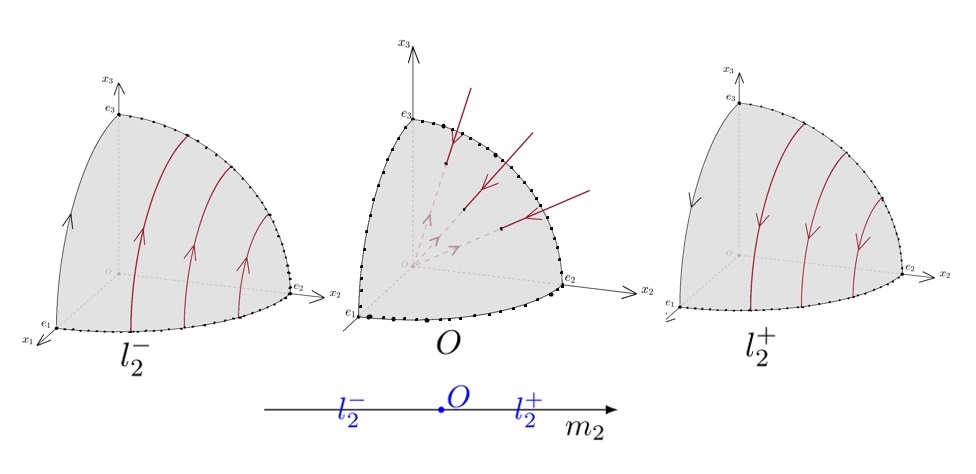}\\
 \caption{Bifurcation diagrams and phase portraits when $m_1=m_3=0$ .}
 \label{line}
\end{figure}

\begin{figure}[H]
  \centering
  \includegraphics[width=11cm, height=8cm]{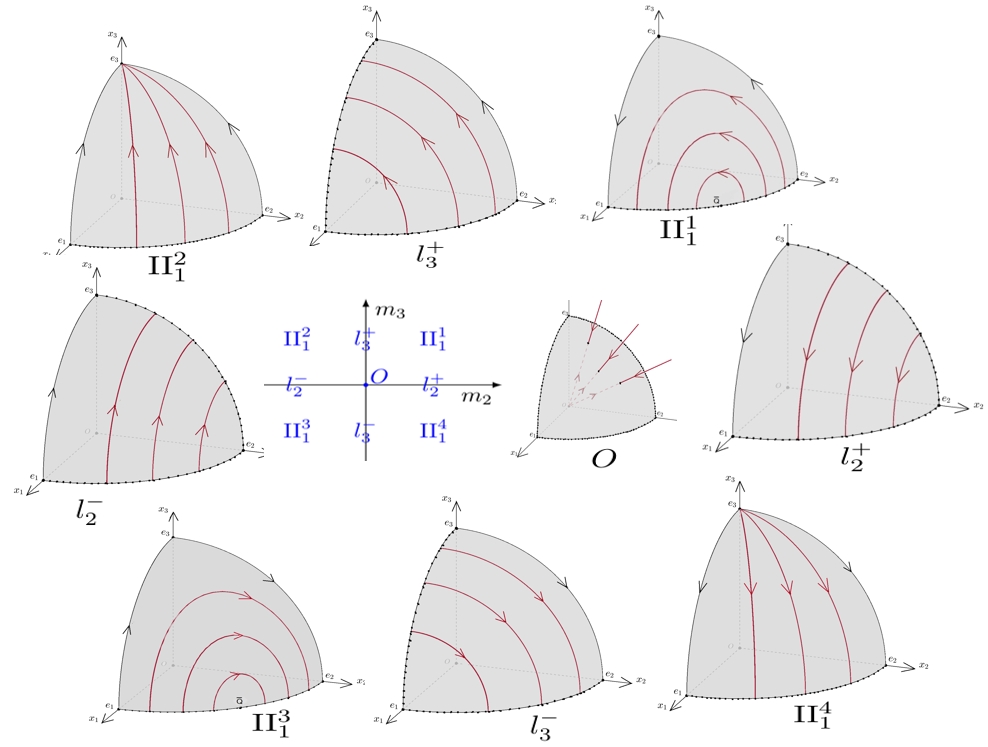}\\
 \caption{Bifurcation diagrams and corrsponding phase portraits in plane $m_1=0$.}
 \label{plane}
\end{figure}

\begin{figure}[H]
  \centering
  \includegraphics[width=11cm, height=8cm]{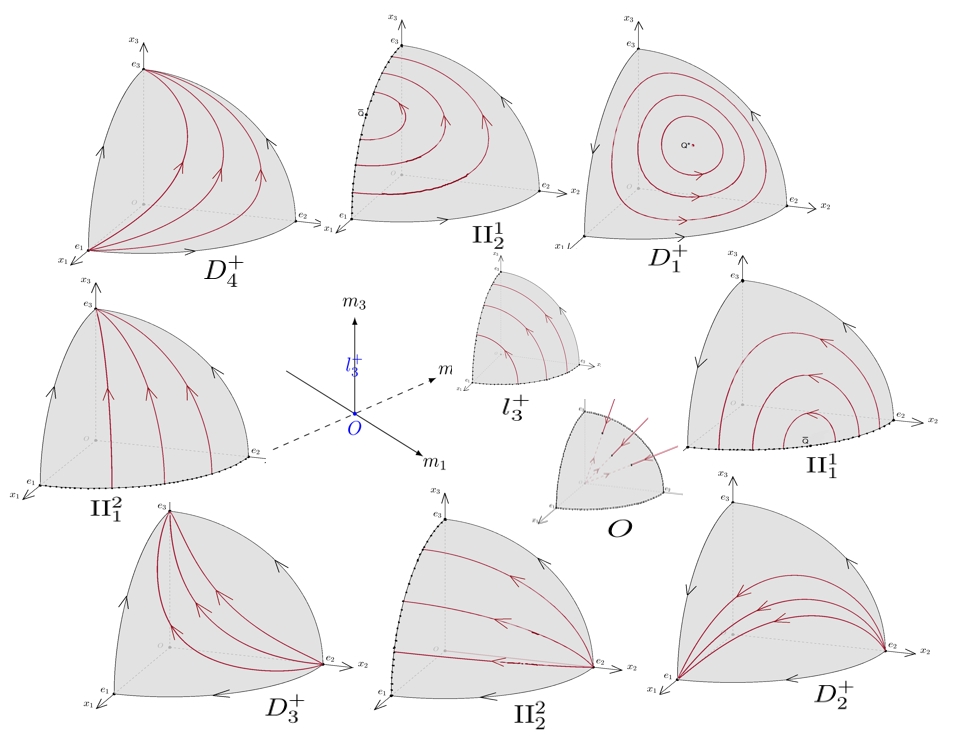}\\
 \caption{Bifurcation diagrams and the corresponding phase portraits in $m_3>0$.}
 \label{positive}
 \end{figure}

\section{Stochastic decomposition formula}   \label{Sec-Stoch-form}
This section contains the key stochastic decomposition formula 
connecting deterministic and stochastic Kolmogorov systems.
One important object here is the stochastic logistic-type equation
(see \eqref{sysg} below).
Several useful dynamical properties of logistic-type equations
are studied in Subsection \ref{logistic-type}.

\subsection{General case}    \label{Subsec-Stoch-general}
Consider more general stochastic Kolmogorov system with identical intrinsic growth rate in $\mathbb{R}^n$
\begin{equation}\label{ck2}
\begin{cases}
d{x_1}=x_1(\alpha+f_1(x_1, x_2, \dots , x_n))dt+\sigma x_1dW_t,\\
d{x_2}=x_2(\alpha+f_2(x_1, x_2,\dots, x_n))dt+\sigma x_2dW_t,\\
\dots\\
d{x_n}=x_n(\alpha+f_n(x_1, x_2,\dots, x_n))dt+\sigma x_ndW_t.
\end{cases}
\end{equation}
Here, $(x_1, x_2,\dots, x_n)\in \mathbb{R}^n$, $\alpha, \sigma \in \mathbb{R}$,
and  $\{f_i\}$ are homogeneous polynomials in $\mathbb{R}[x]$ with degree $m \in [1,\infty)$
of the form
\begin{align*}
   f_i(x)=\sum_{k_1+\dots+k_n=m}a^{(i)}_{k_1, k_2,\dots, k_n}x_1^{k_1}x_2^{k_2}\cdots
   x_n^{k_n}, \ \ i=1,\dots,n,
\end{align*}
where $0\leq k_i\leq m$. In particular,
when $\sigma=0$,
we have the deterministic Kolmogorov system
\begin{equation}\label{ck1}
\begin{cases}
d{x_1}=x_1(\alpha+f_1(x_1, x_2,\dots, x_n))dt,\\
d{x_2}=x_2(\alpha+f_2(x_1, x_2,\dots, x_n))dt,\\
\cdots\\
d{x_n}=x_n(\alpha+f_n(x_1, x_2,\dots, x_n))dt.
\end{cases}
\end{equation}

Theorem \ref{sc} below presents the key stochastic decomposition formula, which is a general form of the formula first proposed by  Chen et al in \cite{jiang1},
\begin{theorem} (Stochastic decomposition formula)  \label{sc}
Let $\Phi=\Phi(t,\omega, x)$ and $\Psi= \Psi(t,x)$
be the solutions to \eqref{ck2} and \eqref{ck1}, respectively,
with the initial value $x\in \mathbb{R}^n$.
Then, we have
\begin{equation}\label{sdf1}
\Phi(t,\omega, x)=g(t, w, g_0)\Psi(\int_0^tg^m(s,\omega, g_0)ds, \frac x{g_0}),\; x\in \mathbb{R}^n,
\end{equation}
where $g=g(t, \omega, g_0)$ is the positive solution of the following stochastic logistic-type equation
\begin{align*}
   dg=g(\alpha-\alpha g^m)dt+\sigma gdW_t
\end{align*}
with the initial value $g(0, \omega, g_0)=g_0>0.$
\end{theorem}

\begin{proof}
Let ${\Phi}_i(t, \omega, x)$ denote the $i$-th component
of the right-hand side of \eqref{sdf1},
$1\leq i\leq n$.
Applying It\^o's formula we derive
\begin{align*}
d{\Phi}_i =&(g(\alpha-\alpha g^m)dt+\sigma gdW_t)  \Psi_i(\int_0^tg^mds,\frac x{g_0})   \\
&+g^{m+1}\Psi_i(\int_0^tg^mds,\frac x{g_0})\times \\
& \ \ \ \big( \alpha+\sum a^{(i)}_{k_1,...,k_m}\Psi_1^{k_1}(\int_0^tg^mds, \frac x{g_0})\cdot\cdot\cdot\Psi_n^{k_n}(\int_0^tg^mds, \frac x{g_0}) \big) dt\\
=& {\Phi}_i(\alpha-\alpha g^m)dt+\sigma \Phi_i dW_t
    + {\Phi}_ig^m(\alpha+\sum a^{(i)}_{k_1,...,k_m}\Psi_1^{k_1}  \cdot\cdot\cdot\Psi_n^{k_n})dt\\
=& {\Phi}_i(\alpha+g^m\sum a^{(i)}_{k_1,...,k_m}\Psi_1^{k_1}\cdot\cdot\cdot\Psi_n^{k_n})dt +\sigma  {\Phi}_idW_t\\
=& {\Phi}_i(\alpha+\sum a^{(i)}_{k_1,...,k_m}\Phi_1^{k_1} \cdot\cdot\cdot\Phi_n^{k_n})dt+\sigma {\Phi}_idW_t.
\end{align*}
This yields that $\Phi=(\Phi_1, \cdots, \Phi_n)$ satisfies the system \eqref{ck2}.
Thus, in view of the uniqueness of solutions,
we obtain \eqref{sdf1} and finish the proof.
\end{proof}

Now, we come back to our specific stochastic Kolmogorov system  \eqref{sys5}.
Proposition \ref{RDS} below gives the global unique existence of solutions to  \eqref{sys5} for almost every sample path,
which guarantee the solutions do not blow up in forward time.

\begin{proposition}[Generation of a random dynamical system]\label{RDS}
Let $\alpha>0$. Then for any $x\in \mathbb{R}^3$ and almost every $\omega\in \Omega$, there exists a global unique solution $\Phi(\cdot, \omega, x)$ to \eqref{sys5}
with the initial condition $x$ such that $\Phi$ forms a $C^1$ random dynamical system on $(\Omega, \mathcal{F}, \mathbb{P}, (\theta_t)_{t\in \mathbb{R}})$ with independent increments.
\end{proposition}

\begin{proof}
Define the Lyapunov function $V: \mathbb{R}^3\rightarrow \mathbb{R}_+$ by
\begin{equation}\label{V}
V(x):=\|x\|^2 = x_1^2+x_2^2+x_3^2,\ \  (x_1, x_2, x_3)\in \mathbb{R}^3,
\end{equation}
and the operator $\mathscr{L}^\sigma$ by
\begin{equation}\label{Fop}
\mathscr{L}^{\sigma}f(x):= \langle \nabla f(x), b(x) \rangle + \frac12a^{ij}\partial^2_{ij}f(x),\ \  f\in C^2(\mathbb{R}^3).
\end{equation}
Then, by a straightforward computation,
\begin{equation*}\label{FPO}
\begin{aligned}
\mathscr{L}^\sigma V(x) &=\langle \nabla V(x), b(x) \rangle + \frac12 a^{ij}\partial^2_{ij}V(x)\\
& = -2\alpha( x_1^2+x_2^2+ x_3^2)(x_1^2+x_2^2+x_3^2-1)+\sigma^2(x_1^2+x_2^2+x_3^2) \\
&= \|x\|^2(-2\alpha\|x\|^2+2\alpha+\sigma^2)\\
& \leq (2\alpha+\sigma^2)V(x).
\end{aligned}
\end{equation*}
Then, by \cite[Theorem 3.3.5]{krz}, we get the global unique existence of solutions to \eqref{sys5}. Naturally, the solutions generates a $C^1$ RDS with independent increments.
\end{proof}

As a consequence of Theorem \ref{sc},
we have the following
stochastic decomposition formula
for system \eqref{sys5},
corresponding
to the case $n=3$ and $m=2$.

\begin{corollary}
Let $\Psi= \Psi(t,x)$ and $\Phi=\Phi(t,\omega, x)$ be the solutions to
\eqref{sys6} and \eqref{sys5}, respectively,
with the initial value $x\in \mathbb{R}^3$.
Then, we have
\begin{equation}\label{sdf}
\Phi(t,\omega, x)=g(t, w, g_0)\Psi(\int_0^tg^2(s,\omega, g_0)ds, \frac x{g_0}),\ \
x\in \mathbb{R}^3,\; g_0> 0,
\end{equation}
where $g(t,\omega, g_0)$ is the positive solution to the stochastic logistic-type equation
\begin{equation}\label{sysg}
dg=g(\alpha-\alpha g^2)dt+\sigma gdW_t
\end{equation}
with the initial value $g_0\in \mathbb{R}_+$.
\end{corollary}

In the next subsection,
we collect several dynamical properties of
stochastic logistic-type equations.

\subsection{Stochastic logistic-type equations}   \label{logistic-type}
Let $\alpha>0$.
The stochastic logistic-type equation \eqref{sysg}
has the explicit expression of solutions
\begin{equation}\label{solug}
g(t,\omega, x)=\frac{x\exp\{(\alpha-\frac12\sigma^2) t
   +\sigma W_t(\omega)\}}{(1+2\alpha x^2\int_0^t\exp\{2((\alpha-\frac12\sigma^2) s +\sigma W_{s}(\omega))\}ds)^{\frac12} },
\end{equation}
for $ x\neq 0$ and $g(t, \omega, 0)=0$.

One has the following characterization of random
equilibria for logistic-type equations, 
which 
follows essentially from \cite[Subsection 2.3.7, Subsection 9.3.2]{al}.

\begin{lemma}(Random equilibria) \label{lem1}
\begin{itemize}
   \item[(i)] Let $\sigma^2\ge2\alpha$.
   Then, the zero point is the unique random equilibrium of \eqref{sysg} in $\mathbb{R}^+$.
   Moreover, for all $x>0$ and $\mathbb{P}$-a.e. $\omega\in \Omega$,
   \begin{align}  \label{g-thetat-0}
      g(t, \theta_{-t}\omega, x)\rightarrow 0,\ \ as\ t\rightarrow \infty.
   \end{align}

  \item[(ii)] Let $\sigma^2<2\alpha$.
  Then,
  there exist two random equilibria in $\mathbb{R}_+$,
  i.e., the zero point and
\begin{equation}\label{randomequilibrium}
  u_g(\omega)=(2\alpha\int_{-\infty}^0 \exp\{2((\alpha-\frac12\sigma^2) s +\sigma W_{s}(\omega))\}ds)^{-\frac12}.
\end{equation}
  Moreover, every positive pull-back trajectory for equation \eqref{sysg}
  converges exponentially to the random equilibrium $u_g(\omega)$,
  that is, there exists $\lambda>0$ such that  for  all $x>0$ and $\omega\in \Omega$,
      \begin{equation}\label{asy1d}
      \lim_{t\rightarrow +\infty}e^{\lambda t}|g(t, \theta_{-t}\omega, x)-u_g(\omega)|=0.
      \end{equation}
\end{itemize}
\end{lemma}

The next lemma describes the stationary measure and related density functions of Markov semigroup corresponding to SDE \eqref{sysg}. The results except \eqref{g1}
follow from \cite[Subsection 2.3.7, Subsection 9.3.2]{al}, 
and \eqref{g1} 
can be proved by \eqref{asy1d}.

\begin{lemma}\label{Lem-Erg-Stologequa}
Let $\sigma^2<2\alpha$ and
   \begin{equation}\label{1dstationary}
   \mu_g^{\sigma}:=\mathbb{P}\circ u_g^{-1}.
   \end{equation}
  Then, $\mu_g^\sigma$ is a stationary measure for the associated Markovian semigroup. Moreover, the associated the density function is
   \begin{equation}\label{den}
    p_{\alpha}^{\sigma}(x)=C_{\alpha}x^{\frac{2 \alpha}{\sigma^2}-2}\exp(-\frac{\alpha}{\sigma^2}x^2),\ \  x>0,
\end{equation}
  where $C_{\alpha}=2(\frac{\alpha}{\sigma^2})^{\frac{\alpha}{\sigma^2}-1}(\Gamma(\frac{\alpha}{\sigma^2}-\frac12))^{-1}$,
  and
  \begin{equation}\label{g1}
\displaystyle \lim_{t\rightarrow \infty} \mathbb{P} (g(t, \cdot, x)\in A)=\mu_g^{\sigma}(A),
\ \ \forall x>0,\ A \in \mathcal{B}(\mathbb{R}_+).
\end{equation}
\end{lemma}

\begin{lemma}\label{time1}
Let $\sigma^2<2\alpha$.
Then there exists a $\theta$-invariant set $\Omega^*$ of
full measure such that for all $x>0$ and $\omega\in \Omega^*$,
\begin{align}
\displaystyle \lim_{t\rightarrow \infty} \frac1t\int_0^tg^2(s,\omega,x)ds
  =\frac{1}{\alpha}(\alpha-\frac12 \sigma^2),  \label{b2.1} \\
\displaystyle \lim_{t\rightarrow \infty} \frac1t\int_0^tg^2(s,\theta_{-t}\omega,x)ds
  =\frac{1}{\alpha}(\alpha-\frac12 \sigma^2).    \label{b2.2}
\end{align}
In particular,
\begin{align} \label{int-g2-infty}
    \displaystyle \lim_{t\rightarrow \infty}
   \int_0^tg^2(s,\omega,x)ds
   =  \displaystyle \lim_{t\rightarrow \infty}
   \int_0^tg^2(s,\theta_{-t}\omega,x)ds
   =\infty.
\end{align}
\end{lemma}

\begin{proof}[Proof of Lemma \ref{time1}]
Let $\Omega^*:= \{\omega\in \Omega: \lim_{\mid t \mid\rightarrow \infty} {W_t(\omega)}/t=0\}$.
Then, by the iterated logarithmic law of Brownian motion (cf. \cite{kis}),
$\mathbb{P}(\Omega^*)=1$.
Moreover, for any $\omega\in \Omega^*$ and $s\in \mathbb{R}$,
$\lim_{t \rightarrow \infty}  {W_t(\theta_s\omega)}/t
 =\lim_{t \rightarrow \infty}  ({W_{t+s}(\omega)-W_s(\omega)})/t=0,$
which yields that $\theta_s \Omega^* \subseteq \Omega^*$,
and so $\Omega^*$ is a $\theta$-invariant set.

Now, for any $x>0$ and $\omega\in \Omega^*$,
using \eqref{solug} we compute
\begin{align*}
&\displaystyle \lim_{t\rightarrow \infty}\frac1t\int_0^tg^2(s, \omega, x)ds  \\
=& \displaystyle \lim_{t\rightarrow \infty}\frac1t\int_0^t \frac{x^2\exp{2\{(\alpha-\frac12\sigma^2) s
+\sigma W_s(\omega)\}}}{1+2\alpha x^2\int_0^s\exp\{2((\alpha-\frac12\sigma^2) r +\sigma W_{r}(\omega))\}dr}ds\\
=& \displaystyle \lim_{t\rightarrow \infty} \frac1{2\alpha t}
    \ln\big(1+2\alpha x^2\int_0^t\exp\{2((\alpha-\frac12\sigma^2)r +\sigma W_{r}(\omega))\}dr \big)\\
=& \displaystyle \lim_{t\rightarrow \infty}\frac1{2\alpha t}
     \ln\int_0^t\exp(2(\alpha-\frac12\sigma^2) r+2\sigma W_{r}(\omega))dr\\
=& \frac{1}{\alpha}(\alpha-\frac12 \sigma^2),
\end{align*}
which yields \eqref{b2.1}.
Similarly, one has
\begin{align*}
 & \displaystyle \lim_{t\rightarrow \infty}\frac1t\int_0^tg^2(s, \theta_{-t}(\omega), x)ds  \\
=& \displaystyle \lim_{t\rightarrow \infty}\frac1t\int_0^t
       \frac{x^2\exp{2\{(\alpha-\frac12\sigma^2) s+\sigma W_s(\theta_{-t}(\omega))\}}}
       {1+2\alpha x^2\int_0^s\exp\{2((\alpha-\frac12\sigma^2)r+\sigma W_{r}(\theta_{-t}(\omega)))\}dr }ds\\
=& \displaystyle \lim_{t\rightarrow \infty}\frac1{2\alpha t}
     \ln\big( 1+2\alpha x^2\int_0^t\exp2\{(\alpha-\frac12\sigma^2) r +\sigma W_{r}(\theta_{-t}(\omega))\}dr \big) \\
=& \displaystyle \lim_{t\rightarrow \infty}\frac1{2\alpha t}\ln\int_0^t\exp(2(\alpha-\frac12\sigma^2) r+2\sigma W_r(\theta_{-t}(\omega)))dr\\
=& \frac{1}{\alpha}(\alpha-\frac12 \sigma^2).
\end{align*}
Thus, \eqref{b2.2} is verified.
\end{proof}

\section{Pull-back $\Omega$-limit sets} \label{pull-back limit set}
Since this section,
we start to study the pull-back $\Omega$-limit sets for the stochastic Kolmogorov system \eqref{sys5} in $\mathbb{R}^3_+$.

The main result of this section is stated in Theorem \ref{attract} below,
which describes the pull-back $\Omega$-limit sets for system \eqref{sys5}.

\begin{theorem}\label{attract} (Pull-back $\Omega$-limit sets)
For every $x\in \mathbb{R}^3_+$ and almost surely $\omega\in \Omega$,
the following holds:
\begin{itemize}
\item [(i)]
  If $\sigma^2\ge 2\alpha$,
  then for any $x\in \mathbb{R}_+^3$,
  $\Omega_x(\omega)=\{O\}$,
  i.e., the origin $O$ is the unique global attractor under pull-back sense in $\mathbb{R}_+^3$.

  \item [(ii)]
  If $ \sigma^2<2\alpha$, then
  for any $x\in \mathbb{R}_+^3$,
  $\Omega_x(\omega)=u_g(\omega)\omega_d(x)$,
  where $u_g$ is the random equilibrium given by \eqref{omega}
  for the logistic-type equation \eqref{sysg},
  and
  $\omega_d(x)$ is the $\omega$-limit set
  given by \eqref{omega} for system \eqref{sys6}.
    More precisely,  we have
    \begin{align}  \label{Omegax-ugQ}
         \Omega_x(\omega) = \{u_g(\omega) Q\}
   \end{align}
  if $x$ lies in the attracting domain of some equilibria $Q$,
  and
   \begin{align} \label{Omegax-ugFh}
          \Omega_x(\omega) = u_g(\omega) \Gamma(h)
   \end{align}
   if $x$ lies in the attracting domain of some non-trivial periodic orbit $\Gamma(h)$.
\end{itemize}
\end{theorem}

\begin{proof}
$(i)$ For $x\in \mathbb{R}^3_+\setminus \{O\}$,
by the stochastic decomposition formula \eqref{sdf},
\begin{equation}\label{sss}
\Phi(t, \theta_{-t}\omega, x)= g(t, \theta_{-t}\omega, 1)\Psi(\int_0^tg^2(s, \theta_{-t}\omega, 1)ds,x).
\end{equation}

Note that for $t$ sufficiently large,
$\|\Psi(\int_0^t g^2(s, \theta_{-t}\omega, 1)ds, x)\|$ is uniformly bounded,
that is, there exists $M\in (0,\infty)$
such that \\
 $$\|\Psi(\int_0^t g^2(s, \theta_{-t}\omega, 1)ds, x)\|\leq M,\, \,  \textrm{for}\,\, t\,\, \textrm{large}\,\, \textrm{enough}.$$
Actually, this is clear if
$\int_0^tg^2(s, \theta_{-t}\omega,1)ds$ is uniformly bounded in $t$, because $\Psi(t, x)$ is continuous in $t$.
Otherwise,
in view of Lemma \ref{dglobal},
$\Psi(\cdot, x)$ is close to $\mathbb{S}_+^2$
for $t$ large enough,
which also implies the uniform boundedness.

Thus, taking into account the decay  \eqref{g-thetat-0} and \eqref{sss} we obtain
$$\displaystyle \lim_{t\rightarrow \infty}\Phi(t, \theta_{-t}\omega, x)=O,\ \  a.s. $$

$(ii)$ We first infer from Corollary \ref{decomposition} that
$$x\in \mathcal{A}(Q)\,\, \textrm{for}\,\, \textrm{some} \,\, Q\in \mathcal{E},\,\, \textrm{or}\,\, x\in \mathcal{A}(\Gamma(h))\,\,
\textrm{for} \,\, \textrm{some}\,\, h\in ({h}^*, \infty).$$

In the first case where $x\in \mathcal{A}(Q)$,
we have $\Psi(t, x)\rightarrow Q$ as $t\rightarrow\infty$,
and so $\omega_d(x)=\{Q\}$. But by \eqref{int-g2-infty}
\begin{equation*}\label{timeinfinity}
\int_0^tg^2(s, \theta_{-t}\omega, 1)ds\rightarrow \infty,\ \ \textrm{as}\,\, t\rightarrow\infty,
\end{equation*}
which implies that
$$\Psi(\int_0^tg^2(s, \theta_{-t}\omega, 1)ds,x)\rightarrow Q,\ \  \textrm{as}\,\, t\rightarrow\infty.$$
Since $g(t, \theta_{-t}\omega, 1)\rightarrow u_g(\omega)$
due to  \eqref{asy1d},
we infer from \eqref{sss} that
$$\Phi(t, \theta_{-t}\omega, x)\rightarrow u_g(\omega)Q,$$
and so \eqref{Omegax-ugQ} follows.

In the second case where $x\in \mathcal{A}(\Gamma(h))$,
we have $u_g(\omega)\Gamma(h)\subseteq \Omega_x(\omega)$.
Actually, for any $y\in \Gamma(h)$,
by \eqref{int-g2-infty}, there exists a sequence $\{t_n\}$ such that
\begin{equation}\label{pull1}
\displaystyle \lim_{n\rightarrow \infty}\Psi(\int_0^{t_n}g^2(s, \theta_{-t_n}\omega, 1)ds, x)=y.
\end{equation}
Then, by \eqref{sdf} and \eqref{asy1d},
 \begin{equation}\label{4.2}
\displaystyle \lim_{n\rightarrow \infty}\Phi(t_n, \theta_{-t_n}\omega, x)=u_g(\omega) y,
\end{equation}
which yields that $u_g(\omega) y \in \Omega_x(\omega)$
for $y\in \Gamma(h)$,
and so $u_g(\omega)\Gamma(h)\subseteq \Omega_x(\omega)$.

Regarding the inverse conclusion $\Omega_x(\omega)\subseteq u_g(\omega)\Gamma(h)$,
for any $z\in \Omega_x(\omega)$,
we infer that there exists a sequence $\{t_n\}$ such that
$$\displaystyle \lim_{n\rightarrow \infty}\Phi(t_n, \theta_{-t_n}\omega, x)=z.$$
Since $x\in \mathcal{A}(\Gamma(h))$,
by \eqref{int-g2-infty},
the $\omega$-limit set of $\{\Psi(\int_0^{t_n}g^2(s,\theta_{-t_n}\omega, 1)ds, x)\}$ is contained in $\Gamma(h)$.
Using the stochastic decomposition formula \eqref{sdf}  again  and \eqref{asy1d} we obtain
$z\in u_g(\omega)\Gamma(h)$ for any $z\in \Omega_x(\omega)$,
and so $\Omega_x(\omega)\subseteq u_g(\omega)\Gamma(h)$.
Therefore, we conclude that
\eqref{Omegax-ugFh} holds and finish the proof.
\end{proof}

\section{Ergodic stationary measures} \label{Sec-Ergo-Meas}

This section studies ergodic stationary measures
of the Markov semigroup $(P_t)_{t\geq 0}$
corresponding to \eqref{sys5} in $\mathbb{R}^3_+$.
There are two types of ergodic stationary measures,
related to equilibria
and invariant cones,
which are studied in Subsections \ref{Subsec-Rel-Equi}
and \ref{Subsec-Rel-Cone}, respectively.
Then,
the relationship, via the vanishing noise limit,
between ergodic stationary measures of stochastic system \eqref{sys5}
and invariant measures of deterministic system \eqref{sys6}
is proved in Subsection \ref{Subsec-Vanish}.

\subsection{Relation to equilibria} \label{Subsec-Rel-Equi}
We first consider the stationary measures related to
equilibria.
Recall that
$\mathcal{L}(y):=\{\lambda y: \lambda>0\}$ denotes the ray passing through the point $y$,
where $y\in \mathbb{R}_+^3$.

\begin{proposition} \label{thmes} (Ergodic stationary measures related to equilibria).
Let $\sigma^2<2\alpha$, $Q\in \mathcal{E}\backslash \{O\}$ any equilibrium of \eqref{sys6}
and $u_g$ the random equilibrium to \eqref{sysg}. Then the following holds:

(i) $Q_g(\cdot):=u_g(\cdot)Q$ is a stationary solution to \eqref{sys5}
and is supported on  $\overline{\mathcal{L}(Q)}$.

(ii) Its probability law $\mu_Q:=\mathbb{P}\circ Q_g^{-1}$
is a stationary measure of the
Markov semigroup $(P_t)$
corresponding to \eqref{sys5}.

$(iii)$  $\mu_Q$ is strongly mixing on $\mathbb{R}_+^3$.
In particular,
$\mu_Q$ is ergodic on $\mathbb{R}_+^3$,
i.e., $\mu_Q \in \mathcal{P}_e(\mathbb{R}_+^3)$.
\end{proposition}

\begin{proof}
$(i)$ Since $u_g$ is the random equilibrium to
the stochastic logistic equation \eqref{sysg},
one has $u_g(\theta_t\omega)=g(t, \omega, u_g(\omega))$.
Moreover, as $Q$ is an equilibrium of deterministic Kolmogorov system \eqref{sys6},
$\Psi(t, Q)=Q$, $\forall\,\, t\geq 0$. Then, by \eqref{sdf},
\begin{equation*}
\begin{aligned}
\Phi(t, \omega, Q_g(\omega))
&=g(t, \omega, u_g(\omega))\Psi(\int_0^tg^2(s, \omega, u_g(\omega))ds, Q)  \\
&=u_g(\theta_t\omega)\Psi(\int_0^tu_g^2(\theta_s\omega)ds, Q)
 =u_g(\theta_t\omega)Q
 =Q_g(\theta_t\omega),
\end{aligned}
\end{equation*}
which verifies that $Q_g$ is a random equilibrium to system \eqref{sys5}. Since $u_g\geq0$, it is clear that $Q_g$ is supported on $\overline{\mathcal{L}(Q)}$.

$(ii)$ Since $Q_g$ is a random equilibrium, the law of $\Phi(t, \omega, Q_g(\omega))$ is always $\mu_Q$. In addition, $u_g$ is $\mathcal{F}_-$-measurable, and so is $Q_g$. Then by Corollary 1.3.22 in \cite{KS12}, $\mu_Q$ is a stationary measure.

$(iii)$
Let us first prove that $\mu_Q$ is strongly mixing on $\mathcal{L}(Q)$. For this purpose, we first claim that for $x \in \mathcal{L}(Q)$ and $t>0$,
\begin{equation}\label{measu1}
  \mu_Q( \mathcal{L}(Q))=1=P(t,x, \mathcal{L}(Q)).
\end{equation}
To this end, let $B_r:=\{y\in \mathbb{R}^3_+: \| y\| <r \}$.
Note that
\begin{align}\label{esSS}
\mu_Q(B_r)&=\mathbb{P} (\|Q_g \|<r)
 =\mathbb{P} (u_g <\frac r{\|Q\|})
 =\int_0^{\frac r{\|Q\|}}p_{\alpha}^{\sigma}(s)ds,
\end{align}
where
$p_{\alpha}^{\sigma}$
is the density of $\mathbb{P}\circ u_g^{-1}$
given by \eqref{den}.
This yields that $$\mu_Q(\{O\})=\displaystyle \lim_{r\rightarrow 0}\mu_Q(B_r)=0,$$
which along with Proposition \ref{thmes} (i) implies
\begin{equation}\label{es5}
\mu_Q(\mathcal{L}(Q))=1.
\end{equation}

Moreover, since $\mathcal{L}(Q)$ is invariant under $\Psi$ by Theorem \ref{gd1}, we have
 $$\Psi(\int_0^tg^2(s,\omega, 1)ds, x)\in \mathcal{L}(Q),
 \ \ \forall \,\, t\geq0, \,\, \mathbb{P}-a.s.$$
Taking into account $g(t, \omega, 1)>0$, $\forall\,\, t\geq0$, $\mathbb{P}-a.s$ and \eqref{sdf}
we come to
\begin{equation*}
P(t, x, \mathcal{L}(Q))=\mathbb{P}\{\omega: g(t, \omega, 1)\Psi(\int_0^tg^2(s, \omega, 1)ds,x)\in \mathcal{L}(Q)\}=1.
\end{equation*}
This together with \eqref{es5} yields \eqref{measu1}, as claimed. Thus, we consider the Markov semigroup $(P_t)$ in $\mathcal{L}(Q)$.

Note that for any $x\in \mathcal{L}(Q)$,
by Lemmas \ref{isolated} and \ref{contiunnm}, 
\begin{equation}\label{infinity1}
 \Psi(t, x)\rightarrow Q \ \  as \; t\rightarrow \infty,
\end{equation}
which yields that $\omega_d(x)=\{Q\}$.
Then for any $f\in C_b(\mathcal{L}(Q))$,
by the $\theta$-invariant property under $\mathbb{P}$, we have
\begin{align*}
 \int_{\mathcal{L}(Q)}f(z)P(t,x, dz)
  = \int_{\Omega}f(\Phi(t, \omega, x))\mathbb{P}(d\omega)
  =  \int_{\Omega}f(\Phi(t, \theta_{-t}\omega, x))\mathbb{P}(d\omega).
\end{align*}
Since $\mathcal{L}(Q)$ is invariant under $\Psi(t, x)$
by Theorem \ref{gd1},
and so is $\Phi(t,\omega, x)$,
taking into account the
Lebesgue dominated convergence theorem,
\eqref{asy1d}, \eqref{infinity1} and  Theorem \ref{attract}
we can pass to the limit to get
\begin{align*}
\int_{\mathcal{L}(Q)}f(z)P(t,x, dz)
\to \int_{\Omega}f(u_g(\omega)Q)\mathbb{P}(d\omega)
 = \int_{\mathcal{L}(Q)}f(z)\mu_Q(dz).
\end{align*}
This yields that for $x\in \mathcal{L}(Q)$,
$$P(t, x, \cdot)\rightarrow \mu_Q(\cdot) \ \  \textrm{weakly}\,\, \textrm{in} \,\, \mathcal{P}(\mathcal{L}(Q)),\  as \;t\rightarrow \infty.$$
Thus, an application of Theorem \ref{prato} gives that
$\mu_Q$ is strongly mixing for the semigroup $(P_t)$ in $\mathcal{L}(Q)$,
and for any $\varphi \in L^2(\mathcal{L}(Q), \mu_Q)$,
\begin{equation}\label{ss}
\displaystyle \lim_{t\rightarrow \infty} P_t\varphi=\langle \varphi, 1\rangle,\,\, \textrm{in} \,\, L^2(\mathcal{L}(Q), \mu_Q).
\end{equation}

Finally,  we conclude from \eqref{measu1} and \eqref{ss}
that for any $\varphi \in L^2(\mathbb{R}_+^3, \mu_Q)$,
$$\displaystyle \lim_{t\rightarrow \infty} P_t\varphi=\langle \varphi, 1\rangle,
 \,\, \textrm{in} \,\, L^2(\mathbb{R}_+^3, \mu_Q).$$
This along with Theorem \ref{prato} yields that $\mu_Q$ is strongly mixing on $\mathbb{R}_+^3$.
In particular,
$\mu_Q$ is ergodic on  $\mathbb{R}_+^3$.
\end{proof}

\subsection{Relation to invariant cones}  \label{Subsec-Rel-Cone}
We now study ergodic stationary measures related to invariant cones.

\subsubsection{Existence}
\begin{lemma}\label{Invariant cone} (Cone invariance)
Let $\sigma^2<2\alpha$, $\alpha+d_i>0$, $i=1, 2,3$,
${h}^*$ be the constant given by \eqref{h2} and $\Gamma(h)$ be the closed orbit as in Theorem \ref{gd1} (i).
Then for any $h\in ({h}^*, \infty)$, the cone
$$\Lambda(h):=\{\lambda y:  y\in \Gamma(h),\,\, \forall \,\, \lambda \geq 0\},$$
is invariant under the RDS $\Phi$.
That is,
for any $x\in \Lambda(h)$, $t\geq0$, $\omega\in \Omega$,
\begin{equation}\label{invariant2}
\Phi(t, \omega,  x) \in \Lambda(h),\ \ and\ \
\Phi(t, \theta_{-t}\omega,  x) \in \Lambda(h).
\end{equation}
\end{lemma}

\begin{proof}
Lemma \ref{Invariant cone} follows from  formula \eqref{sdf},
the positivity of the logistic solution $g$
and the invariance of $\Gamma(h)$ under $\Psi$.
\end{proof}

The existence of stationary measures
on invariant cones
is the content of Proposition \ref{existence} below.

\begin{proposition}\label{existence}(Existence of stationary measures on invariant cones)
Let $\sigma^2<2\alpha$, $\alpha+d_i>0$, $i=1, 2, 3$.
Let $Q^*$ be the equilibrium of $\Psi$ as in Proposition \ref{pro1} (i).
Then for any $x\in Int(\mathbb{R}^3_+)\backslash \mathcal{L}(Q^*)$, there exist 
$h\in ({h}^*, \infty)$ and a stationary measure $\nu_x$
such that $x\in \Lambda(h)\setminus\{O\}$ and  $\nu_x(\Lambda(h)\setminus\{O\})=1$.
\end{proposition}

\begin{proof}
We recall from
Corollary \ref{decomposition} that
$$Int(\mathbb{R}^3_+)\setminus \mathcal{L}(Q^*)=\bigcup_{ {h}^*<h<\infty}(\Lambda(h)\backslash\{O\}).$$
Hence, for  $x\in Int(\mathbb{R}^3_+)\backslash \mathcal{L}(Q^*)$,
we have $x\in \Lambda(h)\setminus\{O\}$
for some $h\in ({h}^*, \infty)$.

Let $V$ and $\mathscr{L}^{\sigma}$
be the Lyapunov function and Fokker-Planck operator
as in \eqref{V} and \eqref{Fop}, respectively.
Then by straightforward computations,
$$\mathscr{L}^{\sigma}V(y)=V(y)(-2\alpha\|y\|^2+2\alpha+\sigma^2),$$
which yields that for sufficiently large $R>0$,
$$\displaystyle \sup_{\|y\|>R}\mathscr{L}^{\sigma}V(y)\leq -A_R:= R^2(-2\alpha R^2+4\alpha). $$
Hence,  as $R\rightarrow \infty$,
\begin{align} \label{lya1}
 \displaystyle \inf_{\|y\|>R}V(y)\rightarrow \infty, \ \
  \displaystyle \sup_{\|y\|>R}\mathscr{L}^{\sigma}V(y)\leq -A_R\rightarrow -\infty.
\end{align}
By virtue of Theorem 3.3.5 in \cite{krz}, we thus
derive that there exists a stationary measure $\nu_x \in \mathcal{P}(\mathbb{R}^3_+)$
of the Markov semigroup $(P_t)$, satisfying that for some sequence $\{t_n\}$ tending to infinity,
\begin{equation}\label{es}
 \frac1{t_n}\int_0^{t_n} P(s, x, \cdot)ds  \stackrel{w}{\rightharpoonup}  \nu_x
 \ \  as \,\, n\rightarrow \infty.
\end{equation}

Next we show that $\nu_x(\Lambda(h))=1$. To this end,
by \eqref{invariant2},  
\begin{equation}\label{cone}
\mathbb{P} (\Phi(t, \cdot, x) \in \mathbb{R}_+^3 \backslash \Lambda(h))=0,\ \  \forall\,\, x\in \Lambda(h).
\end{equation}
Then, using   \eqref{es} we get
\begin{align*}
\nu_x(\Lambda(h)^c) &\leq
\liminf\limits_{n\rightarrow \infty}\frac1{t_n}\int_0^{t_n}P(s, x, \mathbb{R}_+^3 \backslash \Lambda(h))ds\\
&=
\liminf\limits_{n\rightarrow \infty}\frac1{t_n}\int_0^{t_n}
    \mathbb{P}(\Phi(s,\cdot, x) \in \mathbb{R}_+^3 \backslash \Lambda(h)) ds
=0,
\end{align*}
which yields that $\nu_x(\Lambda(h))=1$, as claimed.

It remains to prove that $\nu_x(\{O\})=0$.
Note that, since $O$ is a local repeller by Theorem \ref{gd1}, and $x\neq O$,
there exists  $c(x)>0$ such that
\begin{equation}\label{ps}
     \inf_{t>0}\|\Psi(t, x) \| \geq c(x)>0.
\end{equation}
Then, by \eqref{sdf} and \eqref{ps},
\begin{equation*}\label{es4}
\begin{aligned}
P(t, x, B_R)
=\mathbb{P}(\|g(t, \cdot, 1)\Psi(\int_0^tg^2(s, \cdot, 1)ds, x)\|<R)
\leq \mathbb{P}(g(t, \cdot, 1)<\frac R{c(x)}).
\end{aligned}
\end{equation*}
Taking into account  \eqref{es} and \eqref{g1} we have
\begin{equation}\label{es1}
\begin{aligned}
\nu_x(B_R)&\leq \liminf \limits_{n\rightarrow \infty}\frac1{t_n}\int_0^{t_n}P(t, x, B_R)dt\\
&\leq \liminf \limits_{n\rightarrow \infty}\frac1{t_n}\int_0^{t_n}\mathbb{P}(g(t, \cdot,1)<\frac R{c(x)})dt\\
&\leq \lim \limits_{t\rightarrow \infty}\mathbb{P}(g(t, \cdot,1)<\frac R{c(x)})
= \int_0^{\frac R{c(x)}}p_{\alpha}^{\sigma}(s)ds,
\end{aligned}
\end{equation}
where $p_{\alpha}^{\sigma}$ is the density of $\mu_g^{\sigma}$ given by \eqref{den}.
Hence, letting $R\rightarrow 0$ we have $\nu_x(\{O\})=\lim_{R\rightarrow 0}\nu_x(B_R)=0$.
The proof is thus complete.
\end{proof}

\subsection{Uniqueness}

We further prove that
the stationary measure on each invariant cone without the origin $O$ is indeed unique.
This is the content of Proposition \ref{uniqueness} below.

\begin{proposition}\label{uniqueness}(Uniqueness of stationary measures on invariant cones)
Assume the conditions in Proposition \ref{existence} to hold.
Then for any $h\in ({h}^*, \infty)$,
there exists a unique, ergodic stationary measure $\nu_h$ on $\Lambda(h)\backslash\{O\}$.

Moreover,
$\nu_h$ is strongly mixing on $\mathbb{R}_+^3$.
In particular, $\nu_h \in \mathcal{P}_e(\mathbb{R}_+^3)$.
\end{proposition}

\begin{proof}
We use the analogous arguments as in \cite{jiang1}.
 Fix $h\in ({h}^*, \infty)$ and $y_0\in \Gamma(h)$. Let $\varphi(y):=\inf\{t>0, \Psi(t, y_0)=y\}$ for any $y\in \Gamma(h)$. Let $T:=\varphi(y_0)$ be the period of the orbit $(\Psi(t, y_0))$, and $\mathbb{S}:= \mathbb{R}_+$ mod $T$. Then, $\varphi: \Gamma(h)\rightarrow \mathbb{S}$ is a homeomorphism.

 By the definition of $\Lambda(h)$, for any $z\in \Lambda(h)\setminus\{O\}$, there exist
 $\lambda>0$ and $y\in \Gamma(h)$ such that $z=\lambda y$. Then, define $\psi: \Lambda(h)\setminus\{O\}\rightarrow \mathbb{R}\times \mathbb{S}$ by
$$\psi(z):=(\ln \lambda, \varphi(y)), \ \  \forall \ z\in \Lambda(h)\setminus\{O\},$$
 Note that $\psi: \Lambda(h)\setminus\{O\}\rightarrow \mathbb{R}\times \mathbb{S}$ is a homeomorphism, and its inverse is $\psi^{-1}(x, \tau)=e^x\Psi(\tau, y_0)$. Moreover, for any $z=\lambda y\in \Lambda(h)\backslash \{O\}$ with $\lambda>0$ and $y\in \Gamma(h)$, by \eqref{sdf}
 and the invariance of $\Gamma(h)$ under $\Psi$, 
 \begin{equation*}
   \Psi(\int_0^tg^2(s, \omega, \lambda)ds,y)\in \Gamma(h),
   \end{equation*}
and
\begin{equation}\label{es10}
\Phi(t, \omega, z)=g(t, \omega, \lambda)\Psi(\int_0^tg^2(s, \omega, \lambda)ds,y)\; \in \Lambda(h)\setminus \{O\}.
\end{equation}
Then let $(H_0, T_0)=\psi(z)=(\ln \lambda, \varphi(y))$ and set
$$H(t, \omega, H_0):=\ln (g(t, \omega, \lambda)),\; T(t, \omega, H_0, T_0):=\varphi(\Psi(\int_0^tg^2(s, \omega, \lambda)ds, y)).$$
We get
\begin{equation*}
\begin{aligned}
\psi(\Phi(t, \omega, z))&=(\ln (g(t, \omega, \lambda)), \varphi(\Psi(\int_0^tg^2(s, \omega, \lambda)ds,y)))\\
&=(H(t, \omega, H_0), T(t, \omega, H_0, T_0)).
\end{aligned}
\end{equation*}
Thus, $\Phi$ on $\Lambda(h)\setminus\{O\}$ and $(H, T)$ on $\mathbb{R}\times \mathbb{S}$ are conjugate through the mapping $\psi$.

Note that, by It\^o's formula and the definition of $\varphi$,
\begin{align}\label{newsystem}
H(t, \omega, H_0)&=H_0+\int_0^t(\alpha-\frac12 \sigma^2-\alpha e^{2H(s, H_0)})ds +\int_0^t\sigma dW_s,\\
T(t, H_0, T_0)&=(T_0+\int_0^te^{2H(s, H_0)}ds) \ \  mod \; T.
\end{align}

\paragraph{\bf Strong Feller:}
Let us first prove that the Markov semigroup associated to $(H, T)$ on $\mathbb{R}\times \mathbb{S}$ is strong Feller at any time $t>0$. To this end,
for any $(H_0, \tilde{T}_0)\in \mathbb{R}^2$, consider the stochastic equations
\begin{equation}\label{newnewsystem}
\begin{aligned}
H(t, H_0)&=H_0+\int_0^t(\alpha-\frac12\sigma^2-\alpha e^{2H(s, H_0)})ds+ \int_0^t \sigma dW_s,\\
\tilde{T}(t, H_0, \tilde{T}_0)&=\tilde{T}_0+\int_0^t e^{2H(s, H_0)}ds.
\end{aligned}
\end{equation}
By Theorem 4.2 in \cite{dong1}, the corresponding semigroup $(\tilde{P}_t)_{t\geq 0}$ is strong Feller on $\mathbb{R}^2$ for any $t>0$, i.e.,  $\forall \,\, f\in \mathcal{B}_b(\mathbb{R}^2)$,
$$(H_0, \tilde{T}_0)\in \mathbb{R}^2\mapsto \tilde{P}_t f= \mathbb{E}f(H(t, H_0), \tilde{T}(t, H_0, \tilde{T}_0))\,\, \textrm{is} \,\, \textrm{continuous} .$$
Hence, for any $F\in \mathcal{B}_b(\mathbb{R}\times \mathbb{S})$, letting $f_F(H, \tilde{T}):=F(H, \tilde{T}\; mod \; T)\in \mathcal{B}_b(\mathbb{R}^2)$, we have
\begin{equation*}
\begin{aligned}
(H_0, T_0)\in \mathbb{R}\times \mathbb{S} &\mapsto \mathbb{E}F(H(t, H_0),
T(t, H_0, T_0))\\
&=\mathbb{E}f_F(H(t, H_0), \tilde{T}(t, H_0, T_0)) \,\, \textrm{is} \,\, \textrm{continous},
\end{aligned}
\end{equation*}
which yields that $(H, T)$ is strong Feller on $\mathbb{R}\times \mathbb{S}$ at any $t>0$.

\paragraph{\bf Irreducibility:}
 Next we prove that $(H, T)$ is irreducible on $\mathbb{R}\times \mathbb{S}$,
 that is, for any $t>0$, for any $A:=(a, b)\times (c, d)\in \mathbb{R}\times \mathbb{S}$
 with $a<b$ and $c<d$,
\begin{equation}\label{aim}
\mathbb{P}((H(t, H_0), T(t, H_0, T_0))\in A)>0,  \ \ \forall \,\, (H_0, T_0)\in \mathbb{R}\times \mathbb{S}.
\end{equation}

In order to prove \eqref{aim}, we set
$$\tilde{A}:=(e^a, e^b)\times A_{c, d},\ \  A_{c,d}:=\bigcup_{n=0}^{\infty}(c+nT-T_0, d+nT-T_0).$$
Define the map $\mathbb{L}: C([0,t], \mathbb{R}_+)\rightarrow \mathbb{R}^2$ by
$$\mathbb{L}(f):=(\frac{f(t)}{\sqrt{e^{-2H_0}+2\alpha\int_0^tf^2(s)ds}}, \int_0^t\frac{f^2(s)}{e^{-2H_0}+2\alpha\int_0^sf^2(r)dr}ds).$$
Then $\mathbb{L}$ is continuous on $C([0,t]; \mathbb{R}_+)$,
and by the definition of $(H, T)$,
\begin{equation}\label{aim1}
\mathbb{P}((H(t, H_0), T(t, H_0, T_0))\in A)=\mathbb{P}(\mathbb{L}(e^{(\alpha-\frac12\sigma^2)\cdot+\sigma W_{\cdot}})\in \tilde{A}).
\end{equation}

To analyse the right-hand side above, we set
\begin{equation}\label{aim2}
B:=\{f\in C_1([0,t], \textrm{Int}\mathbb{R}_+):\ \mathbb{L}(f)\in \tilde{A}\}
\end{equation}
and shall prove that $B\neq \varnothing$,
where $C_1([0,t], \textrm{Int}\mathbb{R}_+)$
is the set of all continuous functions in $\textrm{Int}\mathbb{R}_+$ starting from $1$ at time $0$.

To this end, let us first consider the set
$$\tilde{B}:=\{h\in C([0,t], \textrm{Int}\mathbb{R}_+): h(0)=e^{H_0}, (h(t), \int_0^th^2(s)ds)\in \tilde{A}\}.$$
 Take $\tilde{n}$ large enough such that $c+\tilde{n}T-T_0\geq \frac{t(e^{2H_0}+e^{2b})}4$, and let $\tilde{l}:=\frac{c+d+2\tilde{n}T-2T_0}t-\frac12 e^{2H_0}-\frac18(e^a+e^b)^2>0$. Define $h\in C([0,t], \textrm{Int}\mathbb{R}_+)$ by
\begin{equation}
h(s):=
\begin{cases}
\sqrt{\frac{2(\tilde{l}-e^{2H_0})}t\cdot s+ e^{2H_0}},\ \ 0\leq s\leq\frac t2,\\
\sqrt{\frac{2(\frac{e^a+e^b}2)^2-2\tilde{l}} t\cdot s+2\tilde{l}-(\frac{e^a+e^b}2)^2}, \ \ \frac t2<s\leq t.
\end{cases}
\end{equation}
Then $h(0)=e^{H_0}$, $h(t)=\frac{e^a+e^b}2\in (e^a, e^b)$, and $h(\frac t2)=\sqrt{\tilde{l}}$. Note that
 $$\int_0^th^2(s)ds=\frac t4(e^{2H_0}+2\tilde{l}+\frac{(e^a+e^b)^2}4)=\frac{c+d}2+\tilde{n}T-T_0\in A_{c,d}.$$
This yields that $h\in \tilde{B}$, and so, $\tilde{B}\neq\varnothing$.

Then, coming back to the set $B$ defined in \eqref{aim2} we take $h\in \tilde{B}$ and  $$f(s):=e^{-H_0}h(s)e^{\alpha\int_0^sh^2(r)dr},\; s\in [0,t].$$
Then $f\in C_1([0,t], \textrm{Int}\mathbb{R}_+)$, $f(0)=e^{-H_0}h(0)=1$, $$\frac{f(t)}{\sqrt{e^{-2H_0}+2\alpha\int_0^tf^2(s)ds}}=h(t)\in (e^a, e^b),$$
and $$\int_0^t\frac{f^2(s)}{e^{-2H_0}+2\alpha\int_0^sf^2(r)dr}ds=\int_0^th^2(s)ds\in A_{c,d}.$$
This yields that $f\in B$, and so,
$B\neq\varnothing$, as claimed.
In particular,
$\mathbb{L}^{-1}(\tilde{A})$ is a non-empty open set in $C_1([0,t]; \textrm{Int}\mathbb{R}_+)$.

Then, define the map
$\mathscr{E}: C_0([0,t], \mathbb{R})\rightarrow C_1([0,t], \textrm{Int}\mathbb{R}_+)$ by
$$\mathscr{E}(\tilde{f}):=e^{(\alpha-\frac12 \sigma^2)\tilde{f}+\sigma \tilde{f}(\cdot)},
\ \ \tilde{f}\in C_0([0,t];  \mathbb{R}), $$
where $C_0([0,t], \mathbb{R})$ is the set of all
continuous functions in $\mathbb{R}$ starting from $0$.

Note that $\mathscr{E}$ is continuous, and so,
$\mathscr{E}^{-1}\circ \mathbb{L}^{-1}(\tilde{A})$
is a non-empty open set in $C_0([0,t];   \mathbb{R})$, the irreducibility of Wiener process  (see e.g. \cite{Nualart})
then yields
\begin{align*}
\mathbb{P}((H(t, H_0), T(t, H_0, T_0))\in A)&=\mathbb{P}(\mathbb{L}(e^{(\alpha-\frac12\sigma^2)\cdot+\sigma W_{\cdot}})\in \tilde{A}),\\
&=\mathbb{P}(\mathscr{E}^{-1}\circ \mathbb{L}^{-1}(\tilde{A}))>0.
\end{align*}
Thus, $(H, T)$ is irreducible on $\mathbb{R}\times \mathbb{S}$ for any $t>0$.

\paragraph{\bf Uniqueness and strong mixing:}
Now,
since strong Feller and irreducibility are equivalent under
conjugation maps,
we infer that the Markovian semigroup $(P_t)$ associated with $\Phi$ on $\Lambda(h)\setminus \{O\}$
is strong Feller and irreducible at any $t>0$,
which in turn yields the uniqueness and strongly mixing of
$\Phi$ on $\Lambda(h)\setminus\{O\}$.

Let $\nu_h$ be this unique stationary measure on $\Lambda(h)\setminus \{O\}$.
Then, an application of Theorem \ref{prato} gives that for any $\varphi\in L^2(\Lambda(h)\setminus \{O\}, \nu_h)$,
\begin{equation}\label{uni}
\displaystyle\lim_{t \rightarrow \infty}P_t\varphi = \langle \varphi, 1\rangle \ \  in \; L^2(\Lambda(h)\setminus\{O\}, \nu_h).
\end{equation}

Finally, we claim that $\nu_h$ is also strongly mixing on $\mathbb{R}_+^3$.
Actually, for any $x\in \Lambda(h)\setminus \{O\}$ and $t>0$, by \eqref{es10},
\begin{equation}\label{es8}
P(t,x, \Lambda(h)\setminus \{O\})=\mathbb{P}\{\omega: \Phi(t,\omega,x)\in \Lambda(h)\setminus \{O\}\}=1.
\end{equation}
Then,
since  $\nu_h( \Lambda(h)\setminus \{O\})=1$,
by \eqref{uni} and \eqref{es8}, we get that for any $\varphi\in L^2(\mathbb{R}_+^3, \nu_h)$,
\begin{align*}
&\displaystyle\lim_{t \rightarrow \infty}\int_{\mathbb{R}^3_+}| P_t\varphi(x)-\langle \varphi, 1\rangle(x)|^2 \nu_h(dx)\\
=&\displaystyle\lim_{t \rightarrow \infty}\int_{\Lambda(h)\setminus \{O\}}|\int_{\mathbb{R}^3_+}\varphi(y)P(t,x, dy)-\langle \varphi, 1\rangle(x) |^2\nu_h(dx)\\
=&\displaystyle\lim_{t \rightarrow \infty}\int_{\Lambda(h)\setminus \{O\}}|\int_{\Lambda(h)\setminus \{O\}}\varphi(y)P(t,x, dy)-\langle \varphi, 1\rangle(x) |^2 \nu_h(dx)
=0,
\end{align*}
which yields that \eqref{uni} holds in $L^2(\mathbb{R}_+^3, \nu_h)$.
Therefore, by Theorem \ref{prato},
$\nu_h$ is strongly mixing on $\mathbb{R}^3_+$.
\end{proof}

\subsection{Vanishing noise limit}  \label{Subsec-Vanish}
This section concerns the relationship
between stationary measures for the deterministic and stochastic Kolmogorov systems,
when the noise intensity $\sigma$ tends to zero.
In order to indicate the dependence on the noise strength $\sigma$ in \eqref{sys5},
we rewrite the ergodic stationary measures
$\mu_Q$ and $\nu_h$ in Propositions \ref{thmes} and \ref{existence} as $\mu_Q^{\sigma}$ and $\nu_h^{\sigma}$, respectively.

\begin{theorem}\label{limitmeasure}(Vanishing noise limit of stationary measures)
\begin{itemize}
  \item [$(i)$] Let $\sigma^2<2\alpha$ and $Q\in \mathcal{E}\backslash \{O\}$ be any equilibrium of \eqref{sys6}. Then
  \begin{align*}
     \mu_Q^{\sigma} \stackrel{w}{\rightharpoonup}\delta_Q\ \ as\ \sigma \rightarrow 0.
  \end{align*}
  \item [$(ii)$] Assume the conditions in Proposition \ref{existence} to hold.
  Let $\nu_h^\sigma$ denote the corresponding ergodic stationary measure on $\Lambda(h)$,
  $h\in ({h}^*, \infty)$,
  and $\tilde{\nu}_h$ the Haar measure on $\Gamma(h)$.
  Then,
  \begin{align*}
     \nu_h^\sigma \stackrel{w}{\rightharpoonup} \tilde{\nu}_h \ \ as\ \sigma \rightarrow 0.
  \end{align*}
\end{itemize}
\end{theorem}

Let us first show the tightness of stationary measures.

\begin{lemma} (Tightness) \label{limitmeasure1}
Fix $\alpha>0$ and suppose that $\sigma^2< 2\alpha$.
Then both $\{\mu_Q^{\sigma}\}_{\sigma}$ and $\{\nu_h^{\sigma}\}_{\sigma}$ are tight on $\mathcal{P}(\mathbb{R}_+^3)$.

Moreover, if $\mu_Q^{\sigma_i}\stackrel{w}{\rightharpoonup} \mu$
and $\nu_h^{\sigma_j} \stackrel{w}{\rightharpoonup} \nu$ as $\sigma_j \rightarrow 0$,
then both $\mu$ and $\nu$ are invariant measures of system \eqref{sys6}, and $\mu(\mathbb{S}^2_+) =\nu(\mathbb{S}^2_+)=1$.
\end{lemma}

\begin{proof}
Let $C^*:=\sup_{\sigma\in (0, \sqrt{2\alpha})}\sup_{x\in \mathbb{R}^3_+}\mathscr{L}^{\sigma}V(x)$ and define
$U_R^c:=\{x\in \mathbb{R}^3_+: \|x\|>R\}.$
Since the Lyapunov function $V$ associated to \eqref{sys5} satisfies \eqref{lya1},
in view of the proof of Theorem 3.1 in \cite{jiang2},
it follows that any stationary measure $\tilde{\mu}$ satisfies
$\tilde{\mu}(U_R^c)\leq {C^*}/{A_R},$
where $A_R= R^2(-2\alpha R^2+4\alpha).$
In particular, for the stationary measure $\mu_Q^{\sigma}$,
$$\sup_{\sigma\in (0, \sqrt{2\alpha})}\mu_Q^{\sigma}(U_R^c)\leq \frac {C^*} {A_R}\rightarrow 0, \ \ as\ R\rightarrow \infty.$$
Hence, $\{\mu_Q^{\sigma}\}$ is tight.
Similar arguments also give the tightness of $\{\nu_h^{\sigma}\}$.
Moreover, the invariance of $\mu$ and $\nu$ follows from Theorem 3.1 in \cite{jiang2}.

Below we prove that $\mu(\mathbb{S}^2_+)=\nu(\mathbb{S}^2_+)=1$.
To this end, by the Poincar\'e recurrence theorem
(see \cite[Theorem A.1]{jiang2}), one has $\mu(B(\Psi))=\nu(B(\Psi))=1$,
where
$$B(\Psi):=\overline{\{x\in \mathbb{R}^3_+: x\in\omega_d(x)\}}$$
is the Birkhoff center of $\Psi$.
Note that,
by Lemma \ref{dglobal}, $B(\Psi)= \mathbb{S}^2_+\bigcup \{O\}$.
Thus, we only need to prove that $\mu(\{O\})=\nu(\{O\})=0$.

We first prove that $\nu(\{O\})=0$.
For this purpose,
we recall from the proof of Propositions \ref{existence} and \ref{uniqueness} that there exists $x\in \Lambda(h)\setminus\{O\}$ such that $\nu_h^{\sigma_i}=\nu_x^{\sigma_i}$ and \eqref{es1} holds. Then we have
\begin{equation}\label{es2}
\nu(\{O\})=\lim\limits_{R\rightarrow0}\nu(B_R)\leq  \lim\limits_{R\rightarrow0}\liminf\limits_{\sigma_i \rightarrow 0}\nu_x^{\sigma_i}(B_R) \leq \lim\limits_{R\rightarrow0}\liminf\limits_{\sigma_i \rightarrow 0}\int_0^{\frac Rc}p_{\alpha}^{\sigma_i}(s)ds,
\end{equation}
where $c$ is the positive lower bound in \eqref{ps},
which is independent of $\sigma_i$.

Note that, for $\sigma_i>0$ very small, $p_{\alpha}^{\sigma_i}$
satisfies the following properties:
\begin{enumerate}
  \item[(a)] $p_{\alpha}^{\sigma_i}(0)=0$;

  \item[(b)] $p_{\alpha}^{\sigma_i}(s)$ is increasing for $0<s<s_*(b):=\sqrt{1- 1/b}$
  with $b:={\alpha}/{\sigma_i^2}$,
  but decreasing for $s^2>s_*(b)$.
  It reaches the maximum at $s_*(b)$.

  \item[(c)] $\mathop{\lim}_{b\rightarrow\infty}\mathop{max}\limits_{s} p_{\alpha}^{\sigma_i}(s)=\mathop{\lim}\limits_{b\rightarrow\infty}\frac{2b^{0.5}(b-1)^{b-1}}{\Gamma(b-\frac12)}e^{1-b}=\infty$ and $\mathop{\lim}\limits_{b\rightarrow\infty}s_*(b)=1$;

  \item[(d)] $\int_0^{\infty}p_{\alpha}^{\sigma_i}(s)ds=1$, $\forall \,\,\sigma_i>0$.
\end{enumerate}

For any $\epsilon>0$,
take $b$ large enough (or $\sigma_i$ very small) such that $|s_*(b)-1|<\epsilon$,
and for sufficiently small $R$, choose $M>0$ such that $M<1-2\epsilon$ and $R/c<M$.
Then, $p_{\alpha}^{\sigma_i}$ is increasing on $[0, M]$, and by properties (b) and (d),
$$|R/c-M|p_{\alpha}^{\sigma_i}(R/c)\leq \int_{\frac Rc}^Mp_{\alpha}^{\sigma_i}(s)ds<1,$$
so $p_{\alpha}^{\sigma_i}(R/c)<1/{|R/c-M|}$.
This yields that
$$\int_0^{\frac Rc}p_{\alpha}^{\sigma_i}(s)ds\leq \frac Rc p_{\alpha}^{\sigma_i}(R/c)< \frac R{c|R/c-M|}.$$
Thus, plugging this into \eqref{es2}
and passing to the limit $R, \sigma_i\to 0$
we have
$$\nu(\{O\})=\lim\limits_{R\rightarrow0} \frac R{c|R/c-M|}=0,$$
as claimed.

Similarly, by \eqref{esSS},
$$\mu(\{O\})=\lim\limits_{R\rightarrow0}\mu(B_R)\leq\lim\limits_{R\rightarrow0}\liminf\limits_{\sigma_i \rightarrow 0}\int_0^{\frac R{\|Q\|}}p_{\alpha}^{\sigma_i}(s)ds.$$
Arguing as above we get $\mu(\{O\})=0,$
thereby finishing the proof.
\end{proof}

\begin{proof}[Proof of Theorem \ref{limitmeasure}]
$(i)$ By Lemma \ref{limitmeasure1}, for any equilibrium
$Q\in \mathcal{E}\setminus \{O\}$ and for any sequence $\{\sigma_n\}$
converging to zero,
there exist a subsequence (still denoted by $\{\sigma_n\}$)
and $\mu \in \mathcal{P}(\mathbb{R}_+^3)$ such that
\begin{equation}\label{weak limit}
\mu_Q^{\sigma_n} \stackrel{w}{\rightharpoonup}\mu \ \  as \; \sigma_n \rightarrow 0.
\end{equation}

Moreover, since $\mu_Q^{\sigma_n}$ is supported on $\overline{\mathcal{L}(Q)}$ due to Proposition \ref{thmes},
$n\geq 1$,
\eqref{weak limit} yields that the support of $\mu$ is contained in  $\overline{\mathcal{L}(Q)}$. But by Lemma \ref{limitmeasure1}, $\mu(\mathbb{S}^2_+)=1$.
Hence, $\textrm{supp} (\mu) \subseteq \overline{\mathcal{L}(Q)}\cap \mathbb{S}^2_+=\{Q\}$, and so $\mu=\delta_Q$.
Thus, the limit in \eqref{weak limit} is unique,
we infer that \eqref{weak limit} is valid for any sequence $\sigma_n\rightarrow 0$.
The first statement $(i)$ holds.

$(ii)$ Applying Lemma \ref{limitmeasure1} again,
for any sequence $\sigma_n\rightarrow 0$, there exists a subsequence (still denoted by $\{\sigma_n\}$) such that
\begin{equation*}
\nu_h^{\sigma_n} \stackrel{w}{\rightharpoonup}\nu\ \ as\ \sigma_n \rightarrow 0,
\end{equation*}
and $\nu(\mathbb{S}^2_+)=1$.
But by Proposition \ref{uniqueness},
$\textrm{supp}(\nu_h^{\sigma_n}) \subseteq\Lambda(h)$, $n \geq 1$, and so $\textrm{supp}(\nu) \subseteq\Lambda(h)$.
It follows that
$\textrm{supp}(\nu)\subseteq\Lambda(h)\cap \mathbb{S}^2_+= \Gamma(h)$. 
Taking into account that $\nu$ is an invariant measure we infer that $\textrm{supp}(\nu)=\Gamma(h)$,
and so $\nu=\tilde{\nu}_h$ is a Haar measure on $\Gamma(h)$.
Thus, as in the proof of (i),
since the limit is unique, the statement (ii) holds.
\end{proof}

\section{Proof of main results}   \label{Sec-Stocha-Kol}

We are now ready to prove Theorem \ref{bifurcations} and give the complete classification of global dynamics from the perspective of ergodic stationary measures and pull-back $\Omega$-limit sets for the stochastic Kolmogorov system.

\subsection{Stochastic bifurcations}\label{Stochastic bifurcation}

This Subsection is devoted to proving 
the bifurcation of ergodic stationary measures 
in  Theorem \ref{bifurcations}.

Let us first calculate the Lyapunov exponents of ergodic stationary measures associated with random equilibria. Recall that $\mathcal{E}$ denotes the set of all equilibria of system \eqref{sys6}, $\mathcal{P}_e(\mathbb{R}_+^3)$ is the set of all ergodic stationary measures of system \eqref{sys5} and $\mathcal{A}(Q)$ is the attracting domain of some equilibrium $Q$.
\begin{lemma}\label{lyapunov} (Lyapunov exponents)
Let $\lambda_i(\nu)$, $i=1, 2, 3$, denote the Lyapunov exponents of
the stationary measure $\nu$, where
$\nu\in \{\delta_O, \mu_{\textbf{e}_1}, \mu_{\textbf{e}_2}, \mu_{\textbf{e}_3}\}$.
Then, the following holds:
\begin{itemize}
  \item[(i)] $\lambda_i(\delta_O)=\alpha-\frac12\sigma^2$, $i=1, 2, 3$.
  \item[(ii)] If $\sigma^2<2\alpha$,
  then
      \begin{align*}
       \lambda_1(\mu_{\textbf{e}_i})&= -2(\alpha-\frac12\sigma^2),         \\
       \lambda_2(\mu_{\textbf{e}_i})&=(-1)^{i+1} \frac{\alpha+n_i}{\alpha} (\alpha-\frac12\sigma^2), \\
       \lambda_3(\mu_{\textbf{e}_i})&=(-1)^i \frac{\alpha+m_i}{\alpha}(\alpha-\frac12\sigma^2), \end{align*}
        where,
  $n_1=d_1$ and $m_1=d_2$ if $i=1$,
                $n_2=d_1$ and $m_2=d_3$ if $i=2$,
                $n_3=d_2$ and $m_3=d_3$ if $i=3$.

\end{itemize}
\end{lemma}

\begin{proof}
$(i)$ Let $v\in \mathbb{R}_+^3 \setminus \{O\}$.
Consider the linearization $v_t:=D\Phi(t,\omega, x)v$
of system \eqref{sys5}.
Note that $v_t$ satisfies the linear SDE
\begin{equation}\label{linearized}
dv_t=F(\Phi(t, \cdot,  x))v_tdt+\sigma v_tdW_t,
\end{equation}
where

$$F=\scriptsize{\left(
                     \begin{array}{ccc}
                       \alpha-3\alpha x_1^2-(2\alpha+d_1)x_2^2+d_2x_3^2 & -2(2\alpha+d_1)x_1x_2 & 2d_2x_3x_1\\
                       2d_1x_1x_2 & \alpha+d_1x_1^2-3\alpha x_2^2-(2\alpha+d_3)x_3^2 & -2(2\alpha+d_3)x_2x_3 \\
                       -2(2\alpha+d_2)x_1x_3 & 2d_3x_2x_3 & \alpha-(2\alpha+d_2)x_1^2+d_3x_2^2-3\alpha x_3^2 \\
                     \end{array}
                   \right)}.
$$
Note that, using the transform
\begin{align}  \label{u-zv}
   u(t):=z(t)v(t)\ \ \textrm{with}\ z(t):=\exp\{\frac12\sigma^2t-\sigma W_t\}
\end{align}
we can reformulate \eqref{linearized} as follows
\begin{equation}\label{linearized1}
du= F(\Phi(t, \omega, x)) udt, \ \  u(0)=v.
\end{equation}

Below, we solve \eqref{linearized1} to compute the corresponding Lyapunov exponents.
For $\delta_O$, note that
\begin{equation}\label{matrix1}
F(\Phi(t, \omega, O))=\left(
\begin{array}{ccc}
\alpha &0 & 0 \\
 0 & \alpha &0 \\
   0& 0 & \alpha \\
    \end{array}
    \right),
\end{equation}
and \eqref{linearized1} has the unique solution
$$u(t)=\exp(\alpha t)v,$$
which, via \eqref{u-zv},  yields that \eqref{linearized} has the solution
$$D\Phi(t, \omega, x)v=\exp(-\frac12 \sigma^2 t+\sigma W_t)\exp(\alpha t) v.$$
Hence, we compute that for $i=1, 2, 3$
and any $v\in \mathbb{R}_+^3 \setminus  \{O\}$,
\begin{align*}
\lambda_i(\delta_O)
=& \displaystyle\lim\limits_{t\rightarrow \infty} \frac1t \log \|D\Phi(t,\omega, x)v\|  \\
=& -\frac12 \sigma^2+ \displaystyle\lim\limits_{t\rightarrow \infty} \frac1t \log\|\exp(\alpha t)v\|
= -\frac12 \sigma^2+\alpha.
\end{align*}

$(ii)$ 
Concerning the measure $\mu_{\textbf{e}_1}$, 
note that
\begin{equation*}\label{matrix2}
F(\Phi(s, \omega, u_g(\omega)\textbf{e}_1))= \left(\tiny{\begin{array}{ccc}
  \alpha-3\alpha u_g^2(\theta_s\omega) &0 & 0 \\
     0 & \alpha+d_1u_g^2(\theta_s\omega) &0 \\
       0& 0 & \alpha-(2\alpha+d_2)u_g^2(\theta_s\omega) \\
            \end{array}}
              \right).
\end{equation*}
Thus, let $v=(1,0,0)^T$. The solution of \eqref{linearized1}   is
$$u(t)=v \exp(\int_0^t \alpha-3\alpha u_g^2(\theta_s\omega)ds),$$
which, via \eqref{u-zv}, yields that
$$D\Phi(t, \omega, x)v=v\exp(-\frac12 \sigma^2 t+\sigma W_t)\exp(\int_0^t \alpha-3\alpha u_g^2(\theta_s\omega)ds).$$
Hence,  
by the Birkhoff-Khintchin ergodic theorem, we have
\begin{equation}\label{u_g_lyapunov}
\begin{split}
\lambda_1(\mu_{\textbf{e}_1})&=\displaystyle\lim\limits_{t\rightarrow \infty} \frac1t \log\|D\Phi(t, \omega, x)v\|\\
 &= \alpha-\frac12 \sigma^2-3\alpha\displaystyle\lim\limits_{t\rightarrow \infty} \frac1t\int_0^t u_g^2(\theta_s\omega)ds\\
&= \alpha-\frac12 \sigma^2 -3\alpha\mathbb{E}u_g^2. 
\end{split}
\end{equation}
Since 
$$u_g^2(\theta_t\omega)=\frac{\psi'(t,\omega)}{2\alpha \psi(t,\omega)},$$
where
$$\psi(t,\omega)=\int_{-\infty}^t\exp\{2(\alpha-\frac12\sigma^2)s+2\sigma W_s(\omega)\}ds,
$$
we compute 
\begin{align}\label{EU}
\mathbb{E}(u_g^2)=\displaystyle \lim_{s\rightarrow \infty} \frac1s\int_0^s u_g^2(\theta_t\omega)dt
=\frac1{2\alpha}\displaystyle \lim_{s\rightarrow \infty}\frac1s\log \psi(s)
= \frac{1}{\alpha}(\alpha-\frac12\sigma^2).
\end{align}
Plugging this into \eqref{u_g_lyapunov} we get
$$\lambda_1(\mu_{\textbf{e}_1})=-2(\alpha-\frac12\sigma^2).$$

Similarly, taking $v=(0,1,0)^T$ we have
$$\lambda_2(\mu_{\textbf{e}_1})= \alpha-\frac12 \sigma^2 +d_1\mathbb{E}u_g^2=\frac{\alpha+d_1}{\alpha}(\alpha-\frac12\sigma^2),$$
and taking $v=(0,0,1)^T$ we have
$$\lambda_3(\mu_{\textbf{e}_1})= \alpha-\frac12 \sigma^2 -(2\alpha+d_2)\mathbb{E}u_g^2=-\frac{\alpha+d_2}{\alpha}(\alpha-\frac12\sigma^2).$$

The proof for the remaining cases where $Q\in \{\textbf{e}_2, \textbf{e}_3\}$
is similar.
\end{proof}

\medskip
{\bf Proof of Theorem \ref{bifurcations}}
(i)-(ii): Note that $O$ is a random equilibrium and $\delta_O$ is an ergodic stationary measure. Then by lemma \ref{lyapunov}, the Lyapunov exponents of $\delta_O$ are all negative when $ \sigma^2>2\alpha$, and all zero when $\sigma^2=2\alpha$. Moreover, by Theorem \ref{attract} (i), the random equilibrium $O$ is a global attractor. Hence, it remains to prove that $\delta_O$ is the unique ergodic stationary measure of system \eqref{sys5}.

For this purpose, first note that by Theorem \ref{attract} (i), for any $x\in \mathbb{R}^3_+$, $\Phi(t, \theta_{-t}\omega, x)\rightarrow O$ almost surely.
Then, 
using the Lebesgue-dominated convergence theorem and the invariance of $\theta_t$ under $\mathbb{P}$ we derive that for any $f\in C_b(\mathbb{R}^3_+)$,
\begin{align*}
\displaystyle\lim\limits_{t\rightarrow \infty}\int_{\mathbb{R}^3_+}f(z)P(t,x, dz)
= \displaystyle\lim\limits_{t\rightarrow \infty}\int_{\Omega}f(\Phi(t, \theta_{-t}\omega, x))\mathbb{P}(d\omega)
=\int_{\mathbb{R}^3_+}f(z)\delta_O (dz),
\end{align*}
which yields that
\begin{equation}\label{es 1}
\displaystyle\lim\limits_{t\rightarrow \infty} P(t, x, \cdot)\rightarrow \delta_O  \ \  \textrm{weakly} \,\, \textrm{in} \,\, \mathcal{P}(\mathbb{R}^3_+).
\end{equation}

Now assume that $\nu\in \mathcal{P}(\mathbb{R}^3_+)$ is another ergodic stationary measure such that $\nu(\cdot)\ne \delta_O(\cdot)$. 
Then, in view of \eqref{es 1}, one has
\begin{equation}\label{es_3}
\int_{\mathbb{R}^3_+}P(t, x, \cdot)\nu(dx)\stackrel{w}{\rightharpoonup}\delta_O(\cdot), \ \ as\ t \rightarrow \infty.
\end{equation}
But by the definition of stationary measures, for any $t\ge 0$, one has $\int_{\mathbb{R}^3_+}P(t, x, \cdot)\nu(dx)=\nu(\cdot)$, which violates \eqref{es_3}. This gives the statements  (i) and (ii).

(iii) 
By lemma \ref{lyapunov}, the Lyapunov exponents of $\delta_O$ are all positive if $\sigma^2<2\alpha$, which implies that $O$ is unstable. Moreover,
by Proposition \ref{pro1} and Theorem \ref{thmes}, system \eqref{sys5} always has three random equilibria $u_g(\omega)\textbf{e}_i$, $i=1,2,3$ when $\sigma^2<2\alpha$. By lemma \ref{lyapunov} again, the sign of the Lyapunov exponents of $u_g(\omega)\textbf{e}_i, i=1,2,3$ depend on the sign of $\alpha+d_i, i=1,2,3.$

(iii.1) When $\prod_{i=1}^3(\alpha+d_i)= 0$, in view of Proposition \ref{pro1} (iii)-(v) $\mathcal{E}$ consists of infinitely many equilibria. Then, by Theorem \ref{thmes}, $\{\mu_Q: Q\in \mathcal{E}\setminus\{O\}\}$ consists of infinitely many ergodic stationary measures, each of which is supported on the ray $\overline{\mathcal{L}(Q)}$.

(iii.2) In view of Proposition \ref{pro1} (i), $\mathcal{E}=\{O, Q^*,\textbf{e}_i, i=1,2,3\}$. Then by Theorem \ref{thmes}, for any $Q\in \mathcal{E}\setminus\{O\}$, $\mu_Q$ is an ergodic stationary measure supported on the ray $\overline{\mathcal{L}(Q)}$.

Moreover, by Theorem \ref{gd1} (i), for each $h\in ({h}^*, \infty)$, there exists a closed orbit $\Gamma(h)$ and an invariant cone $\Lambda(h)$. 
Then, 
in view of Propositions \ref{existence} and \ref{uniqueness}, 
there exists a unique ergodic stationary measure $\nu_h$ on $\Lambda(h)\setminus\{O\}$. Thus, $\{\nu_h, \  h\in ({h}^*, \infty)\}$ consists of infinitely many ergodic stationary measures supported on invariant cones $\Lambda(h)$.

(iii.3) By Proposition \ref{pro1} (ii),  $\mathcal{E}:=\{O, \textbf{e}_i, i=1,2,3\}$. Again
by Theorem \ref{thmes}, $\delta_O, \mu_{\textbf{e}_i}, i=1,2,3$, are ergodic stationary measures supported on $O$ or rays $\overline{\mathcal{L}(\textbf{e}_i)}, i=1,2,3$. 

It remains to prove that $\mathcal{P}_e(\mathbb{R}_+^3)=\{\delta_O, \mu_{\textbf{e}_i}, i=1,2,3\}$. For this purpose, we only need to prove that for any stationary measure $\nu \in \mathcal{P}(\mathbb{R}^3_+)$,
\begin{equation}\label{four-ergodic-measure-1}
\nu(\cdot)=\sum_{Q\in \mathcal{E}}\nu(\mathcal{A}(Q))\mu_Q(\cdot).
\end{equation}
To this end, by the definition of stationary measures and Corollary \ref{decomposition}, 
\begin{equation}\label{four-ergodic-measure}
\begin{split}
\nu(\cdot) 
=\int_{\mathbb{R}^3_+}P(t, x, \cdot)\nu(dx) 
 =\sum_{Q\in \mathcal{E}}\int_{\mathcal{A}(Q)}P(t, x, \cdot)\nu(dx)
\end{split}
\end{equation}
Note that for any $x\in \mathcal{A}(Q)$, one has $P(t, x, \cdot)\stackrel{w}{\rightharpoonup} \mu_Q$ as $t\to \infty$. Thus, letting $t$ go to infinity in \eqref{four-ergodic-measure} we obtain \eqref{four-ergodic-measure-1} and finish the proof.  \hfill $\square$
\medskip

Combining Theorem \ref{bifurcations} and Lemma \ref{lyapunov} we have the following Corollary about hyperbolicity of finite many ergodic stationary measures.

\begin{corollary}\label{hypebolic ergodic measure} When system \eqref{sys5} has only finite many ergodic stationary measures, these measures are all hyperbolic except the case where $\sigma^2=2\alpha$.
\end{corollary}

Furthermore, we also have the bifurcation for the density functions of ergodic stationary measures, which is stated in Theorem \ref{3dp}.

For this purpose,
let us first derive the density function of the
ergodic stationary measure $\mu_Q$
related to equilibria.

\begin{lemma}\label{density fun}
Let $\sigma^2<2\alpha$ and $Q=(q_1, q_2, q_3)\in \mathcal{E}\backslash \{O\}$ be any equilibrium of \eqref{sys6}.
Let $j\in\{1, 2, 3\}$ be such that $q_j\neq0$.
Then, the density function of $ \mu_Q$ has the expression
\begin{equation}\label{den3}
 f_Q(x)=
 \begin{cases}
  \frac1{\|Q\|}p_{\alpha}^{\sigma}(\frac{x_j}{q_j}),  \,\, x\in \mathcal{L}(Q), \\
  0, \,\, x\notin \mathcal{L}(Q),
 \end{cases}
\end{equation}
where $p_{\alpha}^{\sigma}$ is given by \eqref{den}.
\end{lemma}

\begin{proof}
Let $F_Q$ denote the distribution function of
$Q_g(\omega) = u_g(\omega) Q$, $\omega\in \Omega$.
Then for any $x\in \mathbb{R}_+^3$, we have
\begin{align*}
F_Q(x) =\mathbb{P} (u_g q_i\leq x_i,\; i=1, 2, 3)
 = \int_0^{min\{\frac{x_i}{q_i},\; q_i\neq 0  \}}p_{\alpha}^{\sigma}(s)ds,
\end{align*}
where $p_\alpha^\sigma$ is the density of $\mathbb{P}\circ u_g^{-1}$.
Since $\mu_Q(\mathcal{L}(Q))=1$, its density function is positive only on $\mathcal{L}(Q)$.
Then for any $x\in \mathcal{L}(Q)$,
letting $r={x_i}/{q_i}$, $q_i\neq 0$
we have $x=rQ$ and $\delta x=\delta r Q$,
thus
\begin{align*}
f_Q(x)
  =\displaystyle \lim_{\|\delta x\| \rightarrow 0} \frac{F_Q(x+\delta x)-F_Q(x)}{\|\delta x\| }
 = \frac 1{\|Q\|}\displaystyle \lim_{\delta r\rightarrow 0}\frac{\int_r^{r+\delta r} p_{\alpha}^{\sigma}(s)ds}{\delta r}
 =\frac 1{\|Q\|}p_{\alpha}^{\sigma}(r),
\end{align*}
which is exactly \eqref{den3}.
\end{proof}

Theorem \ref{3dp} shows that the stochastic Kolmogorov system \eqref{sys5} undergoes a stochastic P-bifurcation.

\begin{theorem}\label{3dp}(Stochastic bifurcation of density functions)
Let $\sigma^2<2\alpha$ and $Q \in \mathcal{E}\backslash \{O\}$ be a equilibrium of system \eqref{sys6}.
Let $f_Q$ be the density function corresponding to the ergodic stationary measure $\mu_Q$. Then, $\{f_Q\}$ undergoes a P-bifurcation at $\sigma^2=\alpha$.
\end{theorem}

\begin{proof}
Without loss of generality, we may assume that $Q=(q_1, q_2, q_3)$ with $q_1\neq0$.
Then, by Lemma \ref{density fun},
\begin{equation}
 f_Q(x)=
 \begin{cases}
 \frac 1{\|Q\|}p^{\sigma}_{\alpha}(x_1/q_1), \; x\in \mathcal{L}(Q), \\
 0, \; x\notin \mathcal{L}(Q).
 \end{cases}
 \end{equation}

Moreover, by \eqref{den}, $ p^{\sigma}_{\alpha}(s)=C_\alpha s^{2\frac{\alpha}{\sigma^2}-2}\exp(-\frac{\alpha}{\sigma^2}s^2)$ for $s\in \mathcal{L}(P)$. Since $p^{\sigma}_{\alpha}(s)$ is decreasing in $s$ with pole at $s=0$ for $\alpha< \sigma^2<2\alpha$, and it is a unimodal function with $p^{\sigma}_{\alpha}(0)=0$
for $\sigma^2<\alpha$,
we derive that the density functions $\{f_Q\}$
 admit a $P$-bifurcation at $\sigma^2=\alpha$.
\end{proof}

\subsection{Classification of pull-back $\Omega$-limit sets}
In this subsection, we first give the proof of Theorem \ref{main-thm-2} 
and then state the complete classification
of pull-back $\Omega$-limit sets for system \eqref{sys5} on $\mathbb{R}_+^3$
in Theorem \ref{class1} below.

Recall that
$\mathcal{L}(y):=\{\lambda y: \lambda>0\}$ denotes the ray passing through the point $y$,
where $y\in \mathbb{R}_+^3$.
We still use the notations ${h}^*$, $Q_x$, $\Gamma_{12}$, $\Gamma_{12}^+$, $\Gamma_{13}$, $\mathcal{L}(Q^*)$ and $\Gamma(h)$
as in Theorem \ref{gd1}.
Let $u_g$ denote the random equilibrium of equation \eqref{sysg},
$\Omega_x(\omega)$ the $\Omega$-limit set of the trajectory
$\{\Phi(t, \theta_{-t}\omega, x)\}$.

\medskip
{\bf Proof of Theorem \ref{main-thm-2}}
(i) By Theorem \ref{attract} (i), for any $x\in \mathbb{R}_+^3$, $\Omega_x(\omega)=\{O\}$. Thus it remains to prove that $O$ is the unique random equilibrium.

For this purpose, assume that there exists another $\mathcal{F}_-$-random equilibrium $V$ such that $V\neq O$ almost surely. Then since $V$ is $\mathcal{F}_-$-measurable, the distribution of $V$, denoted by $\nu$, is a stationary measure satisfying $\nu\neq \delta_O$. But this contradicts the uniqueness of $\delta_O$ in  Theorem 1.1 (i) and (ii).

(ii.1) In view of Corollary \ref{decomposition} (II) and Theorem \ref{attract} (ii), for any $x\in \mathbb{R}^3_+$, there exists $Q\in \mathcal{E}$ such that $\Omega_x(\omega)=\{u_g(\omega)Q\}$. Moreover, by Proposition \ref{pro1} and Theorem \ref{thmes}, $\{u_g(\omega)Q: \  Q\in \mathcal{E}\}$ are all random equilibria.

(ii.1$_a$) Without loss of generality, let us assume that $\alpha+d_1=0$. Then by Proposition \ref{pro1} (iii), $\mathcal{E}=\{O, \textbf{e}_3, Q: Q\in \Gamma_{12}\}$. 
Then the statement follows from the fact that, 
for each noise realization, $\{u_g(\omega)Q: Q\in \Gamma_{12}\}$ form a curve on plane $\{x\in \mathbb{R}^3_+: x_3=0\}$.

The statements in 
(ii.1$_b$)-(ii.1$_c$) can be proved by using similar arguments as in the case (ii.1$_a$).

(ii.2) By Proposition \ref{pro1} (i),  $\mathcal{E}:=\{O, Q^*, \textbf{e}_i, i=1,2,3\}$. Again
by Theorem \ref{thmes}, system \eqref{sys5} has 5 random equilibria $O, u_g(\omega)Q^*, u_g(\omega)\textbf{e}_i, i=1,2,3$.
Now, let us prove the existence of infinitely many Crauel random periodic solutions. 

To this end, first by Theorem \ref{gd1} (i), for each $h\in (h^*, \infty)$, $\Gamma(h)$ is a periodic orbit with a minimum positive period defined by $N(h)$. Then for  $y_0\in \Gamma(h)$ fixed, $\Psi(t, y_0)$ is a periodic solution with period $N(h)$ of system \eqref{sys6}. Let us define the mapping $\psi_h: \Omega \times \mathbb{R}_+\to \mathbb{R}^3_+$ by
\begin{equation}\label{periodic-solution}
\psi_h(t, \omega):=u_g(\omega)\Psi(\int_{-t}^0u_g^2(\theta_s\omega)ds, y_0).
\end{equation}
For any $t_0\in \mathbb{R}_+$, by the stochastic decomposition formula \eqref{sdf} and a change of variables,
\begin{align}\label{check-solution}
\Phi(t, \omega, \psi_h(t_0, \omega))&=g(t, \omega, u_g(\omega))\Psi(\int_0^tu_g^2(\theta_s\omega)ds, \Psi(\int_{-t_0}^0u_g^2(\theta_s\omega)ds, y_0))  \notag \\ 
&=u_g(\theta_t\omega)\Psi(\int_{-t_0}^tu_g^2(\theta_s\omega)ds, y_0)) \notag \\
&=u_g(\theta_t\omega)\Psi(\int_{-(t+t_0)}^0u_g^2(\theta_{s+t}\omega)ds, y_0) \notag \\
&=\psi_h(t+t_0, \theta_t\omega).
\end{align}

Now, define $T: \Omega\to \mathbb{R}$ by
\begin{equation}\label{random-periodic-time}
T_h(\omega):=\inf\{t>0: |\int_{-t}^0u_g^2(\theta_s\omega)ds|=N(h)\}.
\end{equation}
We first show that $0<T_h(\omega)<+\infty$ almost surely. Actually, this follows from
$$\displaystyle \lim_{t\rightarrow \infty}\frac1t\int_{-t}^0 u_g^2(\theta_s\omega)ds=\mathbb{E}u_g^2>0,$$
due to the Birkhoff-Khintchin ergodic theorem.
Then, by a change of variables, we have
$$\int_{-(t+T_h(\theta_{-t}\omega))}^{-t}u_g^2(\theta_s\omega)ds=\int_{-T_h(\theta_{-t}\omega)}^0u_g^2(\theta_{s-t}\omega)ds=N(h),$$
which yields that
\begin{equation}\label{check-periodic}
\begin{split}
\psi_h(t+T_h(\theta_{-t}\omega), \omega)&=u_g(\omega)\Psi(\int_{-(t+T_h(\theta_{-t}\omega))}^0u_g^2(\theta_s\omega)ds, y_0)\\
&=u_g(\omega)\Psi(\int_{-t}^0u_g^2(\theta_s\omega)ds+N(h), y_0)\\
&=\psi_h(t, \omega),
\end{split}
\end{equation}
where the last step was due to the fact that $\Psi(t, y_0)$ is a periodic solution with positive period $N(h)$.

Hence, combining \eqref{check-solution} with \eqref{check-periodic} we derive that for each $h\in (h^*, \infty)$ and fixed $y_0\in \Gamma(h)$, the pair $(\psi_h, T_h)$ defined by \eqref{periodic-solution} and \eqref{random-periodic-time} is a Crauel random periodic solution, and so $\{(\psi_h, T_h): h\in (h^*, \infty)\}$ are infinitely many Crauel random periodic solutions.

Finally, it follows from  Corollary \ref{decomposition} (I) and Theorem \ref{attract} (ii) that for any $x\in \mathbb{R}^3_+$, $\Omega_x(\omega)$ is either $\{u_g(\omega)Q\}$ for some $Q\in \mathcal{E}$ or $\{u_g(\omega)\Gamma(h)\}$ for some $h\in (h^*, \infty)$. Thus, the statements are proved.

(ii.3) By Proposition \ref{pro1} (ii), Theorems \ref{thmes} and \ref{attract} (ii), we derive that, for any $x\in \mathbb{R}^3_+$, $\Omega_x(\omega)=\{u_g(\omega)Q\}$ where $Q\in \{O, \textbf{e}_i, i=1,2,3\}$. Moreover, for any $\mathcal{F}_-$-measurable random equilibrium $V$, $\nu:=\mathbb{P}\circ V^{-1}$ is a stationary measure. Then, applying Theorem \ref{bifurcations} (iii.3) we infer that $\nu$ is a convex combination of $\{\delta_O, \mu_{\textbf{e}_i}, i=1,2,3\}$.   The proof is complete.     \hfill $\square$
\medskip

In the end of this section, 
we give a more detailed classification of the pull-back $\Omega$-limit sets of stochastic system \eqref{sys5} 
corresponding to different locations of initial data, 
The proof follows from Theorems \ref{gd1}  
and \ref{attract}. 

\begin{theorem}\label{class1} (Classification of pull-back $\Omega$-limit sets)
For almost surely $\omega\in \Omega$,
the following holds:
\begin{itemize}
\item [(i)] If $\sigma^2<2\alpha$ and $\alpha+d_i>0$, $i=1, 2, 3$,
    then there are 5 random equilibria:
    $$\{O, u_g\textbf{e}_1, u_g\textbf{e}_2, u_g\textbf{e}_3, u_gQ^*\}.$$

   Further,
       $\Omega_x(\omega)=\{u_g(\omega)\Gamma(h)\}$ if $x\in \Lambda_1(h)$ for any $h\in ({h}^*, \infty)$; $\Omega_x(\omega)=\{u_g(\omega)Q^*\}$ if $x\in \mathcal{L}(Q^*)$; $\Omega_x(\omega)\in \{u_g(\omega)\textbf{e}_i, i=1,2,3\}$ if $x\in \partial\mathbb{R}^3_+\setminus\{O\}$.

\item[(ii)] If $\sigma^2<2\alpha$ and $\alpha+d_1<0, \alpha+d_2>0, \alpha+d_3<0$, then there are 4 random equilibria:
     $$\{O, u_g\textbf{e}_1, u_g\textbf{e}_2, u_g\textbf{e}_3\}.$$
     Moreover, $\Omega_x(\omega)=\{u_g(\omega)\textbf{e}_1\}$  if $x\in \textrm{Int}\mathbb{R}^3_+$; $\Omega_x(\omega)\in\{u_g(\omega)\textbf{e}_i, i=1,2,3\}$ for any $x\in \partial\mathbb{R}^3_+\setminus\{O\}$.

\item[(iii.a)] If $\sigma^2<2\alpha$ and $\alpha+d_1=0, \alpha+d_2>0, \alpha+d_3>0$,
 then there are infinitely many random equilibria:
 $$\{O, u_g\textbf{e}_3\}\bigcup \{u_gQ: Q\in \Gamma_{12}\}.$$
 Moreover, $\Omega_x(\omega)\in \{u_g(\omega)Q: Q\in \Gamma_{12}^+\}$ for any
  $x\in \textrm{Int}\mathbb{R}^3_+$; $\Omega_x(\omega)\in\{u_g(\omega)\textbf{e}_3, u_g(\omega)Q: Q\in \Gamma_{12}\}$ if $x\in \partial\mathbb{R}^3_+\setminus \{O\}$.

\item[(iii.b)] If
   $\sigma^2<2\alpha$ and $\alpha+d_1>0,\alpha+d_2=0,\alpha+d_3<0$, then there are infinitely many random equilibria:
   $$\{O, u_g\textbf{e}_2\}\bigcup \{u_gQ: Q\in \Gamma_{13}\}.$$
   Moreover, $\Omega_x(\omega)=\{u_g(\omega)\textbf{e}_2\}$ for any $x\in \textrm{Int}\mathbb{R}^3_+$; $\Omega_x(\omega)\in \{u_g(\omega)\textbf{e}_2, u_g(\omega)Q: Q\in \Gamma_{13}\}$ if $x\in \partial\mathbb{R}^3_+\setminus\{O\}$.
\item[(iv)] If
    $\sigma^2<2\alpha$ and $\alpha+d_1=0, \alpha+d_2=0, \alpha+d_3<0$, then there are infinitely many random equilibria: $$\{O\}\bigcup \{u_gQ: Q\in \Gamma_{12}\,\, \textrm{or} \,\, \Gamma_{13}\}.$$
    Moreover,  $\Omega_x(\omega)\in\{u_g(\omega)Q: Q\in \Gamma_{12}\}$ for any $x\in \textrm{Int}\mathbb{R}^3_+$; $\Omega_x(\omega)\in \{u_g(\omega)Q: Q\in \Gamma_{12}\bigcup\Gamma_{13}\}$ if $x\in \partial\mathbb{R}^3_+\setminus\{O\}$
\item[(v)] If
    $\sigma^2<2\alpha$ and $\alpha+d_i=0$ for all $i\in\{1,2,3\}$,
    then there are infinitely many random equilibria:
    $$\{O\}\bigcup \{u_gQ: Q\in \mathbb{S}^2_+\}.$$
     Moreover, $\Omega_x(\omega)=\{u_g(\omega)Q_x\}$  for any $x\in \mathbb{R}_+^3\setminus \{O\}$, where $Q_x:=\mathcal{L}(x)\bigcap  \mathbb{S}^2_+$.

\end{itemize}
\end{theorem}

\clearpage
\appendix
\section{Appedix}
In this section, we collect some essential definitions and results on random dynamical systems in this paper.
\subsection{Preliminaries of Random dynamical system}\label{P_RDS}
 Let $X:=\mathbb{R}_+^3$ and $(W_t)_{t\in \mathbb{R}}$ be a two-side Brownian motion in $\mathbb{R}$ and
let $\Omega=\{\omega\in \mathcal{C}(\mathbb{R}, \mathbb{R}): \omega(0)=0\}$,
$\mathcal{F}$ the Borel $\sigma$-algebra of $\Omega$, $\mathbb{P}$ the measure induced by $W$
(i.e., Wiener measure).
It is known that the shift
$$\theta_t: \Omega\rightarrow \Omega,\,\, \theta_t\omega(s):=\omega(t+s)-\omega(t),\ \  s, t\in \mathbb{R},$$
is measure-preserving and ergodic with respect to $\mathbb{P}$,
and $\omega(t):=W_t(\omega)$ is a Brownian motion under $\mathbb{P}$.
Thus $(\Omega, \mathcal{F}, \mathbb{P}, (\theta_t)_{t\in \mathbb{R}})$ is a ergodic metric dynamical system. Let us define
$$\mathcal{F}_-=\sigma(\omega(t): t\le0),
\ \ \mathcal{F}_+=\sigma(\omega(t): t\ge0).$$
It is clear that $\mathcal{F}_-$ and $\mathcal{F}_+$ are independent.

A {\it $C^1$ random dynamical system with independent increments} on phase space $X$ over the metric dynamical system $(\Omega, \mathcal{F}, \mathbb{P}, (\theta_t)_{ t\in \mathbb{R}})$ is a measurable mapping
$$\Phi : \mathbb{R}_+\times \Omega \times X\mapsto X, \,\, (t, \omega, x)\mapsto \Phi(t, \omega, x),$$
such that
\begin{enumerate}
  \item[(i)] the mapping $(t,x)\mapsto \Phi(t, \omega, x)$ is continuous for all $\omega\in \Omega$, and the mapping $x\mapsto \Phi(t, \omega, x)$ is $C^1$ for all $t\ge 0$ and $\omega\in \Omega$,
  \item[(ii)] the mappings $\Phi(t, \omega):=\Phi(t, \omega, \cdot)$ satisfy the {\it cocycle} property:
  \begin{align}   \label{Phi-cocycle}
     \Phi(0, \omega)=id,\,\, \Phi(t+s, \omega)=\Phi(t, \theta_s\omega)\circ\Phi(s, \omega)
  \end{align}
for all $t, s\in \mathbb{R}_+$ and $\omega\in \Omega$,
\item[(iii)] if for all $s, t>0$, we have $\Phi(t, \omega)$ is independent of $\Phi(s, \theta_t\omega)$.
\end{enumerate}

For simplicity, we say that $(\theta, \Phi)$ or $\Phi$ is an RDS.

The {\it $\Omega$-limit set} $\Omega_x(\omega)$ of the pull-back trajectory $\Phi(t, \theta_{-t}\omega, x)$ is defined by
\begin{equation*}
\Omega_x(\omega):=\bigcap_{t>0}\overline{\bigcup_{\tau\geq t}\Phi(\tau, \theta_{-\tau}\omega, x)}.
\end{equation*}

\begin{definition} (Random equilibrium, \cite[Definition 1.7.1, p.38]{chu}).
A $\mathcal{F}$-measurable random variable $u: \Omega\mapsto X$ is said to be an {\it equilibrium}
(or, stationary solution) of the RDS $(\theta, \Phi)$ if it is invariant under $\Phi$:
$$\Phi(t, \omega, u(\omega))=u(\theta_t\omega), \ \ a.s.\ \omega\in \Omega,\ \  \forall \ t\geq 0.$$
\end{definition}

Note that the equilibrium is $\mathcal{F}_-$- measurable
if the RDS is generated by the solutions to stochastic differential equations \eqref{sys5}.

\begin{definition} \label{Crauel-def} (Crauel Random periodic solution, \cite[Definition 6]{ELR21}).
A Crauel random periodic solution (CRPS) is a pair $(\psi, T)$ consisting of $\mathcal{F}$-measurable functions $\psi: \Omega\times \mathbb{R}_+\to X$ and $T: \Omega\to \mathbb{R}$ such that for almost all $\omega\in \Omega$,
$$\psi(t, \omega)=\psi(t+T(\theta_{-t}\omega),\omega) \ \ \textrm{and}\ \ \Phi(t, \omega, \psi(t_0, \omega))=\psi(t+t_0, \theta_t\omega), \ \forall \ t, t_0\in \mathbb{R}_+.$$
\end{definition}

\begin{definition}\label{random cycle} (Attracting Random Cycle, \cite[Definition 4]{ELR21}).
We shall say that a random pull-back attractor $A$ with respect to a collection of sets $\mathcal{S}$ is an attracting random cycle if for almost all $\omega\in \Omega$ we have $A(\omega)\cong S^1$, i.e., every fiber is homeomorphic to the circle.
\end{definition}

For more details about random attractors see \cite{ELR21}.

Recall that the {\it derivative cocycle} of $C^1$ RDS $\Phi$ is the jacobian
$$D\Phi(t, \omega, x)=\frac{\partial \Phi(t, \omega)x}{\partial x}:=(\frac{\partial(\Phi(t, \omega) y)_i}{\partial y_j})|_{y=x}. $$

The {\it Lyapunov exponent} at $x\in \mathbb{R}^n$ in the direction $v\in \mathbb{R}^n$ is the following limit (if the limit exists)
$$\displaystyle \lim_{n\rightarrow \infty}\frac 1t \log \|D\Phi(t, \omega, x)v\|.$$

\subsection{Preliminaries of Markov semigroup 
and ergodicity}
A Markov transition function associated to the $C^1$ RDS $\Phi$ with independent increments is defined by
$$ P(t, x, A):=\mathbb{P}(\Phi(t, \cdot, x)\in A), \ x\in X, \ A\in\mathcal{B}(X),$$
which generates a Markov semigroup $(P_t)_{t\ge 0}$ by
$$P_t: \mathcal{B}_b(X)\to \mathcal{B}_b(X), \ P_tf(x):=\int_{X}f(z)P(t, x,dz), \ x\in X.$$
It is clear that this Markov semigroup $P_t$ is {\it stochastically continuous}:
$$\displaystyle \lim_{t\rightarrow 0} P_tf(y)=f(y), \ \ \forall \,\, f\in \mathcal{C}_b(\mathbb{R}^3)\,\, \textrm{and} \,\, y\in \mathbb{R}^3,$$
and {\it Feller}, that is, for any $f\in C_b(X)$ and $t\ge0$, one has $P_tf\in C_b(X)$.

This Markovian semigroup $P_t$
is called a {\it strong Feller} semigroup at time $t_0>0$ if
$P_{t_0}f\in \mathcal{C}_b(X)$ for any $f\in \mathcal{B}_b(X)$.
It is called {\it irreducible} at time $t_0$
if $P(t_0, x, A)>0$
for any non-empty open set $A$ and $x\in A$.

A probability measure $\mu$ on $X$ is called
{\it stationary} with respect to $P_t$,
if
\begin{equation*}
\int_{X} P(t, y,A)\mu(dy)=\mu(A),\ \  \forall\,\, t\geq0,\ A\in \mathcal{B}(X).
\end{equation*}
 Moreover, a stationary measure $\mu$ is {\it ergodic} if the
$(P_t, \mu)$-invariant functions are constants $\mu$-a.s.
A measure $\mu$ is said to be stationary for RDS $\Phi$ if it is stationary for the corresponding Markov semigroup $P_t$.

Let $X$ be a separable and locally compact Hausdorff space. The ergodicity can be derived from the following strongly mixing property.

\begin{theorem}\label{prato} (Strongly mixing, \cite[Theorem 3.4.2, Corollary 3.4.3]{prato})
Let $P_t, t\geq 0$, be a stochastically continuous Markovian semigroup on $X$
and $\mu$ a corresponding stationary measure.
Then, the following statements are equivalent:
\begin{itemize}
  \item [(i)] $\mu$ is strongly mixing;
  \item [(ii)] for any $\varphi\in L^2(X, \mu)$,
      \begin{equation*}\label{aaaa}
      \displaystyle \lim_{t \rightarrow\infty} P_t\varphi=\langle \varphi, 1\rangle \ \  \textrm{in} \  L^2(X, \mu).
      \end{equation*}
\end{itemize}
Moreover, if the corresponding transition probability measure satisfies
$$\displaystyle \lim_{t \rightarrow\infty}P(t, x, \cdot)=\mu\ \
\textrm{weakly} \,\,\textrm{in} \,\, \mathcal{P}(X),\ \forall \,\, x\in X,$$
then $\mu$ is strongly mixing. In particular, $\mu$ is ergodic.
\end{theorem}

It is known that strong Feller
and irreducibility imply
uniqueness of stationary measures
(hence ergodicity)
and strongly mixing.
See, e.g., \cite[Theorem 4.2.1, p.43]{prato}.
Moreover, strong Feller
and irreducibility are equivalent
under conjugate mappings,
see \cite[Theorem 2.6]{jiang1}.

\subsection*{Acknowledgments}
The authors are grateful to Vahagn Nersesyan, Xiaodong Wang, Shennan Yin and Huaizhong Zhao for helpful discussions.
D. Xiao is partially supported by National Key R $\&$ D Program of China (No. 2022YFA1005900), the Innovation Program of
Shanghai Municipal Education Commission (No.2021-01-07-00-02-E00087) and the NSFC grants (No. 12271353 $\&$ 11931016).
D. Zhang is partially supported by NSFC (No.12271352, 12322108)
and Shanghai Rising-Star Program 21QA1404500.
\end{sloppypar}

\bibliographystyle{abbrv}
\bibliography{ref}

\end{document}